\documentclass{amsart}

\usepackage{amsmath,amssymb,amsthm,hyperref,verbatim,url}
\usepackage[margin=1in]{geometry}

\numberwithin{equation}{section}

\theoremstyle{definition}
\newtheorem{definition}{Definition}[section]
\newtheorem{problem}{Problem}
\newtheorem*{acknowledgments}{Acknowledgments}

\theoremstyle{plain}
\newtheorem{theorem}{Theorem}[section]
\newtheorem{corollary}[theorem]{Corollary}
\newtheorem{lemma}[theorem]{Lemma}
\newtheorem{proposition}[theorem]{Proposition}

\theoremstyle{remark}
\newtheorem{remark}[theorem]{Remark}

\title[Transversely Isotropic Elastic Travel Time Tomography]{Partial Global Recovery in the Elastic Travel Time Tomography Problem for Transversely Isotropic Media}
\author{Yuzhou Zou}
\address{Department of Mathematics\\
Stanford University\\
Stanford, CA 94305-2125, U.S.A.}
\email{zou91@stanford.edu}
\date{Originally uploaded on October 2, 2019. Last updated: \today}
\begin{document}
\begin{abstract}
We consider the problem of recovering material parameters in a transversely isotropic medium from the $qP$ and $qSV$ waves' travel times, given the axis of isotropy and the material parameters associated to the $qSH$ wave speed. The operators obtained from the pseudolinearization argument are of parabolic type, and so we discuss inverting operators whose symbols are of parabolic type. We present stability estimates for recovering either one parameter from one wave speed or two parameters from two wave speeds with the remaining parameters either known or with a known functional relationship, and these estimates provide injectivity among parameters that differ on sets of small width.
\end{abstract}
\maketitle
\sloppy

\tableofcontents

\section{Introduction}
\label{intro-sec}

\subsection{Background on transversely isotropic elasticity}
In this paper we consider the travel time tomography problem for transversely isotropic elastic media. The context is the (linear) elastic wave equation $u_{tt} - Eu = 0$ in $\mathbb{R}^3$ describing the evolution of an elastic material over time. Here, $u$ is a vector-valued function of time and space describing the displacement of an elastic material from a rest frame, and $E$ is a second-order differential operator mapping vector-valued functions to vector-valued functions (and hence can be thought of as a matrix of second-order differential operators). Explicitly,
\[(Eu)_i = \rho^{-1}\sum_{jkl}{\partial_j(c_{ijkl}\partial_lu_k)} = \sum_{jkl}{\frac{c_{ijkl}}{\rho}\partial_j\partial_lu_k} + \text{lower order terms}\]
where $\rho(x)>0$ is the density of the material, and $c_{ijkl}(x)$ are the components of the ``elasticity tensor'' which affect the evolution of the equation; these components in turn depend on the physical properties of the material and in general may vary over space. The goal is thus to recover these components from some set of observations regarding the evolution of this equation.

Associated to the elastic wave equations are a set of wave speeds. The wave speeds can be described as follows: the principal symbol of the operator $u \mapsto u_{tt}-Eu$ is given by $-\tau^2\text{Id} + \sigma(-E)(x,\xi)$, where 
\[\sigma(-E)(x,\xi) = \left(\sum_{jl}{\frac{c_{ijkl}(x)}{\rho(x)}\xi_j\xi_l}\right)_{ik}.\]
The matrix $\sigma(-E)(x,\xi)$ is always symmetric and positive definite for all $x$ and all $\xi\ne 0$, and so for those $(x,\xi)$ we have that $\sigma(-E)(x,\xi)$ has three positive eigenvalues (possibly with multiplicity) which depend on $x$ and $\xi$; denote these eigenvalues by $G_j(x,\xi)$. We have that
\[WF(u)\subset\{(t,\tau,x,\xi)\,:\,-\tau^2\text{Id} + \sigma(-E)(x,\xi)\text{ is not invertible}\},\]
for all solutions $u$ of $u_{tt} - Eu = 0$, and the latter set consists precisely of the points where $\tau^2$ is an eigenvalue of $\sigma(-E)(x,\xi)$, i.e. $\tau^2 = G_j(x,\xi)$ for some $j$. If we \emph{assume that the multiplicity of the eigenvalues is constant} among all $(x,\xi)$, so that $\{\tau^2 = G_j(x,\xi)\}$ are disjoint, then a classical propagation of singularities result \cite{taylor} states that the singularities of $u$, which are contained in $\{\tau^2 = G_j(x,\xi)\}$ for some $j$, will then be invariant under the Hamilton flow of $\tau^2 - G_j(x,\xi)$ for that $j$. Note that if $G$ is a positive definite quadratic form in $\xi$, i.e. the dual metric function of some metric $g$, then the Hamilton flow of $\tau^2-G$ restricted to $\{\tau = 1/2\}$ is exactly the geodesic flow with respect to $g$, and the singularities would propagate in the same manner as the singularities for the scalar wave equation $u_{tt} - \Delta_gu$. Thus, the Hamiltonian dynamics with respect to the Hamiltonian $\tau^2-G_j$ describe the dynamics of the so-called \emph{elastic waves}, with $G_j$ called the wave speeds\footnote{For full accuracy, we should call these quantities the ``squared wave speeds''; however for our purposes it is more convenient to work with these quantities, and hence we will refer to these quantities as the wave speeds; in any case data regarding the Hamiltonian dynamics of the wave speeds can be recovered from data regarding the Hamiltonian dynamics of its square, and vice versa.}; we will use knowledge regarding these dynamics to recover the elastic coefficients in $E$.

Since the elasticity tensor is a 4-tensor in 3-dimensional space, it \emph{a priori} has up to 81 independent components; however inherent symmetries of the elasticity tensor reduce the independence to at most 21 independent components in general. In the case of fully isotropic elasticity, this dependence is further reduced to just two independent components, and they are often described by the Lam\'e parameters $\lambda$ and $\mu$. In this case there is multiplicity for the wave speeds as well: the largest eigenvalue has multiplicity $1$ and is called the $p$ wave speed, while the other two eigenvalues coincide and is called the $s$ wave speed; these two wave speeds can be described explicitly in terms of $\lambda$ and $\mu$. We will instead study the case of \emph{transversely isotropic} elasticity, and we follow the notational conventions of \cite{ti}, which in turn borrows conventions from \cite{tsvankin}. In this case, there is an \emph{axis of isotropy} around which the material behaves isotropically. We will denote this axis as a covector field $\overline{\xi}(x)$ normalized under the dual metric function on $T^*\mathbb{R}^3$ associated to the Euclidean metric to have norm $1$. (In \cite{ti}, this axis was denoted by $\omega$; we will reserve $\omega$ for use as a spherical variable.) In addition, there are $5$ independent components of the elasticity tensor, which we denote by $a_{11}$, $a_{33}$, $a_{55}$, $a_{66}$, and $E^2$ (with\footnote{Note, despite the notation, that $E^2$ is not necessarily nonnegative. One can think of $E^2$ as a measure of deviation of the $qP$ or $qSV$ waves from having an ``ellipsoidal slowness surface,'' i.e. having the corresponding wave speeds being quadratic forms.} $E^2 = (a_{11}-a_{55})(a_{33}-a_{55}) - (a_{13}+a_{55})^2$ in the notation of \cite{tsvankin}); they will be referred to as the ``material parameters'' for the elastic material. Note that fully isotropic elasticity is a special case of transversely isotropic elasticity, with $a_{11} = a_{33} = \lambda+2\mu$, $a_{55} = a_{66} = \mu$, and $E^2 = 0$. Transversely isotropic elastic materials appear naturally in the Earth, where rocks are formed in layers over time; within each layer there is isotropic behavior, but the composition is not isotropic across different layers (see Section 1 of \cite{ti} for more examples and details).

The eigenvalues in transversely isotropic elasticity will not always have multiplicity for all $(x,\xi)$, i.e. it is possible for the three eigenvalues $G_1(x,\xi)$, $G_2(x,\xi)$, and $G_3(x,\xi)$ to be distinct at some $(x,\xi)$. However, two of the eigenvalues will tend to be similar, much like the $s$ wave speeds in isotropic elasticity, so these will be called the $qSH$ and $qSV$ wave speeds, while the remaining will be like the $p$ wave speed and will be called the $qP$ wave speed. We thus let $G_{qP}(x,\xi)$, $G_{qSH}(x,\xi)$, and $G_{qSV}(x,\xi)$ denote these eigenvalues. These functions can be explicitly described: assuming the background metric is the Euclidean metric, if we fix a point $x$ and make an orthogonal change of coordinates $(x_1,x_2,x_3)$ so that $\overline{\xi}(x)$ aligns with the $dx_3$ axis, and we write $(x,\xi)\in T^*\mathbb{R}^3$ in the canonical coordinates (i.e. $\xi = \sum_{i=1}^3{\xi_i\,dx_i}$), then
\begin{equation}
\label{qsh-formula}
G_{qSH}(x,\xi) = a_{66}|\xi'|^2+a_{55}\xi_3^2
\end{equation}
where $|\xi'|^2 = \xi_1^2+\xi_2^2$, and $G_{qP/qSV} = \frac{1}{2}G_{\pm}$, with
\begin{equation}
\label{qp-qsv-formula}
\begin{aligned}
G_{\pm}(x,\xi) &= (a_{11}+a_{55})|\xi'|^2+(a_{33}+a_{55})\xi_3^2 \\
&\pm\sqrt{((a_{11}-a_{55})|\xi'|^2+(a_{33}-a_{55})\xi_3^2)^2-4E^2|\xi'|^2\xi_3^2}
\end{aligned}
\end{equation}
where $+$ refers to the $qP$ wave speed and $-$ refers to the $qSV$ wave speed, and all material parameters are evaluated at $x$. More properties of the wave speeds, especially regarding their Hamiltonian dynamics, are explored in Section \ref{hamdyn}.

We note that the eigenvalues can coincide for some value of $(x,\xi)$ (for example, the $qSH$ and $qSV$ speeds always coincide when $\xi'=0$), but that nonetheless we may choose the eigenvalues to vary smoothly in $(x,\xi)$. In the rest of this paper, by ``wave speed data'' we will mean data regarding the Hamiltonian dynamics of the functions $G_{qSH}$, $G_{qP}$, and $G_{qSV}$, as defined in \eqref{qsh-formula} and \eqref{qp-qsv-formula}, and despite possible coincidences at some points we will treat the data regarding the dynamics of these three functions as separate pieces of data. In spirit this data should be obtainable by observing the behaviors of solutions to the elastic wave equation, via the propagation of singularities argument above, though such access may be more difficult in practice due to the coincidence of wave speeds; see Remark \ref{mult-rema} for more details.

\subsection{The travel time tomography problem and main results}
We thus phrase the question as follows: suppose $\Omega\subset\mathbb{R}^3$ is a bounded domain with smooth boundary, and assume the boundary is strictly convex with respect to either $qP$ or $qSV$ Hamiltonian dynamics. Suppose we know the \emph{lens relation} of the Hamiltonian flows of the wave speeds. That is, for any inwards-pointing covector $(x,\xi)\in\partial_-S^*\Omega$, we know the exiting covector of the Hamilton flow $(X(t),\Xi(t))$ starting at $(x,\xi)$, as well as the time of exit (i.e. we know $(t_0,X(t_0),\Xi(t_0))$ where $t_0 = \inf\{t>0\,:\,X(t)\not\in\overline\Omega\}$). Can we use this data to recover the material parameters which determine these Hamiltonian trajectories? (Note that if we only knew the travel times between boundary points, then this gives the lens data; see Lemma \ref{dtau} and Corollary \ref{exit} and the surrounding remarks.) By ``recovery'' we first focus on the injectivity problem. Thus, suppose we have two collections of parameters and isotropy axes $\{a_{11},a_{33},a_{55},a_{66},E^2,\overline{\xi}\}$ and $\{\tilde{a}_{11},\tilde{a}_{33},\tilde{a}_{55},\tilde{a}_{66},\tilde{E}^2,\tilde{\overline{\xi}}\}$, and let $r_{ii} = \tilde{a}_{ii} - a_{ii}$ (write $r_{E^2} = \tilde{E^2} - E^2$). We can phrase our problem as follows:
\begin{problem}
Suppose the two collections of parameters and isotropy axes give the same lens data for the Hamilton flows with respect to the $qP$ and $qSV$ Hamiltonian dynamics. Then is it true that $\tilde{\overline{\xi}} = \overline{\xi}$ and $r_{\nu} = 0$ ($\nu = 11, 33, 55, 66, E^2$)?
\end{problem}
Implicitly, we can think of the parameters without the tildes as a ``background'' or ``known'' collection of parameters, and the parameters with the tildes as a proposed collection of parameters we wish to compare against the background collection, given that the two collections produce the same travel time data.

Inverse problems regarding transversely isotropic elasticity have been studied in \cite{mr}, where the authors showed that the Dirichlet-to-Neumann map for the elastic wave equation determined the travel times for all wave speeds that satisfy the ``disjoint mode'' assumption. They also showed that for such wave speeds that are also quadratic forms (i.e. a dual metric corresponding to some Riemannian metric) in $\xi$ (e.g. the $qSH$ wave speed, or all three wave speeds if $E^2\equiv 0$) that two of the five parameters can be determined from the travel time data, using techniques from boundary rigidity. In \cite{ti}, the authors showed that the axis of isotropy $\overline{\xi}$ and the parameters $a_{55}$ and $a_{66}$ can be recovered from the $qSH$ wave speed (in part due to the $qSH$ wave speed being a quadratic form in $\xi$), assuming that the kernel of the axis of isotropy $\overline{\xi}$ is an integrable hyperplane distribution, i.e. $\overline{\xi}$ is a smooth multiple of some closed 1-form (locally representable as $df$ for some layer function $f$), as well as geometric conditions such as a ``convex foliation'' condition. Those results were also obtained using boundary rigidity results, specifically those developed in \cite{br2} and their predecessors (in particular this is where the ``convex foliation'' assumption comes in). We thus will assume that $a_{55}$, $a_{66}$, and the axis $\overline{\xi}$ are known, and hence focus on recovering $a_{11}$, $a_{33}$, and $E^2$ from the $qP$ and $qSV$ wave speeds. 

For convenience, we will also make the following assumptions:
\begin{itemize} 
\item We assume all parameters involved are smooth. Indeed, later we will construct operators based on the parameters which turn out to be pseudodifferential operators, and hence we will need the parameters to be smooth in order to use the smooth pseudodifferential theory.

\item We assume the differences between the parameters are compactly supported in $\Omega$; in general this can be done by extending the parameters to agree outside $\Omega$.

\item For the wave speeds $G = G_{qP}$ or $G_{qSV}$, we assume that $G$ is strictly convex in the fiber variable. (This is true if $G$ is a quadratic form, i.e. corresponds to a metric, and is always true for the $qP$ wave speed \cite{chapman}, though there are materials for which this does not hold for the qSV wave speed, such as for the Greenhorn shale; such cases are related to the phenomenon of ``wave triplication'' \cite{shale}.) As a consequence, we have that for every $x$ the map $\xi\mapsto\partial_{\xi}G(x,\xi)$ is invertible (if $G$ is a metric then this map is actually linear). For $\omega\in\mathbb{R}^3$, let $\xi(\omega;x)$ denote this inverse map (sometimes this will be written as $\xi(\omega)$ if the dependence on $x$ is not important). That is, let $\xi(\omega;x)$ satisfy
\begin{equation}
\frac{\partial G}{\partial\xi}(x,\xi(\omega;x)) = \omega.\label{xiomega}
\end{equation}
Write\footnote{$\xi_T$ standing for the ``transverse'' component and $\xi_I$ standing for the ``isotropic'' component.} 
\begin{equation}
\label{xit-eq}
\xi_T(\omega;x) := \xi(\omega;x)\cdot\overline{\xi}(x)
\end{equation}
and
\begin{equation}
\label{xii-eq}
\xi_I^2(\omega;x) := |\xi(\omega;x)|^2-\xi_T(\omega;x)^2.
\end{equation}
\item Given $(x_0,\xi_0)$, let $(X(t,x_0,\xi_0),\Xi(t,x_0,\xi_0))$ denote the Hamilton flow starting at $(x_0,\xi_0)$ (with respect to either the $qP$ or $qSV$ wave speeds, and with respect to the background parameters $\{a_{\nu}\}$). Consider the map $\mathbb{R}\times\mathbb{R}^3\times\mathbb{S}^2\ni(t,x,\omega)\mapsto X(t,x,\xi(\omega;x))$. We assume that
\begin{equation}
\label{noconjpts}
\begin{aligned}
&\text{for all }t\ne 0\text{ and all }x,\omega\text{, the derivative }\\
&\frac{\partial}{\partial (t,\omega)}(X(t,x,\xi(\omega)))\text{ has full rank.} 
\end{aligned}
\end{equation}
This is the analogue of the ``no conjugate points'' assumption often found in X-ray inverse problems.

\item We also assume that there are no trapped trajectories, that is, for all $(x,\omega)\in\mathbb{R}^3\times\mathbb{S}^2$ and any compact subset $K\subset\Omega$, the set
\[\{t\in\mathbb{R}\,:\,X(t,x,\xi(\omega))\in K\}\]
is compact.

\item Similarly, let $(\tilde{X},\tilde{\Xi})$ denote the Hamilton flow with respect to the second collection of parameters $\{\tilde{a}_{\nu}\}$. We will make the technical assumption that $\frac{\partial\tilde\Xi}{\partial\xi}$ is always invertible.

\item Finally, as in \cite{ti}, we assume that the kernel of the axis of isotropy $\overline{\xi}$ is an integrable hyperplane distribution. This is a natural local (though not global) geological assumption, as discussed in \cite{ti}.
\end{itemize}
Given these assumptions, we are ready to state our main results. We start with the problem of recovering one of the parameters $a_{11}$, $a_{33}$, or $E^2$, if the other two are known.
\begin{theorem}
\label{oneparam}
Suppose for $\nu = a_{11}$, $a_{33}$, or $E^2$ that the other parameters are known. Furthermore, suppose that \emph{a priori} the difference $r_{\nu}$ is known to be supported in a set of sufficiently small width. Then we can recover $a_{11}$ from the $qP$ travel time, or the $qSV$ travel time if $E^2$ is known to be nonzero, $a_{33}$ from the $qP$ travel time, and $E^2$ from either the $qP$ or $qSV$ travel times. (That is, knowledge of just the $qP$ travel times guarantees $r_{\nu}\equiv 0$ with the assumptions above, while knowledge of just the $qSV$ travel times guarantees $r_{\nu}\equiv 0$ for $\nu = E^2$ and for $\nu = a_{11}$ if $E^2$ is known to be nonzero.) In lieu of support assumptions on $r_{\nu}$, we still have stability estimates for $r_{\nu}$.
\end{theorem}
Note that a precise notion of width is given in Definition \ref{width}. The term ``stability estimates'' roughly refer to estimates of the form
\begin{equation}
\label{rough-stab-eq}
\|\nabla r_{\nu}\|_{L^2}\le C\|r_{\nu}\|_{H^{1/2}}
\end{equation}
which hold \emph{assuming} that the travel times with respect to $\{a_{\nu}\}$ and $\{\tilde{a}_{\nu}\}$ are the same. The term $\|r_{\nu}\|_{H^{1/2}}$ should morally be controlled by $\|\nabla r_{\nu}\|_{L^2}$, given the assumption of the compact support of $r_{\nu}$, and in fact Poincar\'e's inequality offers a way of controlling $\|u\|_{L^2}$ by $\|\nabla u\|_{L^2}$ for any $u\in C_c^{\infty}$ by a constant depending on the size of the support of $u$ (in particular going to zero as the width of the support goes to zero); controlling $\|u\|_{H^{1/2}}$ follows from Poincar\'e's inequality by a simple modification. Thus for $u = r_{\nu}$ with sufficiently small width of support we can absorb the $\|r_{\nu}\|_{H^{1/2}}$ term into the $\|\nabla r_{\nu}\|_{L^2}$ term. A more precise statement will be stated later in Section \ref{extended-outline-subsec} after the appropriate operators for the analysis of the errors $r_{\nu}$ have been introduced, and will be explained further in Section \ref{recovsec}. 

Note that the recovery of $a_{11}$ was already proven in \cite{ti} under the convex foliation condition; here we instead assume an \emph{a priori} small width on the support of $r_{\nu}$ but will otherwise argue globally instead of using the local artificial boundary argument. See Remark \ref{width-remark} regarding the practicability of the small width assumption, as well as Remark \ref{global-remark} regarding the choice of using the global argument instead of the local artificial boundary argument.

We next consider the problem of recovering two of the parameters, with the other parameter known. The results are of the same flavor as before, though in this case knowledge of the travel time data of both wave speeds must be combined to derive the result:
\begin{theorem}
\label{twoparam}
Suppose either $a_{11}$ or $a_{33}$ is known. From the knowledge of both $qP$ and $qSV$ travel times, we can recover $(a_{33},E^2)$ (resp. $(a_{11},E^2)$) if the differences $r_{33}$ and $r_{E^2}$ (resp. $r_{11}$ and $r_{E^2}$) are supported in a set of sufficiently small width. In lieu of support assumptions, we also have stability estimates for $(r_{33},r_{E^2})$ (resp. $(r_{11},r_{E^2})$).
\end{theorem}
At the end of Section \ref{recovsec}, we comment on the obstruction for proving the theorem for the problem of recovering $(a_{11},a_{33})$ from a known value of $E^2$.

Another way of recovering the coefficients is to assume a functional relationship among the coefficients, say with one coefficient represented as a function of the other two, so that the number of effective coefficients to solve for is reduced. A similar case of two coefficients depending on the third was explored in \cite{ti}, and in our case we have a result similar to the ones above:
\begin{theorem} 
\label{func}
Suppose there is a known functional relationship $a_{33} = f(a_{11},E^2)$ with $\frac{\partial f}{\partial a_{11}}\ge 0$, or $E^2 = f(a_{11},a_{33})$ with $\left|\frac{\partial f}{\partial a_{33}}\right|>0$, or $a_{11} = f(a_{33},E^2)$ with the derivatives $\frac{\partial f}{\partial a_{33}}$ and $\frac{\partial f}{\partial E^2}$ constant and $\frac{\partial f}{\partial a_{33}}\ne 0$, and if the $r_{\nu}$ have sufficiently small width of support, then we can recover $(a_{11},E^2)$ (resp. $(a_{11},a_{33})$ and $(a_{33},E^2)$) from the combined $qP$ and $qSV$ travel time data. In lieu of support assumptions, we also have stability estimates, as before.
\end{theorem}
An extended outline of the proofs of these three theorems is given in Section \ref{extended-outline-subsec}.

We now make a few remarks regarding the applicability of the above results:
\begin{remark}
\label{width-remark}
The main mathematical content of each of the theorems above is a stability estimate on $r_{\nu}$, roughly of the form \eqref{rough-stab-eq}, with the stability estimates upgrading to an injectivity statement (i.e. $r_{\nu}\equiv 0$) given an additional \emph{a priori} assumption on the support of $r_{\nu}$ via a modified form of Poincar\'e's inequality. Such \emph{a priori} assumptions are natural in time-lapse monitoring problems, where one wishes to keep track of elastic changes in a relatively small ``reservoir'' region, outside of which the elasticity can be assumed to remain constant. (Note that the transversely isotropic elasticity in the Earth does not technically satisfy our assumptions due to our simplified ``no conjugate points'' assumption above; however it turns out that the information we use in the inversion problem will only use trajectories whose velocity vectors are roughly orthogonal to the axis of isotropy; hence it may be possible to apply the above results near the boundary of the Earth, where the trajectories connecting nearby points do not have conjugate points.) Furthermore, \emph{a priori} assumptions regarding the \emph{width} of the support are natural in monitoring problems near fault lines, where changes in elasticity due to fault movement should be supported in a thin region near the fault (although in such cases the parameters are often discontinuous and thus require a modification of the argument presented here; see Remark \ref{global-remark} regarding a possible modification).
\end{remark}
\begin{remark}
In these results, there are no \emph{a priori} assumptions on the location of the support of $r_{\nu}$ (beyond having sufficiently small width), and that there is no ``diffeomorphism invariance'' ambiguity as is present in many boundary rigidity-related inverse problems. This is obscured by the fact that we have \emph{chosen to represent our ambient manifold as the Euclidean space $\mathbb{R}^3$ with the Euclidean metric}; note that any diffeomorphism fixing the boundary of a nonempty bounded open set and preserving the Euclidean metric must actually be the identity.
\end{remark}
\begin{remark}
\label{mult-rema}
We note that obtaining the travel time data in practice, say from observing the behavior of solutions to the elastic wave equation, may be difficult due to the non-constancy of the multiplicity of the eigenvalues of the elastic wave operator. Indeed, note that we can rewrite \eqref{qsh-formula} (without using a pointwise orthogonal change of coordinates to align the axis of isotropy with the $dx_3$ axis) as
\[G_{qSH} = a_{55}\xi_T^2 + a_{66}\xi_I^2,\]
where $\xi_T$ and $\xi_I$ are defined in \eqref{xit-eq} and \eqref{xii-eq}, and furthermore \eqref{qp-qsv-formula} can be rewritten as
\[G_{qSV} = a_{55}\xi_T^2 + C\xi_I^2\]
for some function $C$ which is smooth away from $\xi=0$; explicitly
\[C = a_{55} + \frac{2E^2\xi_T^2}{A+\sqrt{A^2-B}}\]
with
\[A = (a_{11}-a_{55})\xi_I^2+(a_{33}-a_{55})\xi_T^2, \quad B = 4E^2\xi_I^2\xi_T^2.\]
In particular, we have $G_{qSH}-G_{qSV} = O(\xi_I^2)$, and hence the $qSH$ and $qSV$ wave speeds will always coincide at $\xi$ satisfying $\xi_I = 0$, i.e. for $\xi$ parallel to the axis of isotropy. In fact, the $O(\xi_I^2)$ difference guarantees that the slowness surfaces, i.e. the level sets, of the $qSV$ and $qSH$ wave speeds will intersect tangentially at $\{\xi_I=0\}$. Moreover, for generic values of the material parameters (in particular away from the case of full isotropy), the two wave speeds will not coincide everywhere. It follows that the multiplicities of the eigenvalues need not be constant, and hence the elastic wave operator need not be of principal type, i.e. the standard propagation of singularities result need not apply. Physically, this corresponds to the elastic waves ``switching mode'' at points of non-constant multiplicity.

We nonetheless can take our eigenvalue functions to vary smoothly in $x$ and $\xi$, and hence we can still make sense of Hamiltonian trajectories associated to each wave speed. We emphasize that in this paper we will assume that we somehow have access to the travel time data of these trajectories (without concern for the behavior of the elastic waves in the interior) and aim to prove results assuming we somehow have access to this data. Practical methods of obtaining such data are not immediately clear due to the non-constant multiplicity and would require further investigation.
\end{remark}
The paper is organized as follows. In Section \ref{pseudolinear-sec}, we discuss the pseudolinearization procedure which turns the inverse problem of interest into a problem in microlocal analysis regarding certain matrix-valued pseudodifferential operators. The vast majority of the remainder of the paper is dedicated to setting up the analysis to studying the operators of interest. In Section \ref{comp-sec}, we compute the principal symbols of the operators of interest, to show that our operators are of ``parabolic type,'' in the sense that the principal symbols are scalar-valued, with the scalar quantity being non-negative but not elliptic, but with a subprincipal symbol which is purely imaginary and non-degenerate on the characteristic set of the principal symbol. In Section \ref{invparsec}, we discuss a symbol calculus, first studied by Boutet de Monvel in \cite{bdm}, developed in part to provide inverses to ``parabolic type'' operators such as our operators of interest. Finally, in Section \ref{recovsec}, we apply the theory discussed in Section \ref{invparsec} to prove Theorems \ref{oneparam}, \ref{twoparam}, and \ref{func}.

\begin{acknowledgments}
The author would like to thank Professors Maarten de Hoop, Gunther Uhlmann, and Andr\'as Vasy for their helpful comments in this work. The author would also like to thank the referee for useful comments in improving the work's background and exposition. The author gratefully acknowledges partial support from the National Science Foundation under grant number DMS-1664683.
\end{acknowledgments}

\section{The Pseudolinearization Argument}
\label{pseudolinear-sec}

\subsection{The pseudolinearization formula and associated operators}
\label{pseudolinear-subsec}
We will make use of the \emph{Stefanov-Uhlmann pseudolinearization formula} to convert our inverse problems of interest into problems in microlocal analysis, regarding the behavior of certain operators obtained from the pseudolinearization formula. This formula first appeared in \cite{su} and has been used to solve problems in boundary rigidity \cite{br1,br2}, which in turn has been used to solve the travel time tomography problem for fully isotropic elasticity \cite{isoe}. 

The formula says the following: given two vector fields $V$ and $\tilde{V}$ on some manifold, and given their corresponding flows $Z(t,z)$ and $\tilde{Z}(t,z)$, we have
\begin{equation}
\label{su-pseudolin-eq}
\tilde{Z}(t,z) - Z(t,z) = \int_0^t{\frac{\partial\tilde{Z}}{\partial z}(t-s,Z(s,z))\cdot(\tilde{V}-V)|_{Z(s,z)}\,ds}.
\end{equation}
(On manifolds the statement requires coordinates to make sense, but the result does not depend on the choice of coordinates.) The proof follows from an application of the fundamental theorem of calculus to the function $s\mapsto \tilde{Z}(t-s,Z(s,z))$. 

Suppose now that $\{a_{\nu}\}$ and $\{\tilde{a}_{\nu}\}$ are two sets of parameters giving the same travel time data. We apply the pseudolinearization formula \eqref{su-pseudolin-eq} to $V$ and $\tilde{V}$ corresponding to the Hamilton flow of $G$, where $G$ is one of the wave speeds above. (Thus in the pseudolinearization formula we replace $z$ and $Z$ by $(x,\xi)$ and $(X,\Xi)$ since the manifold of interest is the cotangent bundle.) If $(x,\xi)$ is an inward-pointing covector on the boundary of $\Omega$, and $\tau(x,\xi)$ denotes the travel time of the Hamiltonian trajectory starting at $(x,\xi)$ (for either collection of parameters), then $(\tilde{X}(\tau(x,\xi),(x,\xi)),\tilde{\Xi}(\tau(x,\xi),(x,\xi))) = (X(\tau(x,\xi),(x,\xi)),\Xi(\tau(x,\xi),(x,\xi)))$ since the lens relations agree, and hence
\begin{align*} \vec{0}_6 = \int_0^{\tau(x,\xi)}&\frac{\partial(\tilde{X},\tilde{\Xi})}{\partial(x,\xi)}(\tau(x,\xi)-s,(X(s,x,\xi),\Xi(s,x,\xi)))\\
&\cdot(\tilde{V}-V)|_{(X(s,x,\xi),\Xi(s,x,\xi))}\,ds.
\end{align*}
Note that we can extend the limits of the integral to infinity by extending the trajectories to never return to $\Omega$ since $\tilde{V}-V$ can be extended to zero outside $\Omega$. If for $(x,\xi)\in T^*\Omega$ (i.e. the interior) we also define $\tau(x,\xi)$ as the travel time of the Hamilton trajectory starting at $(x,\xi)$ before the trajectory leaves $\Omega$, then we have that
\[\tau(X(t,x,\xi),\Xi(t,x,\xi)) = \tau(x,\xi) - t\]
and hence the above equation can be written as
\begin{equation}
\begin{aligned}
\vec{0}_6 = \int_{\mathbb{R}}&\frac{\partial(\tilde{X},\tilde{\Xi})}{\partial(x,\xi)}(\tau(X(t,x,\xi),\Xi(t,x,\xi)),(X(t,x,\xi),\Xi(t,x,\xi)))\\
&\cdot(\tilde{V}-V)|_{(X(t,x,\xi),\Xi(t,x,\xi))}\,dt 
\end{aligned}
\label{pseudo1}
\end{equation}
for any $(x,\xi)\in T^*\Omega$ (not just for inward-pointing covectors on the boundary).

Let $\tilde{G}$ denote the wave speed with respect to the $\tilde{a}_{\nu}$. We then have
\[\tilde{V} - V = -\partial_x(\tilde{G}-G)\cdot\partial_{\xi} + \partial_{\xi}(\tilde{G}-G)\cdot\partial_x.\]
Now, if $r_{\nu} = 0$ for all $\nu$, then of course $\tilde{G} - G\equiv 0$; however, in general we can use the fundamental theorem of calculus to write
\begin{align*} \tilde{G} - G &= G(a_{11}+r_{11},a_{33}+r_{33},E^2+r_{E^2};\xi) - G(a_{11},a_{33},E^2;\xi) \\
&=\int_0^1{\frac{d}{ds}\left(G(a_{11}+sr_{11},a_{33}+sr_{33},E^2+sr_{E^2};\xi)\right)\,ds} \\
&=\sum_{\nu}{\int_0^1{\frac{\partial G}{\partial\nu}(a_{11}+sr_{11},a_{33}+sr_{33},E^2+sr_{E^2};\xi)r_{\nu}\,ds}}\\
&=\sum_{\nu}{E^{\nu}(x,\xi)r_{\nu}(x)}
\end{align*}
with $\nu$ ranging over $a_{11}$, $a_{33}$, and $E^2$ and
\begin{equation}
E^{\nu}(x,\xi) = \int_0^1{\frac{\partial G}{\partial\nu}(a_{11}(x)+sr_{11}(x),a_{33}(x)+sr_{33}(x),E^2(x)+sr_{E^2}(x);\xi)\,ds},\label{enu}
\end{equation}
and hence
\[\tilde{V}-V = -\left(\sum_{\nu}{E^{\nu}\nabla r_{\nu} + \partial_xE^{\nu}r_{\nu}}\right)\cdot\partial_{\xi} + \left(\sum_{\nu}{\partial_{\xi}E^{\nu}r_{\nu}}\right)\cdot\partial_x.\]
Substituting this into \eqref{pseudo1} and keeping the bottom three rows (i.e. the rows corresponding to $\frac{\partial\tilde{\Xi}}{\partial(x,\xi)}$) gives
$\vec{0}_3 = \sum_{\nu}{I^{\nu}[\nabla r_{\nu}](x,\xi) + \tilde{I}^{\nu}[r_{\nu}](x,\xi)}$ for all $(x,\xi)$, where
\begin{equation}
\begin{aligned}
&I^{\nu}[f_1,f_2,f_3](x,\xi) \\
&= -\int_{\mathbb{R}}{\left(E^{\nu}\frac{\partial\tilde\Xi}{\partial\xi}(\tau(\cdot),\cdot)\right)(X(t),\Xi(t))\cdot\begin{pmatrix} f_1 \\ f_2 \\ f_3 \end{pmatrix}(X(t))\,dt}
\end{aligned}
\label{inu}
\end{equation}
and
\begin{equation}
\begin{aligned}
&\tilde{I}^{\nu}[f](x,\xi) \\
&= \int_{\mathbb{R}}{\left(-\partial_xE^{\nu}\frac{\partial\tilde\Xi}{\partial\xi}(\tau(\cdot),\cdot) + \partial_{\xi}E^{\nu}\frac{\partial\tilde\Xi}{\partial x}(\tau(\cdot),\cdot)\right)(X(t),\Xi(t))f(X(t))\,dt}
\end{aligned}
\label{itildenu}
\end{equation}
with $(X(t),\Xi(t)) = (X(t,x,\xi),\Xi(t,x,\xi))$. In other words, we have constructed operators $I^{\nu}$ and $\tilde{I}^{\nu}$, \emph{which depend on the unknown parameters $a_{\nu}$ and $\tilde{a}_{\nu}$}, for which the differences $r_{\nu}$ satisfy a linear equation. (The terms in the integrand depend on the choice of dynamics used, i.e. whether we are considering the dynamics of the $qP$ or $qSV$ wave speed, and so we will denote these operators $I^{\nu}_{\pm}$ and $\tilde{I}^{\nu}_{\pm}$ depending on the choice of dynamics used.) Note that these operators map functions on $\mathbb{R}^3$ to functions on $T^*\mathbb{R}^3$, so we will compose with a ``formal adjoint'' operator to map back to functions on $\mathbb{R}^3$. Thus, for $v: T^*\mathbb{R}^3\rightarrow\mathbb{C}$, define
\begin{equation}
L_{\pm}[v](x) = \int_{\mathbb{S}^2}{\chi(x,\omega)\left(\frac{\partial\tilde\Xi}{\partial\xi}(\tau(x,\xi(\omega)),(x,\xi(\omega)))\right)^{-1}v(x,\xi(\omega))\,d\omega} \label{adjoint}
\end{equation}
with $\xi(\omega)$ defined in \eqref{xiomega}, $\chi$ real-valued and smooth (we will mostly consider $\chi$ which are identically $1$ in a neighborhood of the equatorial sphere $\{(x,\omega)\in\mathbb{R}^3\times\mathbb{S}^2\,:\,\overline{\xi}(x)\cdot\omega = 0\}$ perpendicular to $\overline{\xi}$), and the $\pm$ determining whether to consider the dynamics of the $qP$ or $qSV$ wave speed. We now let $N^{\nu}_{\pm} = L_{\pm}\circ I^{\nu}_{\pm}$ and $\tilde{N}^{\nu}_{\pm} = L_{\pm}\circ \tilde{I}^{\nu}_{\pm}$, so that we have the formulas
\begin{equation}
\vec{0}_3 = \sum_{\nu}{N^{\nu}_{\pm}[\nabla r_{\nu}] + \tilde{N}^{\nu}_{\pm}[r_{\nu}]}.\label{npseudo}
\end{equation}
We now analyze the operators $N^{\nu}_{\pm}$ and $\tilde{N}^{\nu}_{\pm}$, in hopes of turning the equation \eqref{npseudo} into a stability estimate, or even better, to conclude that $r_{\nu}\equiv 0$. In Section \ref{excomp}, we prove that these operators are (matrix-valued) pseudodifferential operators (abbreviated $\Psi$DO) and analyze the behavior of their symbols, as summarized in the following theorem:
\begin{theorem}
\label{symbolthm}
Let $\chi\equiv 1$ near the equatorial sphere. For $\nu = 11, 33, E^2$, we have that $N^{\nu}_{\pm}$ and $\tilde{N}^{\nu}_{\pm}$ are matrix-valued $\Psi$DOs of order $-1$, with $N^{\nu}_{\pm}$ having scalar-valued (i.e. multiples of the identity) principal symbols. In addition, $N^{11}_+$ is elliptic, while all other principal symbols $\sigma_{-1}(N^{\nu}_{\pm})$ vanish at least quadratically on
\[\Sigma = \text{span }\overline{\xi} = \{(x,\xi)\in T^*\mathbb{R}^3\,:\,\xi = s\overline{\xi}(x)\text{ for some }s\},\]
with all cases of quadratic vanishing being non-degenerate except for $N^{33}_-$, and also except for $N^{11}_-$ if $E^2$ is known to be identically zero. Moreover, $\sigma_{-1}(N^{33}_-)$ will vanish quartically on $\Sigma$. 

In addition, for each $N^{\nu}_{\pm}$ with vanishing principal symbol on $\Sigma$, the subprincipal symbol (of the left-reduced symbol) on $\Sigma$ is purely imaginary and is linear modulo an overall factor of $|\zeta|^{-3}$, and under suitable geological assumptions (see the remarks following \eqref{subsymb}) we have that the subprincipal symbols of $N^{11}_-$ (if it is known that $E^2>0$ everywhere), $N^{33}_+$, and $N^{E^2}_{\pm}$ are nonvanishing on $\Sigma$ away from the zero section. On the other hand, $N^{33}_-$ has vanishing subprincipal symbol at $\Sigma$, in addition to having quartically vanishing principal symbol on $\Sigma$.

Finally, all operators $\tilde{N}^{\nu}_{\pm}$ except $\tilde{N}^{11}_{\pm}$ have (matrix-valued) principal symbols which vanish on $\Sigma$.
\end{theorem}
A formula for the subprincipal symbol is given by \eqref{subsymb}. In particular, for $N = N^{11}_-$ (if $E^2>0$), $N^{33}_+$, and $N^{E^2}_{\pm}$, we can write their (left) full symbols in the form
\[\sigma_L(N)(x,\zeta) = (p_m(x,\zeta)+ip_{m-1}(x,\zeta))\text{Id} + P_{m-1}(x,\zeta) + P_{m-2}(x,\zeta)\]
with $m = -1$, where $p_i\in S^i(T^*\mathbb{R}^3;\mathbb{R})$, $p_m$ is nonnegative and vanishes only at $\Sigma$, where it vanishes nondegenerately quadratically, $p_{m-1}$ satisfies a uniform nonzero bound on $\Sigma\backslash\{0\}$, and $P_i\in S^i(T^*\mathbb{R}^3;\text{Mat}_{3\times 3}(\mathbb{C}))$, with $P_{m-1}$ vanishing on $\Sigma$. 
Note that if the principal symbol of an operator vanishes quadratically on its characteristic set, then its subprincipal symbol is well-defined there, and hence it makes invariant sense to discuss the non-vanishing of these operators' subprincipal symbols.

We note that the form of the operators obtained above depended heavily on the explicit formulas for the eigenvalues of the elastic wave operator and their dependence on the material parameters. For elasticity with different kinds of symmetries, it is \emph{a priori} unclear what form the corresponding operators should take without looking at explicit expressions for the corresponding eigenvalues.

\subsection{Extended outline of the proofs of the main theorems}
\label{extended-outline-subsec}
With Theorem \ref{symbolthm} establishing the qualitative behavior of the operators $N^{\nu}_{\pm}$ and $\tilde{N}^{\nu}_{\pm}$, we now ask if there are ways of studying these operators, with a particular goal to obtain parametrices for these operators. We note that the theorem gives that the symbols of some of the operators $N^{\nu}_{\pm}$ are of ``parabolic type,'' i.e. are of the form $p_m+ip_{m-1}$, where $p_m\in S^m$ is nonnegative and vanishes nondegenerately on $\Sigma$ and $p_{m-1}\in S^{m-1}$ is real-valued and nondegenerate on $\Sigma$. A prototypical example of such an operator where $m=2$ is $|\xi|^2+i\tau$ on $T^*(\mathbb{R}^{n-1}_x\times\mathbb{R}_t)$, $\Sigma = \{(x,t,\xi,\tau)\,:\,\xi = 0\}$, which is the symbol of the heat operator on $\mathbb{R}^{n-1}$. Note that the heat operator, while not elliptic, still satisfy certain desirable properties; in particular it is hypoelliptic.

In Section \ref{invparsec}, we analyze symbols of ``inverse parabolic type'', i.e. of the form $q = 1/(p_m+ip_{m-1})$. Despite the lack of (order $m$) ellipticity of the parabolic symbol on $\Sigma$, its inverse does belong to a symbol calculus first studied by Boutet de Monvel in \cite{bdm}, consisting of symbols satisfying estimates of the form
\[|W^{\alpha}V^{\beta}q(x,\zeta)|\le C_{\alpha,\beta}|\zeta|^md_{\Sigma}^{k-|\alpha|}\quad\text{for }|\zeta|\gg 1\]
whenever $V^{\beta}$ is a product of homogeneous vector fields of degree $0$ on $T^*\mathbb{R}^n$ tangent to $\Sigma$ (e.g. a derivative in $x$ under appropriate coordinates) and $W^{\alpha}$ is product of homogeneous vector fields of degree $0$ not necessarily tangent to $\Sigma$ (e.g. a derivative in $\zeta$ times a power of $\zeta$), and $d_{\Sigma}^2 = |p|^2 + 1/|\zeta|$ where $p = (p_1,\dots,p_{\nu})$ are boundary defining functions for $\Sigma$ which are homogeneous of degree $0$ (so e.g. for $\Sigma = \{\zeta'=0\}$ we can take $d_{\Sigma}^2 = \frac{|\zeta'|^2}{|\zeta|^2}+\frac{1}{|\zeta|}$). For example, the inverse heat symbol satisfies the above estimates for $m=-2$ and $k=-2$. Symbols satisfying such estimates turn out to be invariantly defined regardless of coordinates, and the corresponding calculus based on such symbols enjoys properties similar to that of the $(1/2,0)$-type H\"ormander symbol calculus. In our case, recalling our assumption that $\overline{\xi}$ has integrable kernel, and hence $\Sigma$ is a line subbundle with integrable kernel, we can obtain even better properties about the calculus (such as a well-defined notion of principal symbol which is compatible with composition), which we develop in Section \ref{invparsec}. This is due to the fact that near any point we can take ``foliated'' local coordinates where $\overline{\xi}$ is locally a multiple of $dx_n$, and hence improved properties can follow if we quantize our symbol class only with respect to the foliated coordinates.

We use the symbol calculus developed in Section \ref{invparsec} to make recovery arguments in Section \ref{recovsec}, where we prove Theorems \ref{oneparam}, \ref{twoparam}, and \ref{func}. The main technical results presented in this section are the ``stability estimates'' mentioned earlier: in essence, in each situation we are trying to recover one or two parameters, with the other parameters either known or reducing to the parameters of interest via a functional relationship. The pseudolinearization formulas from \eqref{npseudo} provide operators (representable as a matrix-valued operator $N$) upon which applying the differences of the parameters of interest $r_{\nu}$ gives identically zero, assuming the coefficients give the same travel time data. We thus aim to obtain an estimate for general functions $u$ in terms of $N[u]$ for $u$ suitably supported, and in most cases we are able to obtain a stability estimate of the form
\begin{equation}
\label{stabest}
\|\nabla u\|_{L^2}\le C\left(\|N[u]\|_{H^2} + \|u\|_{H^{1/2}}\right).
\end{equation}
Thus, if the travel times agree, so that we have $N[r_{\nu}]\equiv 0$, then we obtain the estimate $\|\nabla r_{\nu}\|_{L^2}\le C\|r_{\nu}\|_{H^{1/2}}$, which combined with width assumptions discussed above gives injectivity. To obtain this stability estimate, we use Theorem \ref{symbolthm} to show that the corresponding operator $N$, while not elliptic in the classical sense, are elliptic in the modified calculus developed in Section \ref{invparsec}, and hence admit a parametrix in this class, which eventually leads to the desired stability estimate, thus leading to the main results.

\begin{remark}
\label{global-remark}
We remark that the argument is being made globally, i.e. unlike \cite{ti} we are not considering localizing via an artificial boundary. This is because attempting the analogous argument in this case, where the operators are now put in the framework of the Melrose scattering calculus to deal with the boundary, will result in operators whose subprincipal parts degenerate near the boundary, thus precluding the use of a symbol calculus analogous to that developed in Section \ref{invparsec}. A sketch of the relevant calculation is given in Section \ref{scat-op-sec}, specifically Proposition \ref{scatsub}. Further work in this direction would be desirable in order to fully apply the results in this work to the setting of \cite{ti}. Moreover, if a local result is obtained, then one may be able to obtain a similar result in the case of discontinuous parameters, for example by using the methods in \cite{faults} where the authors used a scattering control method together with results in the smooth case to make an argument for the piecewise smooth case, thus allowing an application to the fault monitoring problem where the parameters may be discontinuous.
\end{remark}

\section{Symbol computations}
\label{comp-sec}
In this section, we compute several quantities related to the symbols of the operators $N^{\nu}_{\pm}$ and $\tilde{N}^{\nu}_{\pm}$, with the purpose to prove Theorem \ref{symbolthm} regarding the structure of the symbols of these operators. In Section \ref{matrixweight} we analyze the symbol of operators arising from general matrix-weighted ray transforms. In Section \ref{hamdyn} we compute several quantities related to the Hamiltonian dynamics with respect to the $qP$ and $qSV$ wave speeds, as well as a justification of why the travel time data determines the lens relation data. In Section \ref{excomp}, we apply the general formulas obtained in Section \ref{matrixweight}, together with quantities computed in Section \ref{intro-sec} and Section \ref{hamdyn}, to prove Theorem \ref{symbolthm}, as well as to perform a more quantitative analysis of the symbols which will be useful in Section \ref{invparsec}.

\subsection{Symbol of operators associated to matrix-weighted ray transforms}
\label{matrixweight}

We analyze the symbol of operators arising from matrix-weighted ray transforms for trajectories arising from Hamiltonian dynamics. So let $p$ be a Hamiltonian function, i.e. function on $T^*\mathbb{R}^n$. Denote $(X(t,x_0,\xi_0),\Xi(t,x_0,\xi_0))$ the Hamiltonian flow with respect to $p$ starting at $(x_0,\xi_0)\in T^*\mathbb{R}^n$. Suppose that
\[I[u](x_0,\xi_0) = \int_{\mathbb{R}}{A(X(t,x_0,\xi_0),\Xi(t,x_0,\xi_0))u(X(t,x_0,\xi_0))\,dt}\]
and
\[L[v](x) = \int_{\mathbb{S}^{n-1}}{B(x,\omega)v(x,\xi(\omega;x))\,d\mathbb{S}^{n-1}(\omega)}\]
where $\xi(\omega;x)$ satisfies $\frac{\partial p}{\partial \xi}(x,\xi(\omega;x)) = \omega$. Then for $N = L\circ I$ we have
\begin{align*} N[u](x) &= \int_{\mathbb{S}^{n-1}\times\mathbb{R}}{B(x,\omega)A(X(t,x,\xi(\omega)),\Xi(t,x,\xi(\omega)))u(X(t,x,\xi(\omega)))\,dt\,d\omega} \\
&= \int_{\mathbb{S}^{n-1}\times\mathbb{R}\times\mathbb{R}^n}B(x,\omega)A(X(t,x,\xi(\omega)),\Xi(t,x,\xi(\omega)))\delta(X(t,x,\xi(\omega))-y)u(y)\,dy\,dt\,d\omega \\
&= (2\pi)^{-n}\int_{\mathbb{S}^{n-1}\times\mathbb{R}\times\mathbb{R}^n\times\mathbb{R}^n}B(x,\omega)A(X(t,x,\xi(\omega)),\Xi(t,x,\xi(\omega)))  e^{i(X(t,x,\xi(\omega))-y)\cdot\zeta}u(y)\,d\zeta\,dy\,dt\,d\omega
\end{align*}
and so the Schwartz kernel of $N$ is given by
\begin{align*} K(x,y) &= (2\pi)^{-n}\int_{\mathbb{S}^{n-1}\times\mathbb{R}\times\mathbb{R}^n}e^{i(X(t,x,\xi(\omega))-y)\cdot\zeta} B(x,\omega)A(X(t,x,\xi(\omega)),\Xi(t,x,\xi(\omega)))\,d\zeta\,dt\,d\omega \\
&= (2\pi)^{-n}\int_{\mathbb{R}^n}{e^{i(x-y)\cdot\zeta}\sigma_L(N)(x,\zeta)d\zeta}
\end{align*}
where
\[ \sigma_L(N)(x,\zeta) = \int_{\mathbb{S}^{n-1}\times\mathbb{R}}{e^{i(X(t,x,\xi(\omega))-x)\cdot\zeta}C(x,t,\omega)\,dt\,d\omega}, \quad
C(x,t,\omega) = B(x,\omega)A(X(t,x,\xi(\omega)),\Xi(t,x,\xi(\omega))).
\]
It follows that $N$ is a $\Psi$DO corresponding to the symbol $\sigma_L(N)$, provided that this is indeed a symbol. To show this is a symbol (and analyze its properties), we make a stationary phase argument.

Fixing $(x,\zeta)$, and letting $\gamma_{x,\omega}(t) = X(t,x,\xi(\omega;x))$, we have that
\[\gamma_{x,\omega}(t) - x = \omega t + \alpha(\omega;x)t^2 + O(t^3)\]
where we can calculate $\alpha(\omega;x)$ from Hamilton dynamics as
\begin{equation}
\begin{aligned} \alpha(\omega;x) &= \frac{1}{2}\frac{d^2}{dt^2}\Big|_{t=0}X(t,x,\xi(\omega)) \\
&= \frac{1}{2}\frac{d}{dt}\Big|_{t=0}\partial_{\xi}p(X(t,x,\xi(\omega)),\Xi(t,x,\xi(\omega))) \\
&=\frac{1}{2}(\dot{X}(0)\cdot\partial_x\partial_{\xi}p(x,\xi(\omega)) + \dot{\Xi}(0)\cdot\partial_{\xi}\partial_{\xi}p(x,\xi(\omega)))\\
&=\frac{1}{2}(\omega\cdot\partial_x\partial_{\xi}p(x,\xi(\omega)) - \partial_xp(x,\xi(\omega))\cdot\partial_{\xi}\partial_{\xi}p(x,\xi(\omega))).
\end{aligned}
\label{alpha}
\end{equation}
For $\zeta\ne 0$, if we decompose $\omega$ with respect to $\zeta$ as $\omega = \omega_{\parallel}\frac{\zeta}{|\zeta|} + \sqrt{1-\omega_{\parallel}^2}\omega'$, $\omega_{\parallel}\in[-1,1]$, $\omega'\in\zeta^{\perp}\cap \mathbb{S}^{n-1}$, then $(\omega_{\parallel},\omega')$ provide valid coordinates on $\mathbb{S}^{n-1}$ away from $\omega = \pm\frac{\zeta}{|\zeta|}$, and we have
\[\phi(t,\omega;x) := \frac{\zeta\cdot(\gamma_{x,\omega}(t)-x)}{|\zeta|} = \omega_{\parallel}t + \frac{\zeta}{|\zeta|}\cdot\alpha(\omega;x)t^2+ O(t^3).\]
We can also view $\phi$ as a function of $t$ and $\omega_{\parallel}$ (with the remaining coordinate $\omega'$ a parameter). Since
\begin{align*} 
\partial_t\phi(t,\omega;x) &= \omega_{\parallel} + 2\frac{\zeta}{|\zeta|}\cdot\alpha(\omega;x)t+O(t^2)\\
\partial_{\omega_{\parallel}}\phi(t,\omega;x) &= t + O(t^2),
\end{align*}
it follows that (for sufficiently small $t$) the only critical points of $\phi$ are at $\{t=0,\omega_{\parallel} = 0\}$, and those critical points are non-degenerate. 

We thus write
\[\phi(t,\omega;x) = \omega_{\parallel}t + \frac{\zeta}{|\zeta|}\cdot\alpha(0,\omega';x)t^2 + R(t,\omega;x)\]
where $R(t,\omega;x) = O(\omega_{\parallel}t^2) + O(t^3)$, so that $\phi$ is written as a quadratic form in $(t,\omega_{\parallel})$ (with coefficients depending on $\omega'$) plus a remainder. Let $\phi_s(t,\omega;x) = \omega_{\parallel}t+\frac{\zeta}{|\zeta|}\cdot\alpha(0,\omega';x)t^2+sR(t,\omega;x)$. We now let
\begin{equation}
I(x,\zeta;s) = \int_{\mathbb{S}^{n-1}}\int_{\mathbb{R}}{e^{i|\zeta|\phi_s(t,\omega;x)}C(x,t,\omega)\,dt\,d\mathbb{S}^{n-1}(\omega)}.\label{I(s)}
\end{equation}
Then $\sigma_L(N)(x,\zeta) = I(x,\zeta;1)$, and for any $N$ we can write $I(x,\zeta;1) = \sum_{j=0}^{N-1}{\frac{I^{(j)}(x,\zeta;0)}{j!}} + \frac{I^{(N)}(x,\zeta;s)}{N!}$ for some $s\in(0,1)$ by Taylor's theorem. We can thus study the asymptotics of terms of the form $I^{(j)}(x,\zeta;0)$ to analyze the asymptotics of $\sigma_L(N)(x,\zeta)$. (Note that we can also insert a cutoff in $t$ which is identically $1$ in a neighborhood of $0$ without affecting asymptotics, since on the difference we can integrate by parts using the ``no conjugate points'' assumption of \eqref{noconjpts}. In particular, we can insert a cutoff in $t$ such that $\phi_s$ has no critical points in $(t,\omega_{\parallel})$ besides $(t,\omega_{\parallel}) = (0,0)$ for all $s\in[0,1]$ for all $t$ in the support of the cutoff, so that in effect we are free to assume that $\phi_s$ really has no critical points aside from $(t,\omega_{\parallel}) = (0,0)$.)

We first compute the asymptotics of
\[ I(x,\zeta;0) = \int_{\mathbb{S}^{n-1}}\int_{\mathbb{R}}{e^{i|\zeta|\omega_{\parallel}t+\frac{\zeta}{|\zeta|}\cdot\alpha(0,\omega';x)t^2}C(x,t,\omega)\,dt\,d\omega}.\]
We change to variables $\omega = (\omega_{\parallel},\omega')$, with $d\mathbb{S}^{n-1}(\omega) = {(1-\omega_{\parallel}^2)^{(n-3)/2}}\,d\omega_{\parallel}\,d\mathbb{S}^{n-2}(\omega')$. Letting $\tilde{C}(x,t,\omega_{\parallel},\omega') = {(1-\omega_{\parallel}^2)^{(n-3)/2}C(t,x,\omega)}$, we can rewrite the above equation as
\begin{align*}&I(x,\zeta;0) = \\
&\int_{\zeta^{\perp}\cap\mathbb{S}^{n-1}}{\left(\int_{-1}^1\int_{\mathbb{R}}{e^{i|\zeta|\omega_{\parallel}t+\frac{\zeta}{|\zeta|}\cdot\alpha(0,\omega';x)t^2}\tilde{C}(x,t,\omega_{\parallel},\omega')\,dt\,d\omega_{\parallel}}\right)\,d\mathbb{S}^{n-2}(\omega')}.
\end{align*}
The phase can thus be written as $|\zeta|\langle Q(\omega',x)(t,\omega_{\parallel}),(t,\omega_{\parallel})\rangle$ where, with respect to the coordinates $(t,\omega_{\parallel})$, $Q(\omega',x)$ is the matrix
\[Q(\omega',x) = \begin{pmatrix} 2\frac{\zeta}{|\zeta|}\cdot\alpha(0,\omega';x) & 1 \\ 1 & 0 \end{pmatrix}.\]
Thus $\det Q(\omega',x) = -1$, $\text{sgn }Q(\omega',x) = 0$, and
\[\quad Q(\omega',x)^{-1} = \begin{pmatrix} 0 & 1 \\ 1 & -2\frac{\zeta}{|\zeta|}\cdot\alpha(0,\omega';x) \end{pmatrix}\]
and hence by the method of stationary phase (cf. \cite{alpdo1}) we have
\begin{align*}&\int_{\zeta^{\perp}\cap\mathbb{S}^{n-1}}{\left(\int_{-1}^1\int_{\mathbb{R}}{e^{i|\zeta|\omega_{\parallel}t+\frac{\zeta}{|\zeta|}\cdot\alpha(0,\omega')t^2}\tilde{C}(x,t,\omega_{\parallel},\omega')\,dt\,d\omega_{\parallel}}\right)\,d\mathbb{S}^{n-2}(\omega')} \\
&= a_{-1}|\zeta|^{-1} + a_{-2}|\zeta|^{-2} + O(|\zeta|^{-3})
\end{align*}
where
\[a_{-1} = 2\pi\int_{\zeta^{\perp}\cap\mathbb{S}^{n-1}}{\tilde{C}|_{t=0,\omega_{\parallel}=0}\,d\mathbb{S}^{n-2}(\omega')} = 2\pi\int_{\zeta^{\perp}\cap\mathbb{S}^{n-1}}{C(x,0,\omega')\,d\mathbb{S}^{n-2}(\omega')}\]
and
\begin{align*} a_{-2} &= 2\pi i \int_{\zeta^{\perp}\cap\mathbb{S}^{n-1}}{\left.\left[\left(\partial_t\partial_{\omega_{\parallel}} - \frac{\zeta}{|\zeta|}\cdot\alpha(0,\omega')\partial^2_{\omega_{\parallel}}\right)\tilde{C}\right]\middle|\right._{t=0,\omega_{\parallel} = 0}\,d\mathbb{S}^{n-2}(\omega')}\\
&=2\pi i\int_{\zeta^{\perp}\cap\mathbb{S}^{n-1}}\left(\partial_t\partial_{\omega_{\parallel}} - \frac{\zeta}{|\zeta|}\cdot\alpha(0,\omega')\partial^2_{\omega_{\parallel}}\right)C(x,0,\omega')\\
&\phantom{=2\pi i}- (n-3)\frac{\zeta}{|\zeta|}\cdot\alpha(0,\omega';x)C(x,0,\omega')\,d\mathbb{S}^{n-2}(\omega').
\end{align*}
(Note that the two quantities above depend only on $\frac{\zeta}{|\zeta|}$.) In particular, this shows that $\sigma_L(N)$ is a symbol of order (at most) $-1$.

Now suppose that $\zeta$ has the property that $\omega\in\zeta^{\perp}\implies C(x,0,\omega) = 0$. Then for such $\zeta$, we have $a_{-1} = 0$, and $a_{-2}$ has the simpler formula
\[a_{-2} = 2\pi i\int_{\zeta^{\perp}\cap\mathbb{S}^{n-1}}{\left(\partial_t\partial_{\omega_{\parallel}} - \frac{\zeta}{|\zeta|}\cdot\alpha(0,\omega';x)\partial^2_{\omega_{\parallel}}\right)C(x,0,\omega')\,d\mathbb{S}^{n-2}(\omega')}.\]
We now consider $C$ such that $C(x,t,\omega) = F(x,t,\omega)g^2(x,t,\omega)$, where $F$ is matrix-valued but $F|_{t=0}$ is scalar-valued, and $g$ is scalar-valued and has the property that there exists a nowhere vanishing 1-form $\xi_0(x)$ such that for every $x$ we have
\[\{\omega\in\mathbb{S}^{n-1}\,:\,g(x,0,\omega) = 0\} = \ker\xi_0(x)\cap\mathbb{S}^{n-1}.\]
Then $a_{-1}$ is scalar-valued, and furthermore $a_{-1}$ vanishes when $\zeta$ is a multiple of $\xi_0(x)$. (If we furthermore assume $F|_{t=0}$ is bounded away from zero, i.e. uniformly positive or negative, then $a_{-1}$ is always nonnegative/nonpositive and vanishes only on the span of $\xi_0$.) Furthermore, the expression for the subprincipal symbol at a multiple of $\xi_0(x)$ can be rewritten as well: indeed for $\omega'$ annihilated by $\xi_0$ we have
\[\partial_t\partial_{\omega_{\parallel}}C(x,0,\omega') = \partial_t\partial_{\omega_{\parallel}}[Fg^2](x,0,\omega') = [2F\partial_tg\partial_{\omega_{\parallel}}](x,0,\omega')\]
since all other terms in Leibniz's rule would contain a factor of $g$ which vanishes when $t=0$ and $\omega'$ is annihilated by $\xi_0$. Similarly
\[\partial^2_{\omega_{\parallel}}C(x,0,\omega') = 2F(x,0,\omega')(\partial_{\omega_{\parallel}}g(x,0,\omega'))^2\]
and hence for $\zeta$ parallel to $\xi_0$ we have
\begin{align*}
a_{-2} &= 2\pi i\int_{\xi_0^{\perp}\cap\mathbb{S}^{n-1}}2F(x,0,\omega')\left(\partial_tg(x,0,\omega')\partial_{\omega_{\parallel}}g(x,0,\omega') \right.\\
&\phantom{=2\pi i}\left.- \frac{\zeta}{|\zeta|}\cdot\alpha(\omega')(\partial_{\omega_{\parallel}}g(x,0,\omega'))^2\right)\,d\mathbb{S}^{n-2}(\omega').
\end{align*}
In particular, the expression for the subprincipal part depends only on the values of the prefactor $F$ and not its derivatives.

To recap, we have analyzed the asymptotic expansion of the term $I(x,\zeta;0)$ with the coefficients of $|\zeta|^{-1}$ and $|\zeta|^{-2}$ given by the above expressions. We now show the remaining terms in the Taylor expansion $I(x,\zeta;1) = \sum_{j=0}^{N-1}{\frac{I^{(j)}(x,\zeta;0)}{j!}} + \frac{I^{(N)}(x,\zeta;s)}{N!}$ do not contribute to the asymptotic expansion. From \eqref{I(s)}, we have
\[I^{(j)}(x,\zeta;s) = (i|\zeta|)^j\int_{\mathbb{S}^{n-1}}\int_{\mathbb{R}}{e^{i|\zeta|\phi_s(t,\omega)}R(t,\omega;x)^jC(x,t,\omega)\,dt\,d\mathbb{S}^{n-1}(\omega)}\]
with $R(t,\omega;x) = O(\omega_{\parallel}t^2)+O(t^3)$. Since this vanishes cubically at the set of critical points $\{t=0,\omega_{\parallel} = 0\}$, we have that $I^{(2j)}(x,\zeta;s) = O(|\zeta|^{-j})$ for all $s\in(0,1)$ (cf. \cite{alpdo1}, Theorem 7.7.1 or 7.7.5). So we set $N = 6$ so that $I^{(6)}(x,\zeta;s)/6! = O(|\zeta|^{-3})$ and analyze $I^{(j)}(x,\zeta;0)$ for $1\le j\le 5$. We have
\[I^{(j)}(x,\zeta;0) = (i|\zeta|)^{j}\int_{\mathbb{S}^{n-1}}\int_{\mathbb{R}}{e^{i|\zeta|\phi_0(t,\omega)}R(t,\omega;x)^jC(x,t,\omega)\,dt\,d\mathbb{S}^{n-1}(\omega)}.\]
The stationary phase formula gives an asymptotic expansion where the coefficients are obtained by integrating appropriate powers of the differential operator $-i\langle Q^{-1}D,D\rangle/2$ applied to the amplitude $R^jC$, where $-\langle Q^{-1}D,D\rangle/2 = {\partial_t\partial_{\omega_{\parallel}} - \frac{\zeta}{|\zeta|}\cdot\alpha(0,\omega')\partial^2_{\omega_{\parallel}}}$. In particular, this differential operator is a sum of terms with at most one $t$ derivative, and thus $(\langle Q^{-1}D,D\rangle)^k$ is a sum of terms each with at most $k$ derivatives in $t$. This matters since $R(t,\omega;x) = O(t^2)$ (i.e. $t^2$ times smooth), so applying differential operators with at most one derivative in $t$ will only reduce the power of $t$ by one (and thus the term vanishes to higher order than initially expected).  The coefficients of $|\zeta|^{-1}$ and $|\zeta|^{-2}$ in the asymptotic expansion of $I^{(j)}$ are the coefficients of $|\zeta|^{-1-j}$ and $|\zeta|^{-2-j}$ obtained in the stationary phase expansion of the integral $\int_{\mathbb{S}^{n-1}}\int_{\mathbb{R}}{e^{i|\zeta|\phi_0(t,\omega)}R(t,\omega;x)^jC(x,t,\omega)\,dt\,d\mathbb{S}^{n-1}(\omega)}$, which in turn is a multiple of
\[\int_{\{t=0,\omega_{\parallel}=0\}}{\left(\langle Q^{-1}D,D\rangle\right)^k[R^jC](0,\omega')\,d\mathbb{S}^{n-2}(\omega')},\quad k=j\text{ or }j+1.\]
Since $[R^jC](t,\omega) = O(t^{2j})$, we have that $\left(\langle Q^{-1}D,D\rangle\right)^k[R^jC](t,\omega) = O(t^{2j-k})$ since at most $k$ derivatives in $t$ are applied, and all other derivatives fall on the smooth prefactor which does not affect decay. In particular, if $k = j$ and $j\ge 1$, we have that $2j-k>0$, and hence $\left(\langle Q^{-1}D,D\rangle\right)^k[R^jC]$ vanishes at the critical set, i.e. the above integrals and the corresponding coefficients are zero. Thus the derivatives do not contribute to the principal symbol at all. If $k = j+1$ and $j\ge 2$, we also have that $2j-k>0$, and thus the derivatives of order $2$ and higher do not contribute to the subprincipal symbol at all. Thus to recap:
\[
\text{the principal symbol of }N\text{ is given by the }|\zeta|^{-1}\text{ term in the asymptotic expansion of }I(x,\zeta;0)
\]
and
\begin{gather*}
\text{the subprincipal symbol of }N\text{ is given by the }|\zeta|^{-2}\text{ term}\\
\text{in the asymptotic expansion of }I(x,\zeta;0)+I^{(1)}(x,\zeta;0).
\end{gather*}
Finally, if $C$ is of the form $Fg^2$ where $F$ and $g$ are as above, then the subprincipal contribution of $I^{(1)}(0)$ will in fact vanish when $\zeta$ is a multiple of $\xi_0$. Indeed, in this case we have that $CR = Fg^2R$ vanishes quintically at the critical set when $\zeta$ is a multiple of $\xi_0$; hence applying the fourth order differential operator $(\langle Q^{-1}D,D\rangle)^2$ will still have it vanish at the critical set. Thus $I^{(1)}(0)$ does not contribute to the subprincipal symbol at the characteristic set where $\zeta$ is a multiple of $\xi_0$.

Thus, to recap, the only contribution to the principal symbol of $N$ is from $I(0)$, while the only contribution to the subprincipal symbol on $\Sigma$ is also from $I(0)$. Hence, we have $\sigma_L(N)(x,\zeta) = \sigma_{-1}(N)(x,\zeta) + \sigma_{-2}(N)(x,\zeta) + O(|\zeta|^{-3})$, where
\begin{equation}
\sigma_{-1}(N)(x,\zeta) = |\zeta|^{-1}\cdot 2\pi\int_{\zeta^{\perp}\cap\mathbb{S}^{n-1}}{C(x,0,\omega')\,d\mathbb{S}^{n-2}(\omega')} \label{psymgenfinal}
\end{equation}
and, when $C = Fg^2$ as above, for $\zeta$ a multiple of $\xi_0$ we have
\begin{equation}
\begin{aligned}
\sigma_{-2}(N)(x,\zeta) &= |\zeta|^{-2}\cdot 2\pi i\int_{\xi_0^{\perp}\cap\mathbb{S}^{n-1}}2F(x,0,\omega')\cdot\left[\partial_tg(x,0,\omega')\partial_{\omega_{\parallel}}g(x,0,\omega') \right.\\
&\phantom{=|\zeta|^{-2}\cdot 2\pi i}\left.- \frac{\zeta}{|\zeta|}\cdot\alpha(\omega')(\partial_{\omega_{\parallel}}g(x,0,\omega'))^2\right]\,d\mathbb{S}^{n-2}(\omega'). 
\end{aligned}
\label{subsymgenfinal}
\end{equation}

\subsection{Hamiltonian dynamics}
\label{hamdyn}

We now compute several quantities related to the Hamiltonian dynamics of the wave speeds which arise in computing the principal and subprincipal symbols of $N^{\nu}_{\pm}$ and $\tilde{N}^{\nu}_{\pm}$ using \eqref{psymgenfinal} and \eqref{subsymgenfinal} derived above.

For a fixed point $x\in\mathbb{R}^3$, if we perform an orthogonal change of coordinates so that $\overline{\xi}(x) = dx_3|_{x}$, then at the point $x$ we have
\begin{align*} G_{\pm}(x,\xi) &= (a_{11}(x)+a_{55}(x))|\xi'|^2+(a_{33}(x)+a_{55}(x))\xi_3^2 \\
&\pm\sqrt{((a_{11}(x)-a_{55}(x))|\xi'|^2+(a_{33}(x)-a_{55}(x))\xi_3^2)^2-4E^2(x)|\xi'|^2\xi_3^2}
\end{align*}
where $+$ denotes the $qP$ speed, $-$ denotes the $qSV$ speed, and $|\xi'|^2 = \xi_1^2+\xi_2^2$. From this, we have that
\[\partial_{\xi}G_{\pm}(x,\xi) = 2\begin{pmatrix}(a_{11}+a_{55})\xi_1 \pm \frac{((a_{11}-a_{55})|\xi'|^2+(a_{33}-a_{55})\xi_3^2)(a_{11}-a_{55})\xi_1-4E^2|\xi_3|^2\xi_1}{\sqrt{((a_{11}-a_{55})|\xi'|^2+(a_{33}-a_{55})\xi_3^2)^2-4E^2|\xi'|^2|\xi_3|^2}} \\ 
(a_{11}+a_{55})\xi_2 \pm \frac{((a_{11}-a_{55})|\xi'|^2+(a_{33}-a_{55})\xi_3^2)(a_{11}-a_{55})\xi_2-4E^2|\xi_3|^2\xi_2}{\sqrt{((a_{11}-a_{55})|\xi'|^2+(a_{33}-a_{55})\xi_3^2)^2-4E^2|\xi'|^2|\xi_3|^2}}\\
(a_{33}+a_{55})\xi_3 \pm \frac{((a_{11}-a_{55})|\xi'|^2+(a_{33}-a_{55})\xi_3^2)(a_{33}-a_{55})\xi_3-4E^2|\xi'|^2\xi_3}{\sqrt{((a_{11}-a_{55})|\xi'|^2+(a_{33}-a_{55})\xi_3^2)^2-4E^2|\xi'|^2|\xi_3|^2}}
\end{pmatrix}.\]
In particular if $\xi_3 = 0$, i.e. $\xi$ is orthogonal to $\overline{\xi}$, then
\begin{equation}
\label{wavespdorth}
G_{\pm}(x,\xi) = 2a_{\pm}(x)|\xi|^2
\end{equation}
where $a_+ = a_{11}$ and $a_- = a_{55}$, and
\begin{equation}
\label{wavespdgradorth}
\partial_{\xi}G_{\pm}(x,\xi) = 4a_{\pm}(x)\begin{pmatrix} \xi_1 \\ \xi_2 \\ 0 \end{pmatrix} = 4a_{\pm}(x)\xi.
\end{equation}
This shows that $\partial^2_{\xi_i\xi_j}G_{\pm}(x,\xi) = 0$ for $i\ne j$ and $\xi_3 = 0$, i.e. the Hessian $\partial^2_{\xi}G_{\pm}$ is diagonal with respect to the orthogonal coordinates at $\xi$ with $\xi_3 = 0$. In addition, we have $\omega_3 = 0\iff\xi_3(\omega) = 0$, in which case we have we have $\xi(\omega) = \frac{\omega}{4a_{\pm}}$. In other words, $\omega$ is annihilated by $\overline{\xi}$ (=$dx_3$ at $x$) if and only if $\xi(\omega)\cdot\overline{\xi} = 0$, in which case $\xi(\omega)$ is a multiple of $\omega$. Furthermore, taking the $\xi_3$ derivative of $\partial_{\xi_3}G_{\pm}$ and evaluating at $\xi_3 = 0$ yields
\[\partial^2_{\xi_3\xi_3}G_{\pm}(x,\xi) = 2\left[(a_{33}(x)+a_{55}(x))\pm\left(a_{33}(x)-a_{55}(x) - \frac{4E^2(x)}{a_{11}(x)-a_{55}(x)}\right)\right].\]
Thus, in general the value $(\partial^2_{\xi}G_{\pm}(x,\xi))\cdot(\overline{\xi},\overline{\xi})$ will equal the above value for any $\xi$ orthogonal to $\overline{\xi}(x)$; note that the value is independent of $\xi$, as long as it is orthogonal to $\overline{\xi}(x)$. We thus let\footnote{$h_{\pm}$ standing for the Hessian of $G_{\pm}$} $h_{\pm}(x)$ denote this value. In other words,
\[h_{\pm}(x) =\left\{\begin{matrix} 4\left(a_{33}(x) - \frac{E^2(x)}{a_{11}(x)-a_{55}(x)}\right) &\text{ for }qP\,(+) \\ 4\left(a_{55}(x) + \frac{E^2(x)}{a_{11}(x)-a_{55}(x)}\right) &\text{ for }qSV\,(-)\end{matrix}\right..\]
Notice that if the elasticity is actually isotropic (i.e. $a_{11} = a_{33} = \lambda+2\mu$, $a_{55} = \mu$, $E^2 = 0$), then $h_+ = 4(\lambda+2\mu) = 4a_{11}$ and $h_{-} = 4\mu = 4a_{55}$, i.e. $h_{\pm} = 4a_{\pm}$.

In computing the subprincipal symbol, we will need to calculate several quantities related to these dynamics. The subprincipal symbol will end up only being relevant when $\zeta$ is a multiple of $\overline{\xi}(x)$, and in such cases we integrate over $\omega\in\mathbb{S}^2$ which are annihilated by $\overline{\xi}(x)$. Thus, for the \emph{rest of this section}, we assume that $\omega\in\overline{\xi}(x)^{\perp}\cap\mathbb{S}^2$, and all subsequent results in this section are valid for such $\omega$.

From \eqref{subsymgenfinal}, we see that we should calculate $\overline{\xi}\cdot\alpha(\omega)$, as well as $\partial_t\xi_T(0,\omega)$ and $\partial_{\omega_{\parallel}}\xi_T(0,\omega)$, with $\xi_T$ taking the role of $g$ in \eqref{subsymgenfinal}. We start with computing $\partial_t\xi_T$. Recall that
\[\xi_T(x,t,\omega) := \overline{\xi}(X(t,x,\xi(\omega)))\cdot\Xi(t,x,\xi(\omega)).\]
Thus we have
\begin{align*} &\partial_t\xi_T(x,0,\omega) \\
&= \partial_t|_{t=0}[\overline{\xi}(X(t,x,\xi(\omega)))]\cdot\Xi(0,x,\xi(\omega)) + \overline{\xi}(X(0,x,\xi(\omega)))\cdot\partial_t\Xi(0,x,\xi(\omega))  \\
&= [(\partial_tX(0,x,\xi(\omega))\cdot\partial_x)\overline{\xi}(x)]\cdot\xi(\omega) - \overline{\xi}(x)\cdot\partial_xG_{\pm}(x,\xi(\omega)) \\
&=[(\partial_{\xi}G_{\pm}(x,\xi(\omega)))\cdot\partial_x\overline{\xi}(x)]\cdot\xi(\omega) - \overline{\xi}(x)\cdot\partial_xG_{\pm}(x,\xi(\omega)) \\
&=\frac{[(\omega\cdot\partial_x)\overline{\xi}(x)]\cdot\omega}{4a_{\pm}(x)} - \overline{\xi}(x)\cdot\partial_xG_{\pm}(x,\xi(\omega)).
\end{align*}
The last line follows by noting that $\partial_{\xi}G_{\pm}(x,\xi(\omega)) = \omega$ by definition and $\xi(\omega) = \omega/(4a_{\pm}(x))$ for $\omega$ annihilated by $\overline{\xi}(x)$. To compute $\overline{\xi}(x)\cdot\partial_xG_{\pm}(x,\xi(\omega))$, we consider a path $(x(t),\xi(t))$ satisfying $x(0) = x$, $\dot{x}(0) = \overline{\xi}(x)$, $\xi(0) = \xi(\omega)$, and $\xi(t)\cdot\overline{\xi}(x(t)) = 0$. Differentiating the last equation and evaluating at $t=0$ yields
\[\dot\xi(0)\cdot\overline{\xi}(x) + \xi(0)\cdot(\overline{\xi}(x)\cdot\partial_x\overline{\xi}(x)) = 0;\]
notice that actually $\overline{\xi}(x)\cdot\partial_x\overline{\xi}(x) = 0$ since $\overline{\xi}$ has constant norm, and hence $\dot\xi(0)\cdot\overline{\xi}(x) = 0$. Moreover, since $\xi(t)$ is orthogonal to $\overline{\xi}(x(t))$ along this path, it follows from \eqref{wavespdorth} that
\[G_{\pm}(x(t),\xi(t)) = 2a_{\pm}(x(t))|\xi(t)|^2.\]
Differentiating this equation at $t=0$ yields
\[\overline{\xi}\cdot\partial_xG_{\pm} + \dot\xi(0)\cdot\partial_{\xi}G_{\pm} = 2\overline{\xi}\cdot\partial_xa_{\pm}|\xi|^2 + 4a_{\pm}\dot\xi(0)\cdot\xi;\]
the terms $\dot\xi(0)\cdot\partial_{\xi}G_{\pm}$ and $4a_{\pm}\dot\xi(0)\cdot\xi$ cancel since $\partial_{\xi}G_{\pm} = 4a_{\pm}\xi$ for $\xi$ orthogonal to $\overline{\xi}$. Thus we have
\[\overline{\xi}(x)\cdot\partial_xG_{\pm}(x,\xi) = 2(\overline{\xi}(x)\cdot\partial_x)a_{\pm}(x)|\xi|^2\]
when $\xi$ is orthogonal to $\overline{\xi}$. In particular,
\[\overline{\xi}(x)\cdot\partial_xG_{\pm}(x,\xi(\omega)) = 2(\overline{\xi}(x)\cdot\partial_x)a_{\pm}(x)|\xi(\omega)|^2 = \frac{\overline{\xi}(x)\cdot\partial_xa_{\pm}(x)}{8a_{\pm}(x)^2}\]
since $|\xi(\omega)|^2 = |\omega/(4a_{\pm}(x))|^2 = 1/(16a_{\pm}(x)^2)$ for $\omega\in\overline{\xi}(x)^{\perp}\cap\mathbb{S}^2$. Thus, we have
\begin{equation}
\partial_t\xi_T(x,0,\omega) = \frac{[(\omega\cdot\partial_x)\overline{\xi}(x)]\cdot\omega}{4a_{\pm}(x)} - \frac{\overline{\xi}(x)\cdot\partial_xa_{\pm}(x)}{8a_{\pm}(x)^2}. \label{partialtxit}
\end{equation}
Note that the term $(\omega\cdot\partial_x)\overline{\xi}(x)\cdot\omega$ is a curvature term: in fact, if we assume that $\overline{\xi}(x) = \frac{df}{|df|}$ for some $f$ so that $f$ labels the ``layers'' of the transverse isotropy, then this term is the second fundamental form of the layer (viewed as a surface in $\mathbb{R}^3$) applied to $(\omega,\omega)$.

We now consider $\partial_{\omega_{\parallel}}\xi_T$, recalling that we have $\zeta$ parallel to $\overline{\xi}$. Writing $\zeta = s\overline{\xi}$, we have $\partial_{\omega_{\parallel}} = \text{sgn}(s)\overline{\xi}\cdot\partial_{\omega}$. From the definition $\partial_{\xi}G_{\pm}(x,\xi(\omega)) = \omega$, taking a directional $\omega$ derivative in the direction of $\overline{\xi}$ gives
\[[(\overline{\xi}(x)\cdot\partial_{\omega}\xi)\cdot\partial_{\xi}]\partial_{\xi}G_{\pm}(x,\xi(\omega)) = \overline{\xi}\implies \partial^2_{\xi}G_{\pm}(x,\xi(\omega))\cdot((\overline{\xi}(x)\cdot\partial_{\omega}\xi),\overline{\xi}) = 1.\]
By the diagonalization of the Hessian $\partial^2_{\xi}G_{\pm}$, we have that
\begin{align*}\partial^2_{\xi}G_{\pm}(x,\xi(\omega))\cdot((\overline{\xi}(x)\cdot\partial_{\omega}\xi),\overline{\xi}) &= \partial^2_{\xi}G_{\pm}(x,\xi(\omega))\cdot(\overline{\xi},\overline{\xi})\cdot((\overline{\xi}(x)\cdot\partial_{\omega}\xi)\cdot\overline{\xi}) \\
&= h_{\pm}(x)[(\overline{\xi}\cdot\partial_{\omega})\xi\cdot\overline{\xi}].
\end{align*}
Since
\[(\overline{\xi}(x)\cdot\partial_{\omega})\xi(\omega)\cdot\overline{\xi}(x) = \overline{\xi}(x)\cdot\partial_{\omega}[\xi(\omega)\cdot\overline{\xi}(x)] = \overline{\xi}(x)\cdot\partial_{\omega}\xi_T(\omega)\]
it follows that $\overline{\xi}\cdot\partial_{\omega}\xi_T(\omega) = 1/h_{\pm}(x)$. Hence
\begin{equation}
\partial_{\omega_{\parallel}}\xi_T(\omega) = \frac{\text{sgn}(s)}{h_{\pm}(x)}\quad\text{for }\zeta = s\overline{\xi}(x). \label{partialparallel}
\end{equation}
We now compute the terms $\overline{\xi}\cdot[(\omega\cdot\partial_x)\partial_{\xi}G_{\pm}]$ and $\overline{\xi}\cdot(\partial_xG_{\pm}\cdot\partial^2_{\xi}G_{\pm})$. To calculate the first term, we proceed similarly as above, and we now consider a path $(x(t),\xi(t))$ with $x(0) = x$, $\dot{x}(0) = \omega$, $\xi(0) = \xi(\omega)$, and $\xi(t)\cdot\overline{\xi}(x(t)) = 0$. Taking the derivative of the last equation at $t=0$ yields
\[\dot\xi(0)\cdot\overline{\xi}(x) + (\omega\cdot\partial_x\overline{\xi})(x)\cdot\xi(\omega) = 0.\]
Furthermore, from \eqref{wavespdgradorth} we have
\[\partial_{\xi}G_{\pm}(x(t),\xi(t)) = 4a_{\pm}(x(t))\xi(t),\]
since along the path we have that $\xi(t)$ is orthogonal to $\overline{\xi}(x(t))$, and hence taking the derivative at $0$ yields
\[\omega\cdot\partial_x\partial_{\xi}G_{\pm} + \dot\xi(0)\cdot\partial_{\xi}\partial_{\xi}G_{\pm} = 4(\omega\cdot\partial_xa_{\pm})\xi + 4a_{\pm}\dot\xi(0).\]
Thus we have
\[\overline{\xi}(x)\cdot(\omega\cdot\partial_x\partial_{\xi}G_{\pm}) = -(\dot\xi(0)\cdot\partial_{\xi}\partial_{\xi}G_{\pm}\cdot\overline{\xi} - 4(\omega\cdot\partial_xa_{\pm})\xi\cdot\overline{\xi} - 4a_{\pm}\dot\xi(0)\cdot\overline{\xi}).\]
By the diagonalization of $\partial^2_{\xi}G_{\pm}$, we have 
\[\dot\xi(0)\cdot\partial_{\xi}\partial_{\xi}G_{\pm}\cdot\overline{\xi} = {h_{\pm}(x)(\dot\xi(0)\cdot\overline{\xi})}.\]
Substituting $\xi\cdot\overline{\xi} = 0$ and $\dot\xi(0)\cdot\overline{\xi} = -(\omega\cdot\partial_x\overline{\xi})\cdot\xi$ gives
\begin{align*} \overline{\xi}(x)\cdot(\omega\cdot\partial_x\partial_{\xi}G_{\pm}(x,\xi(\omega))) &= ((\omega\cdot\partial_x)\overline{\xi}(x)\cdot\xi(\omega))(h_{\pm}(x)-4a_{\pm}(x)) \\
&= [(\omega\cdot\partial_x)\overline{\xi}(x)\cdot\omega]\left(\frac{h_{\pm}(x)}{4a_{\pm}(x)}-1\right).
\end{align*}
(Note that this quantity vanishes in the case of isotropic elasticity.) For the term $\overline{\xi}\cdot(\partial_xG_{\pm}\cdot\partial^2_{\xi}G_{\pm})$, we note that the diagonalization of the Hessian $\partial^2_{\xi}G_{\pm}$ implies that
\begin{align*} \overline{\xi}(x)\cdot(\partial_xG_{\pm}\cdot\partial^2_{\xi}G_{\pm})(x,\xi(\omega)) &= h_{\pm}(x)(\overline{\xi}(x)\cdot\partial_x)G_{\pm}(x,\xi(\omega)) \\
&= \frac{h_{\pm}(x)(\overline{\xi}(x)\cdot\partial_x)a_{\pm}(x)}{8a_{\pm}(x)^2}.
\end{align*}
Thus, we have
\begin{align*} \overline{\xi}(x)\cdot\alpha(\omega) &= \frac{1}{2}\overline{\xi}(x)\cdot\left(\omega\cdot\partial_x\partial_{\xi}G_{\pm} - \partial_xG_{\pm}\cdot\partial^2_{\xi}G_{\pm})(x,\xi(\omega)\right) \\
&=\frac{1}{2}\left([(\omega\cdot\partial_x)\overline{\xi}(x)\cdot\omega]\left(\frac{h_{\pm}(x)}{4a_{\pm}(x)}-1\right)-\frac{h_{\pm}(x)(\overline{\xi}(x)\cdot\partial_x)a_{\pm}(x)}{8a_{\pm}(x)^2}\right).
\end{align*}
Combining the above calculations yields, for $\zeta = s\overline{\xi}(x)$,
\begin{equation}
\begin{aligned}
&\partial_t\xi_T(x,0,\omega)\partial_{\omega_{\parallel}}\xi_T(x,0,\omega) - \frac{\zeta}{|\zeta|}\cdot\alpha(\omega)(\partial_{\omega_{\parallel}}\xi_T(x,0,\omega))^2\\
&=\partial_{\omega_{\parallel}}\xi_T(x,\omega)\left(\partial_t\xi_T(x,0,\omega) - \text{sgn}(s)\overline{\xi}\cdot\alpha(\omega)\partial_{\parallel}\xi_T(x,\omega)\right) \\
&=\frac{\text{sgn}(s)}{h_{\pm}(x)}\left[\frac{[(\omega\cdot\partial_x)\overline{\xi}(x)]\cdot\omega}{4a_{\pm}(x)} - \frac{\overline{\xi}(x)\cdot\partial_xa_{\pm}(x)}{8a_{\pm}(x)^2} \right.\\
&\hphantom{=}\left.- \text{sgn}(s)\cdot\frac{1}{2}\left([(\omega\cdot\partial_x)\overline{\xi}(x)\cdot\omega]\left(\frac{h_{\pm}(x)}{4a_{\pm}(x)}-1\right)-\frac{h_{\pm}(x)(\overline{\xi}(x)\cdot\partial_x)a_{\pm}(x)}{8a_{\pm}(x)^2}\right)\cdot\frac{\text{sgn}(s)}{h_{\pm}(x)}\right]\\
&=\frac{\text{sgn}(s)}{h_{\pm}(x)}\left(\frac{[(\omega\cdot\partial_x)\overline{\xi}(x)]\cdot\omega}{4a_{\pm}(x)}\left(1 - \frac{1}{2} + \frac{1}{2}\cdot\frac{4a_{\pm}(x)}{h_{\pm}(x)}\right) - \frac{\overline{\xi}(x)\cdot\partial_xa_{\pm}(x)}{8a_{\pm}(x)^2}\left(1-\frac{1}{2}\right)\right)\\
&=\frac{\text{sgn}(s)}{2h_{\pm}(x)}\left(\frac{[(\omega\cdot\partial_x)\overline{\xi}(x)]\cdot\omega}{4a_{\pm}(x)}\left(1 + \frac{4a_{\pm}(x)}{h_{\pm}(x)}\right) - \frac{\overline{\xi}(x)\cdot\partial_xa_{\pm}(x)}{8a_{\pm}(x)^2}\right).
\end{aligned}
\label{subsymintegrand}
\end{equation}
Finally, we conclude the Hamiltonian dynamics section by showing that the travel time knowledge in fact determines the lens relation. This argument is a generalization of the argument first presented as Proposition 2.2 and Corollary 2.3 in \cite{michel}, now applied to any Hamiltonian system with a strictly convex Hamiltonian homogeneous of degree $2$:
\begin{lemma}
\label{dtau}
Consider Hamiltonian dynamics on a manifold $M$ with respect to a Hamiltonian $G(x,\xi)$ which is homogeneous of degree $2$ and strictly convex in the fiber variable, and fix $x_0\in M$. Let $U$ be a neighborhood of $x_0$ such that the Hamilton trajectories with base point starting at $x_0$ cover $U$. For $x\in U$, define
\begin{align*}\tau_{x_0}(x) &= \inf\left\{t>0\,:\,x = X(t)\text{ for some Hamilton trajectory }\right. \\
&\left.(X(t),\Xi(t))\text{ with }X(0) = x_0\text{ and }G(x_0,\Xi(0)) = \frac{1}{2}\right\}.
\end{align*}
Suppose $x_1\in U$ has the property that
\begin{equation}
\begin{aligned}&\text{for every }x\text{ in a neighborhood of }x_1\text{, there exists a unique }\xi\text{ such}\\
&\text{that }G(x_0,\xi) = 1/2\text{ and }x = X(\tau_{x_0}(x))\text{ with }(X(0),\Xi(0)) = (x_0,\xi).
\end{aligned}\label{utraj}
\end{equation}
Then $\tau_{x_0}$ is differentiable at $x_1$, and if $(X,\Xi)$ satisfies $(X(0),\Xi(0)) = (x_0,\xi_0)$ with $G(x_0,\xi_0) = 1/2$ and $x_1 = X(\tau_{x_0}(x_1))$, then
\[\Xi(\tau(x_1))\cdot dx = d\tau_{x_0}|_{x_1}.\]
\end{lemma}
Notice that the function $\tau_{x_0}$ is just the travel time from the point $x_0$ on the level set $\{G=1/2\}$ (this normalization is chosen for consistency with geodesic flow in the case that $G$ is a dual metric.) 

Assuming this lemma for now, consider $\Omega\subset M$ an open subset whose boundary is strictly convex with respect to the Hamilton flow of $G$, i.e if $\gamma$ is a Hamilton trajectory with $\gamma(0)\in\Omega$ and $\gamma(t)\in\partial\Omega$, then $\gamma'(t)\in T_{\gamma(t)}M\backslash T_{\gamma(t)}\partial\Omega$, and in fact must point outwards away from $\Omega$. Suppose $x_0\in\overline\Omega$ and every point in $\partial\Omega\backslash\{x_0\}$ satisfies property \eqref{utraj}. For $x_1\in\partial\Omega\backslash\{x_0\}$ and $\xi\cdot dx\in T_{x_1}^*M$, we have that $\xi\cdot dx = d\tau_{x_0}|_{x_1}$ if and only if the following three properties hold:
\begin{enumerate}
\item $\xi\cdot dx|_{T_{x_1}\partial\Omega} = d\tau_{x_0}|_{T_{x_1}\partial\Omega}$.
\item $G(x_1,\xi) = 1/2$.
\item If $\xi^{\perp}\cdot dx$ is an outward conormal to $\partial\Omega$ at $x_1$, then ${\xi^{\perp}\cdot (\partial_{\xi}G(x_1,\xi))} > 0$ (i.e. the corresponding vector $\partial_{\xi}G(x_1,\xi)$ is outward-pointing).
\end{enumerate}
The necessity is obvious. Conversely, if $\xi$ satisfies the first property, then $\xi\cdot dx$ is determined up to a multiple of the conormal to the boundary (i.e. there is a certain line $\xi$ must lie on), while the second property further reduces the possibilities for $\xi$ to at most two points since $G$ is strictly convex. If there are two possibilities, say $\xi_+$ and $\xi_-$ with $\xi_+\cdot dx$ differing from $\xi_-\cdot dx$ by a positive multiple of an outward conormal $\xi^{\perp}\cdot dx$, then in fact we will have $\pm\xi^{\perp}\cdot(\partial_{\xi}G(x_1,\xi_{\pm}))>0$, i.e. the two possibilities correspond to inward/outward pointing vectors (so that $d\tau$ is then uniquely specified as the covector corresponding to the outward pointing vector). Indeed, the function $g(s) = {G(x_1,\xi_-+s(\xi_+-\xi_-))}$ is strictly convex with $g(0) = g(1)$, and hence $g'(0)<0$ while $g'(1)>0$; the two derivatives are precisely $(\xi_+-\xi_-)\cdot(\partial_{\xi}G(x_1,\xi_{\pm}))$, which shows the claim by noting that $(\xi_+-\xi_-)\cdot dx$ is a positive multiple of $\xi^{\perp}\cdot dx$.

The benefit of these three properties is that they can be checked using just the knowledge of the travel times between points on the boundary, as well as the Hamilton $G$ restricted to the boundary, so as an immediate consequence we have:
\begin{corollary}
\label{exit}
Suppose $\Omega\subset M$ has strictly convex boundary, and for every $x_0\in\partial\Omega$ we have that every point in $\partial\Omega\backslash\{x_0\}$ satisfies property \eqref{utraj}. Then for any distinct pair of points $x_0,x_1\in\partial\Omega$ the exiting covector on the Hamilton trajectory connecting $x_0$ and $x_1$ is determined by the knowledge of the Hamiltonian $G$ on the boundary $\partial\Omega$ and the travel time function $\tau$.
\end{corollary}
Since our Hamiltonian $G$ is even in the fiber variable, it follows that all trajectories are reversible, and hence the starting and ending covector for any trajectory connecting two points on the boundary is determined by the travel time function (in particular there is a unique trajectory for every pair of points). So in fact the travel time data also determines if there are any trapped trajectories; assuming there are none, it follows that the travel time data determines the lens relation data. Thus we are free to study the lens rigidity problem.

It thus suffices to prove Lemma \ref{dtau}.

\begin{proof}
For $x$ in a neighborhood of $x_0$ and $t>0$, define the action $S_{x_0}(x,t)$ as
\[S_{x_0}(x,t) = \inf_{\substack{\gamma(0) = x_0 \\ \gamma(t) = x}}\int_0^t{L(\gamma(s),\dot{\gamma}(s))\,ds}\]
where $L$ is the Lagrangian associated to $G$, i.e.
\[L(x,v) = \inf_{\xi}{[\xi\cdot v - G(x,\xi)]}.\]
Note that by strict convexity the infimum in the right-hand side is indeed attained, and furthermore it is attained at $\xi$ satisfying $v = \partial_{\xi}G(x,\xi)$, in which case
\[L(x,v) = \xi\cdot\partial_{\xi}G(x,\xi) - G(x,\xi) = G(x,\xi),\]
using that $G$ is homogeneous of degree $2$. Furthermore, the least action principle gives that, for fixed $t$, the curve $\gamma$ minimizing the integral in the definition of $S$ is a projection of a Hamilton trajectory. If $(X(s),\Xi(s))$ is a Hamilton trajectory with $X(0) = x_0$ and $X(\tau) = x$ for $\tau = \tau_{x_0}(x)$, then $(X_t(s),\Xi_t(s)) = \left(X\left(\frac{\tau}{t}s\right),\frac{\tau}{t}\Xi\left(\frac{\tau}{t}s\right)\right)$ is also a Hamilton trajectory, now with the property that $X_t(0) = x_0$ and $X_t(t) = x$. Since $G(X(s),\Xi(s)) = 1/2$ for all $0\le s\le\tau$, it follows that $G(X_t(s),\Xi_t(s)) = \frac{\tau^2}{2t^2}$ by homogeneity. It follows that
\[ S_{x_0}(x,t) = \int_0^t{L(X_t(s),\dot{X}_t(s))\,ds} =\int_0^t{G(X_t(s),\Xi_t(s))\,ds}= t\cdot \frac{\tau^2}{2t^2} = \frac{1}{2}\frac{\tau^2}{t}.\]
Differentiating the above equation thus gives
\[dS_{x_0}|_{(x,t)} = \frac{\tau}{t}\,d\tau|_x - \frac{1}{2}\frac{\tau^2}{t^2}\,dt|_t.\]
On the other hand, we also have (cf. \cite{arnold} Section 46C)
\[dS_{x_0}|_{(x,t)} = \xi\cdot dx|_x - G\,dt|_t\]
where $\xi\cdot dx = \Xi_t(t)\cdot dx$ is the corresponding exiting covector. Equating the coefficients at $t = \tau$ thus gives
\[d\tau|_x = \Xi_{\tau}(\tau)\cdot dx = \Xi(\tau(x))\cdot dx,\]
as desired. (Note that $G = \frac{\tau^2}{2t^2}$ at the exiting covector, so the coefficients of $dt$ also match, as expected.)
\end{proof}

\subsection{Computing the symbols of the operators $N^{\nu}_{\pm}$ and $\tilde{N}^{\nu}_{\pm}$}
\label{excomp}

We now apply the calculations of Sections \ref{matrixweight} and \ref{hamdyn} to our situation. In the notation of Section \ref{matrixweight}, and recalling formulas \eqref{inu} and \eqref{adjoint}, for the operators $I^{\nu}_{\pm}$, we have 
\[A(x,\xi) = -E^{\nu}(x,\xi)\frac{\partial\tilde\Xi}{\partial\xi}(\tau(x,\xi),(x,\xi)),\]
and for the generalized adjoint we have 
\[B(x,\omega) = \chi(x,\omega)\left(-\frac{\partial\tilde\Xi}{\partial\xi}(\tau(x,\xi(\omega)),(x,\xi(\omega)))\right)^{-1},\]
so
\begin{align*}
C(x,t,\omega) &= \chi(x,\omega)E^{\nu}(X(t),\Xi(t))\cdot \left(\frac{\partial\tilde\Xi}{\partial\xi}(\tau(x,\xi(\omega)),(x,\xi(\omega)))\right)^{-1}\\
&\cdot\frac{\partial\tilde\Xi}{\partial\xi}(\tau(X(t),\Xi(t)),(X(t),\Xi(t)))
\end{align*}
with $(X(t),\Xi(t)) = (X(t,x,\xi(\omega)),\Xi(t,x,\xi(\omega)))$. In particular,
\[C(x,0,\omega) = \chi(x,\omega)E^{\nu}(x,\xi(\omega))\]
(note that this is scalar-valued), so $\sigma_{-1}(N^{\nu})(x,\zeta) = a_{-1}|\zeta|^{-1}$, with
\begin{equation}
a_{-1}(x,\zeta) = 2\pi\int_{\zeta^{\perp}\cap\mathbb{S}^2}{\chi(x,\omega)E^{\nu}(x,\xi(\omega))\,d\mathbb{S}^1(\omega)}.\label{psymb}
\end{equation}
Thus the principal symbol is scalar-valued.

Furthermore, by computing
\begin{align*} \frac{\partial G_{\pm}}{\partial a_{11}} &= \xi_I^2\left(1\pm\frac{(a_{11}-a_{55})\xi_I^2+(a_{33}-a_{55})\xi_T^2}{\sqrt{((a_{11}-a_{55})\xi_I^2+(a_{33}-a_{55})\xi_T^2)^2-4E^2\xi_I^2\xi_T^2}}\right) \\
\frac{\partial G_{\pm}}{\partial a_{33}} &= \xi_T^2\left(1\pm\frac{(a_{11}-a_{55})\xi_I^2+(a_{33}-a_{55})\xi_T^2}{\sqrt{((a_{11}-a_{55})\xi_I^2+(a_{33}-a_{55})\xi_T^2)^2-4E^2\xi_I^2\xi_T^2}}\right)\\
\frac{\partial G_{\pm}}{\partial E^2} &= \xi_T^2\left(\frac{\mp 2\xi_I^2}{\sqrt{((a_{11}-a_{55})\xi_I^2+(a_{33}-a_{55})\xi_T^2)^2-4E^2\xi_I^2\xi_T^2}}\right)
\end{align*}
and noting that we can write
\begin{align*} \frac{\partial G_-}{\partial a_{11}} &= -\frac{4E^2\xi_I^4\xi_T^2}{\sqrt{A^2-B}(\sqrt{A^2-B}+A)}\\
\text{and}\quad\frac{\partial G_-}{\partial a_{33}} &= -\frac{4E^2\xi_I^2\xi_T^4}{\sqrt{A^2-B}(\sqrt{A^2-B}+A)}
\end{align*}
where $A = {(a_{11}-a_{55})\xi_I^2+(a_{33}-a_{55})\xi_T^2}$ and $B = 4E^2\xi_I^2\xi_T^2$, from the algebraic identity $\frac{A}{\sqrt{A^2-B}} = 1 + \frac{B}{\sqrt{A^2-B}(\sqrt{A^2-B}+A)}$,
we see the following:
\begin{itemize}
\item $\frac{\partial G_+}{\partial a_{11}}$ is a positive smooth multiple of $\xi_I^2$. Since $E^{11}_+$ is obtained by integrating $\frac{\partial G_+}{\partial a_{11}}$ over a range of parameter values, it follows that $E^{11}_+$ is also a positive smooth multiple of $\xi_I^2$.
\item $\frac{\partial G_+}{\partial a_{33}}$, $\frac{\partial G_+}{\partial E^2}$, $\frac{\partial G_-}{\partial a_{11}}$, $\frac{\partial G_-}{\partial a_{33}}$, and $\frac{\partial G_-}{\partial E^2}$, are all smooth multiples of $\xi_T^2$, and this multiple is everywhere nonnegative (resp. nonpositive, nonpositive, nonpositive, nonnegative). In other words, for these cases we can write
\[\frac{\partial G_{\pm}}{\partial\nu}(x,\xi) = g^{\nu}_{\pm}(x,\xi)\xi_T^2.\]
Thus, the same is true for $E^{33}_+$, $E^{E^2}_+$, $E^{11}_-$, $E^{33}_-$, and $E^{E^2}_-$, since we can write $E^{\nu}_{\pm}(x,\xi) = f^{\nu}_{\pm}(x,\xi)\xi_T^2$, where
\[f^{\nu}_{\pm}(x,\xi) = \int_0^1{g^{\nu}_{\pm}(a_{11}+sr_{11},a_{33}+sr_{33},E^2+sr_{E^2};x,\xi)\,ds}.\]
Moreover, the $f^{\nu}$ are smooth, and $f^{33}_+$ and $f^{E^2}_-$ are nonnegative while $f^{E^2}_+$, $f^{11}_-$, and $f^{33}_-$ are nonpositive. Moreover, $f^{33}_+$ is everywhere positive, while $f^{E^2}_{\pm}$ is a negative (resp. positive) multiple of $\xi_I^2$ and is thus nonzero everywhere except when $\xi_I = 0$. Since $g^{11}_-$ is a negative multiple of $E^2\xi_I^4$, it follows that $f^{11}_-$ is also a negative multiple of $\xi_I^4$ and is thus nonzero away from $\xi_I=0$, \emph{provided} that we assume $E^2$ is nonzero either in the background elasticity or the perturbed elasticity, although it can be a very small multiple if we assume instead that $E^2$ is known to be small. Finally, since $g^{33}_{-}$ is a multiple of $E^2\xi_I^2\xi_T^2$, it follows that $f^{33}_-$ will vanish when $\xi_I = 0$ or $\xi_T = 0$, and like $f^{11}_-$ it can be very small if $E^2$ is assumed to be small.
\end{itemize}

For the $qP$ wave speed and $\nu = a_{11}$, we have $E^{11}_+(x,\xi(\omega))>0$ unless $\xi_I(\omega) = 0$. In particular, if we choose $\chi$ to be identically one in a neighborhood of the equatorial sphere $\{\xi_T(\omega) = 0\}$ (where $|\xi_I|$ is bounded away from zero), then the integral over any $S^1$ circle will contain points where $E^{11}_+$ is positive. Hence for such $\chi$ we recover the fact that $N_+^{11}$ is elliptic. 

Now for $\nu$ for which we can write $E^{\nu}_{\pm} = f^{\nu}_{\pm}\xi_T^2$ we have that
\[a_{-1} = 2\pi\int_{\zeta^{\perp}\cap\mathbb{S}^2}{\chi(x,\omega)f^{\nu}_{\pm}(x,\xi(\omega))\xi_T^2(\omega)\,d\mathbb{S}^1(\omega)}.\]
Since $\xi_T^2$ is nonnegative and $f^{\nu}_{\pm}$ is nonnegative/nonpositive, it follows that $a_{-1}$ is a nonnegative/nonpositive scalar multiple of the identity. Moreover, in order for $a_{-1}$ to vanish, we must have $\xi_T(0,(0,\omega')) = 0$ for all $\omega'$, i.e. $\xi_T(\omega) = 0$ for all $\omega$ perpendicular to $\zeta$. This happens precisely when $\zeta$ is a multiple of $\overline{\xi}(x)$. Moreover, $a_{-1}$ will actually vanish quadratically on $\Sigma$ due to nonnegativity, and as long as $\chi(x,\omega)f^{\nu}_{\pm}(x,\xi(\omega))$ does not vanish on the equatorial sphere $\overline{\xi}(x)^{\perp}\cap\mathbb{S}^2$, the quadratic vanishing is nondegenerate (essentially due to the fact that the quadratic vanishing of $\xi_T^2$ is nondegenerate; cf. Lemma 3.5 of \cite{ti}). For $\chi\equiv 1$ near the equatorial sphere, this will be the case for all $f^{\nu}_{\pm}$ except $f^{33}_-$. Moreover, since $f^{33}_-$ is nonpositive and also vanishes on the equatorial sphere $\overline{\xi}(x)^{\perp}\cap\mathbb{S}^2$, it will in fact vanish quadratically, and so overall $E^{33}_-$ will vanish quartically on the equatorial sphere. This implies that the principal symbol of $N^{33}_-$ will actually vanish quartically on $\Sigma$ as well. 

We now analyze the behavior of the subprincipal term when $\zeta$ is a multiple of $\overline{\xi}(x)$. 
From Section \ref{matrixweight}, specifically \eqref{subsymgenfinal}, we have that the subprincipal term is $a_{-2}|\zeta|^{-2}$, with
\begin{align*}
a_{-2} &= 2\pi i\int_{\overline{\xi}^{\perp}\cap\mathbb{S}^2}\chi(x,\omega)f^{\nu}_{\pm}(x,\xi(\omega))\left[\partial_t\xi_T(x,0,\omega')\partial_{\omega_{\parallel}}\xi_T(x,0,\omega) \right.\\
&\phantom{=2\pi i}\left.- \frac{\zeta}{|\zeta|}\cdot\alpha(\omega)(\partial_{\omega_{\parallel}}\xi_T(x,0,\omega))^2\right]\,d\mathbb{S}^1(\omega).
\end{align*}
From \eqref{subsymintegrand}, we thus have
\begin{equation}
\begin{aligned}
a_{-2}(x,s\overline{\xi}(x)) &= \text{sgn}(s)\frac{2\pi i}{h_{\pm}(x)}\int_{\overline{\xi}^{\perp}\cap\mathbb{S}^2}\left[\chi(x,\omega)f^{\nu}_{\pm}(x,\xi(\omega))\right.\\
&\left.\left(\frac{[(\omega\cdot\partial_x)\overline{\xi}(x)]\cdot\omega}{4a_{\pm}(x)}\left(1 + \frac{4a_{\pm}(x)}{h_{\pm}(x)}\right) - \frac{\overline{\xi}(x)\cdot\partial_xa_{\pm}(x)}{8a_{\pm}(x)^2}\right)\right]\,d\mathbb{S}^1(\omega).
\end{aligned}
\label{subsymb}
\end{equation}
We make the following remarks\footnote{The author wishes to thank Maarten de Hoop for helpful discussions regarding these remarks.}:
\begin{itemize}
\item In the case of isotropic elasticity we have that the ratio $\left(1+\frac{4a_{\pm}(x)}{h_{\pm}(x)}\right)$ equals $2$ since $h_{\pm} = 4a_{\pm}$; in any case the ratio is positive as it is greater than $1$.

\item For the Earth, if we take the axis of isotropy $\overline{\xi}(x)$ to point roughly out of the earth, then the curvature term $(\omega\cdot\partial_x\overline{\xi}(x))\cdot\omega$ will be positive if the layers curve inwardly and negative if the layers curve outwardly. On a macroscopic scale the layers represent varying depths of the interior of the Earth and hence are roughly spherical, so this term would be positive.

\item Similarly, again taking $\overline{\xi}(x)$ to point out of the earth, the parameter gradient term $-\overline{\xi}(x)\cdot\partial_xa_{\pm}(x)$ will be positive if the material parameter $a_{11}$ increases with depth (since the axis points away from deeper regions) and negative if it decreases. It is geologically reasonable to assume that the material parameters increase with depth, and hence this term would also be positive.

\item Reversing the axis of isotropy will make both of the above terms negative, but in either case the signs agree.
\end{itemize}
Thus, under appropriate assumptions, the factor $\frac{[(\omega\cdot\partial_x)\overline{\xi}(x)]\cdot\omega}{4a_{\pm}(x)}\left(1 + \frac{4a_{\pm}(x)}{h_{\pm}(x)}\right) - \frac{\overline{\xi}(x)\cdot\partial_xa_{\pm}(x)}{8a_{\pm}(x)^2}$ will have a definite (nonzero) sign over all $\omega\in\overline{\xi}^{\perp}\cap\mathbb{S}^2$. For the $f^{\nu}$ which do not vanish on the equatorial sphere, it follows that if we take $\chi\equiv 1$ near the equatorial sphere (in which case $\chi$ drops out from the formula since we are integrating on the equatorial sphere), then the integrand in \eqref{subsymb} will always be nonnegative/nonpositive sign, and since $f^{\nu}$ is nonzero somewhere, it follows that the entire integral will be nonzero. For those $\nu$, it follows that the corresponding operator has a nonvanishing subprincipal term.

We now analyze the operators $\tilde{N}^{\nu}_{\pm}$. Recall from \eqref{itildenu} that the terms $\partial_xE^{\nu}$ and $\partial_{\xi}E^{\nu}$ appear in the matrix weight defining $\tilde{I}^{\nu}_{\pm}$. For the wave speeds and parameters such that $E^{\nu} = f^{\nu}\xi_T^2$ (i.e. all except the $qP$ speed for $a_{11}$), we have that $\partial_xE^{\nu}$ and $\partial_{\xi}E^{\nu}$ can be written as smooth multiples of $\xi_T$. Thus in these cases we have
\[C(x,t,\omega) = F^{\nu}_{\pm}(x,t,\omega)\xi_T(x,t,\omega)\]
for some smooth (matrix-valued) $F^{\nu}$. In such cases, we have that the principal symbol
\[\sigma_{-1}(\tilde{N}^{\nu}_{\pm})(x,\zeta) = |\zeta|^{-1}\cdot 2\pi\int_{\zeta^{\perp}\cap\mathbb{S}^2}{F^{\nu}_{\pm}(x,0,\omega)\xi_T(\omega;x)\,d\mathbb{S}^1(\omega)}\]
vanishes on $\Sigma$ since $\xi_T(\omega) = 0$ for all $\omega$ annihilated by $\overline{\xi}$. In other words, if the principal symbol of $N^{\nu}$ vanishes on $\Sigma$, then so does the principal symbol of $\tilde{N}^{\nu}$ (although \emph{a priori} we cannot say it vanishes quadratically). Since these operators are associated to an ``error'' term that will be controlled by a Poincar\'e inequality argument, we will not investigate further properties of these operators, beyond that they (aside from $\tilde{N}^{11}_+$) have vanishing principal symbol on $\Sigma$.
\begin{proof}[This thus proves Theorem \ref{symbolthm}]\end{proof}

\begin{remark}
We can in fact explicitly calculate $f^{\nu}(x,\xi(\omega))$ for $\omega$ annihilated by $\overline{\xi}$: indeed, for those $\omega$ we have that $\xi(\omega)$ is also orthogonal to $\overline{\xi}$, i.e. $\xi_T(\omega) = 0$, and since
\[g^{11}_-(x,\xi)|_{\xi_T=0} = \frac{-4E^2\xi_I^4}{\sqrt{A^2-B}(\sqrt{A^2-B}+A)}\Big|_{\xi_T=0}=-\frac{2E^2(x)}{(a_{11}(x)-a_{55}(x))^2}\]
where $A = (a_{11}-a_{55})\xi_I^2+(a_{33}-a_{55})\xi_T^2, B = 4E^2\xi_I^2\xi_T^2$,
\begin{align*}
g^{33}_{\pm}(x,\xi)|_{\xi_T=0} &= \left(1\pm\frac{(a_{11}-a_{55})\xi_I^2+(a_{33}-a_{55})\xi_T^2}{\sqrt{((a_{11}-a_{55})\xi_I^2+(a_{33}-a_{55})\xi_T^2)^2-4E^2\xi_I^2\xi_T^2}}\right)\Big|_{\xi_T=0} \\
&= 1\pm 1 \\
g^{E^2}_{\pm}(x,\xi)|_{\xi_T=0} &= \left(\frac{\mp 2\xi_I^2}{\sqrt{((a_{11}-a_{55})\xi_I^2+(a_{33}-a_{55})\xi_T^2)^2-4E^2\xi_I^2\xi_T^2}}\right)\Big|_{\xi_T = 0} \\
&= \mp\frac{2}{a_{11}(x)-a_{55}(x)}
\end{align*}
(note that the expressions above do not depend on $\xi$ so long as it is orthogonal to $\overline{\xi}$), it follows that for $\omega$ annihilated by $\overline{\xi}$ we can write $f^{\nu}_{\pm}(x,\xi(\omega)) = f^{\nu}_{\pm}(x)$, with
\begin{equation}
\begin{aligned}
f^{11}_-(x) &= \int_0^1{-\frac{2(E^2(x)+sr_{E^2}(x))}{(a_{11}(x)+sr_{11}(x)-a_{55}(x))^2}\,ds} \\
f^{33}_{\pm}(x) &= 1\pm 1\\
f^{E^2}_{\pm}(x) &= \int_0^1{\mp\frac{2}{a_{11}(x)+sr_{11}(x)-a_{55}(x)}\,ds}.
\end{aligned}
\label{f(x)}
\end{equation}
Thus $f^{\nu}_{\pm}(x,\xi(\omega)) = f^{\nu}_{\pm}(x)$ can be factored out of the integral in \eqref{subsymb}. In particular, if $\chi\equiv 1$ near the equatorial sphere, then the integral in \eqref{subsymb} can be explicitly evaluated to yield
\[a_{-2}(x,s\overline{\xi}(x)) = \text{sgn}(s)i\cdot\frac{\pi^2 f^{\nu}_{\pm}(x)}{h_{\pm}(x)}\left(\frac{H(x)}{a_{\pm}(x)}\left(1+\frac{4a_{\pm}(x)}{h_{\pm}(x)}\right)-\frac{\overline{\xi}(x)\cdot\partial_xa_{\pm}(x)}{2a_{\pm}(x)^2}\right)\]
where $H(x)$ is the mean curvature\footnote{This is obtained by using the fact that $\int_{\mathbb{S}^{n-1}}{A(\omega,\omega)\,d\omega}$ for a quadratic form $A$ is precisely the volume of $\mathbb{S}^{n-1}$ times the average of its eigenvalues.} of the layer at $x$.
\end{remark}

We now make a more quantitative estimate of the principal symbols, to be used in the inversion arguments. Note that the subprincipal behavior of the operators only depend on the behavior of the integrand near the equatorial sphere. Thus, let $\chi$ be a cutoff such that $\chi$ is identically $1$ in a neighborhood of the equatorial sphere $\{(x,\omega)\,:\,\xi_T(\omega;x) = 0\}$, and suppose it is supported in a region of the form $\{|\xi_T|<\epsilon|\xi|\}$. Note that on $\{|\xi_T|<\epsilon|\xi|\}$ we have $|\xi_I|^2 = |\xi|^2(1+O(\epsilon^2))$, and
\[\frac{(a_{11}-a_{55})\xi_I^2+(a_{33}-a_{55})\xi_T^2}{\sqrt{((a_{11}-a_{55})\xi_I^2+(a_{33}-a_{55})\xi_T^2)^2-4E^2\xi_I^2\xi_T^2}} = 1 + O(\epsilon^2)\]
and
\[\frac{\xi_I^2}{\sqrt{((a_{11}-a_{55})\xi_I^2+(a_{33}-a_{55})\xi_T^2)^2-4E^2\xi_I^2\xi_T^2}} = \frac{1}{a_{11}-a_{55}} + O(\epsilon^2)\]
where we can make the $O(\epsilon^2)$ estimate uniformly assuming \emph{a priori} uniform bounds on $a_{11}-a_{55}$ and $a_{33}-a_{55}$ (in particular from below by a positive constant), as well as on $E^2$. Thus, in the region where $|\xi_T|<\epsilon|\xi|$, we have
\begin{align*} \frac{\partial G_{\pm}}{\partial a_{11}} &= \xi_I^2\left(1\pm 1 + O(\epsilon^2)\right) \\
\frac{\partial G_{\pm}}{\partial a_{33}} &= \xi_T^2\left(1\pm 1 + O(\epsilon^2)\right)\\
\frac{\partial G_{\pm}}{\partial E^2} &= \xi_T^2\left(\frac{\mp 2}{a_{11}-a_{55}}+O(\epsilon^2)\right)
\end{align*}
This then implies that
\begin{align*} E^{11}_{\pm} &= \xi_I^2\left(1\pm 1 + O(\epsilon^2)\right) \\
E^{33}_{\pm} &= \xi_T^2\left(1\pm 1 + O(\epsilon^2)\right)\\
E^{E^2}_{\pm} &= \xi_T^2\left(\frac{\mp 2}{(a_{11}-a_{55})_l}+O(\epsilon^2)\right)
\end{align*}
where $(a_{11}-a_{55})_l$ is the logarithmic mean of $a_{11}-a_{55}$ and $\tilde{a}_{11} - a_{55}$ satisfying $\frac{1}{(a_{11}-a_{55})_l} = \int_0^1{\frac{1}{a_{11}+sr_{11}-a_{55}}\,ds}$. Plugging this into \eqref{psymb}, we see that if we let
\begin{equation}
a_{\pm,I/T}(x,\zeta) = \int_{\zeta^{\perp}\cap\mathbb{S}^2}{4\pi\chi(x,\xi_{\pm}(\omega))\xi_{I/T,\pm}^2(\omega)\,d\mathbb{S}^1(\omega)}\label{aint}
\end{equation}
then we have
\begin{equation}
\begin{aligned}
\sigma_{-1}(N^{11}_+) &= (1+O(\epsilon^2))a_{+,I},& \sigma_{-1}(N^{11}_-) &= O(\epsilon^2)a_{-,I},\\
\sigma_{-1}(N^{33}_+) &= (1+O(\epsilon^2))a_{+,T},& \sigma_{-1}(N^{33}_-) &= O(\epsilon^2)a_{-,T},\\
\sigma_{-1}(N^{E^2}_{\pm}) &= \left(\mp\frac{1}{(a_{11}-a_{55})_l}+O(\epsilon^2)\right)a_{\pm,T}.
\end{aligned} \label{cutoffest}
\end{equation}
Furthermore, we have that $a_{\pm,I}$ is everywhere positive, while $a_{\pm,T}$ is everywhere nonnegative and vanishes precisely on $\Sigma$, where the vanishing is nondegenerately quadratic.

We also consider the problem of when there is a functional relationship and calculate the relevant symbols in this situation. The heuristic in this case is the following rough idea: if for some parameter $\nu_0$ we know that $\nu_0 = f(\nu_1,\dots,\hat{\nu_0},\dots,\nu_n)$ (i.e. $\nu_0$ is a knkown function of the other parameters $\nu\ne\nu_0$), then we can write $\tilde{G} - G = \sum_{\nu\ne \nu_0}{E^{\nu}_{eff}r_{\nu}}$, where
\begin{align*} E^{\nu}_{eff}(x,\xi) &= \int_0^1{\frac{\partial}{\partial\nu}[G(\nu_1,\dots,f(\nu_1,\dots,\hat{\nu_0},\dots,\nu_n),\dots,\nu_n)](\nu + sr_{\nu};x,\xi)\,ds} \\
&= \int_0^1{\left(\frac{\partial G}{\partial\nu} + \frac{\partial G}{\partial\nu_0}\frac{\partial f}{\partial\nu}\right)(\nu+sr_{\nu};x,\xi)\,ds}
\end{align*}
The behavior of the associated operator $N^{\nu}_{eff,\pm}$ depends heavily on the behavior of the integrand $\frac{\partial G}{\partial\nu} + \frac{\partial G}{\partial\nu_0}\frac{\partial f}{\partial\nu}$ (note that for $N^{\nu}_{\pm}$, i.e. without the functional relationship, that this term is just $\frac{\partial G}{\partial\nu}$). By abuse of notation, we set
\[E^{\nu}_{\pm}(x,\xi) = \int_0^1{\frac{\partial G}{\partial\nu}(\nu + sr_{\nu};x,\xi)\,ds}\]
and $N^{\nu}_{\pm}$ denote the operator constructed with the above functions $E^{\nu}_{\pm}$ as if we were considering the non-functional case; the qualitative behavior of these objects is still the same as in the non-functional case.
For such cases, let
\[\tilde{f}_{\nu} = \int_0^1{\frac{\partial f}{\partial\nu}(\nu + sr_{\nu})\,ds}\]
(so this depends on $x$ via the parameters' values at $x$, but not $\xi$). Then $N^{\nu}_{eff,\pm}$ is the sum of $N^{\nu}_{\pm}$ times a smooth multiple of $N^{\nu_0}_{\pm}$, where this multiple is close to $\tilde{f}_{\nu}$.

Thus, suppose $a_{33} = f(a_{11},E^2)$. Then in essence we are adding a multiple times $N^{33}_{\pm}$ to the unmodified operators $N^{\nu}_{\pm}$ to obtain $N^{\nu}_{eff,\pm}$. Since $\sigma_{-1}(N^{33}_+)$ vanishes quadratically on $\Sigma$, it follows that $\sigma(N^{11}_{eff,+})$ is still elliptic near $\Sigma$, while since $\sigma\left(N^{33}_-\right)$ vanishes quartically near $\Sigma$, it follows that $\sigma(N^{E^2}_{eff,-})$ still has nondegenerately quadratically vanishing principal symbol near $\Sigma$, with the subprincipal behavior unchanged. Finally, the (at least quadratic) vanishing of $\sigma_{-}(N^{33}_{\pm})$ guarantees that $\sigma(N^{E^2}_{eff,+})$ and $\sigma(N^{11}_{eff,-})$ still have quadratically vanishing principal symbols. Thus in the effective matrix symbol $\begin{pmatrix} \sigma(N^{11}_{eff,+}) & \sigma(N^{E^2}_{eff,+}) \\ \sigma(N^{11}_{eff,-}) & \sigma(N^{E^2}_{eff,-})\end{pmatrix}$ we have that the qualitative behavior near $\Sigma$ of the diagonal terms are the same as in the non-functional problem, and that the qualitative off-diagonal behavior is also the same, aside from possible increased vanishing at $\Sigma$. Furthermore, away from $\Sigma$ we can estimate the terms by their principal symbols, and making the same approximations as above we have
\begin{equation}
\begin{aligned}
&\begin{pmatrix} \sigma_{-1}(N^{11}_{eff,+}) & \sigma_{-1}(N^{E^2}_{eff,+}) \\ \sigma_{-1}(N^{11}_{eff,-}) & \sigma_{-1}(N^{E^2}_{eff,-})\end{pmatrix} \\
&= \begin{pmatrix} (1 + O(\epsilon^2))(a_{+,I} + \tilde{f}_{11}a_{+,T}) & \left(\frac{-1}{(a_{11}-a_{55})_l}+\tilde{f}_{E^2}+O(\epsilon^2)\right)a_{+,T}\\ O(\epsilon^2)(a_{-,I}+a_{-,T})& \left(\frac{1}{(a_{11}-a_{55})_l}+O(\epsilon^2)\right)a_{-,T} \end{pmatrix}.
\end{aligned}
\label{a33func}
\end{equation}
Suppose instead that $E^2 = f(a_{11},a_{33})$. Then as before we have that $\sigma(N^{11}_{eff,+})$ is still elliptic near $\Sigma$ since $\sigma_{-1}(N^{E^2}_+)$ vanishes quadratically on $\Sigma$. Furthermore, $\sigma(N^{33}_{eff,+})$ and $\sigma(N^{11}_{eff,-})$ still have quadratically vanishing principal symbols. Finally, since $\sigma_{-1}(N^{E^2}_-)$ vanishes nondegenerately quadratically on $\Sigma$, it follows that if $\frac{\partial f}{\partial a_{33}}$ is always nonzero, then $\sigma(N^{33}_{eff,-})$ will have nondegenerately quadratically vanishing principal symbol (compared with quartic vanishing of $\sigma_{-1}(N^{33}_-)$ in the non-functional case), with nonvanishing subprincipal symbol as well. Away from $\Sigma$ we can estimate
\begin{equation}
\begin{aligned}
&\begin{pmatrix} \sigma_{-1}(N^{11}_{eff,+}) & \sigma_{-1}(N^{33}_{eff,+}) \\ \sigma_{-1}(N^{11}_{eff,-}) & \sigma_{-1}(N^{33}_{eff,-})\end{pmatrix} =\\
& \begin{pmatrix} (1 + O(\epsilon^2))a_{+,I} - \left(\frac{1}{(a_{11}-a_{55})_l}\tilde{f}_{11}+O(\epsilon^2)\right)a_{+,T} & \left(1 - \frac{1}{(a_{11}-a_{55})_l}\tilde{f}_{33}+O(\epsilon^2)\right)a_{+,T}\\ O(\epsilon^2)a_{-,I}+\left(\frac{1}{(a_{11}-a_{55})_l}\tilde{f}_{11}+O(\epsilon^2)\right) a_{-,T}& \left(\frac{1}{(a_{11}-a_{55})_l}\tilde{f}_{33}+O(\epsilon^2)\right)a_{-,T} \end{pmatrix}.
\end{aligned} \label{e2func}
\end{equation}
Finally, suppose $a_{11} = f(a_{33},E^2)$. In this case we add multiples of $N^{11}_{\pm}$, noting that $N^{11}_+$ is actually elliptic, and hence the argument must be made more carefully. We note, for example, that $\sigma(N^{33}_{eff,+})$ will be elliptic near $\Sigma$ if $\frac{\partial f}{\partial a_{33}}$ is bounded away from zero, and furthermore that since $\frac{\partial G_+}{\partial a_{11}}|_{\xi_T = 0} = 2|\xi|^2 = E^{11}_+|_{\xi_T=0}$ (independent of the parameter values), it follows that we have
\begin{align*} E^{33}_{eff,+}|_{\xi_T = 0} &= \int_0^1{\left(\frac{\partial G_+}{\partial a_{33}}|_{\xi_T = 0} + \frac{\partial G_+}{\partial a_{11}}|_{\xi_T = 0}\frac{\partial f}{\partial a_{33}}\right)(\nu+sr_{\nu})\,ds} \\
&= 2|\xi|^2\int_0^1{\frac{\partial f}{\partial a_{33}}(\nu+sr_{\nu})\,ds} = \tilde{f}_{33}E^{11}_+|_{\xi_T=0}
\end{align*}
and hence
\begin{equation}
\label{n33+1}
\sigma_{-1}(N^{33}_{eff,+})|_{\Sigma} = \tilde{f}_{33}\sigma_{-1}(N^{11}_+)|_{\Sigma}.
\end{equation}
Similarly, we have
\begin{equation}
\label{ne2+1}
\sigma_{-1}(N^{E^2}_{eff,+})|_{\Sigma} = \tilde{f}_{E^2}\sigma_{-1}(N^{11}_+)|_{\Sigma}.
\end{equation}
Furthermore, since $\sigma_{-1}(N^{11}_{-})$ still vanishes quadratically on $\Sigma$, it follows that $\sigma_{-1}(N^{33}_{eff,-})$ and $\sigma_{-1}(N^{E^2}_{eff,-})$ will also vanish quadratically on $\Sigma$. In the case where $\frac{\partial f}{\partial a_{33}}$ and $\frac{\partial f}{\partial E^2}$ are constant, the above arguments give that the subprincipal parts can be written as
\begin{equation}
\label{n33-2}
\sigma_{-2}(N^{33}_{eff,-})|_{\Sigma} = \tilde{f}_{33}\sigma_{-2}(N^{11}_-)|_{\Sigma}
\end{equation}
and
\begin{equation}
\label{ne2-2}
\sigma_{-2}(N^{E^2}_{eff,-})|_{\Sigma} = \sigma_{-2}(N^{E^2}_-)|_{\Sigma} + \tilde{f}_{E^2}\sigma_{-2}(N^{11}_-)|_{\Sigma}.
\end{equation}
In general the expressions will be the same except with the $\tilde{f}_{33}$ and $\tilde{f}_{E^2}$ prefactors replaced by a weighted average of the derivative values evaluated at $\nu+sr_{\nu}$ for $s\in(0,1)$ (so if the differences are known to be small then the expressions for the subprincipal symbol will not differ much from the above expressions). To estimate away from $\Sigma$, we rewrite
\[\frac{\partial G_-}{\partial a_{11}} = \xi_T^2\left(-\frac{2E^2}{a_{11}-a_{55}}+O(\epsilon^2)\right)\]
so that if $\left(\frac{E^2}{a_{11}-a_{55}}\right)_l = \int_0^1{\frac{E^2+sr_{E^2}}{a_{11}+sr_{11}-a_{55}}\,ds}$ then
\[\sigma\left(N^{11}_-\right) = \left(-\left(\frac{E^2}{a_{11}-a_{55}}\right)_l+O(\epsilon^2)\right)a_{-,T}.\]
Then
\begin{equation}
\begin{aligned}
&\begin{pmatrix} \sigma_{-1}(N^{33}_{eff,+}) & \sigma_{-1}(N^{E^2}_{eff,+}) \\ \sigma_{-1}(N^{33}_{eff,-}) & \sigma_{-1}(N^{E^2}_{eff,-})\end{pmatrix} \\
&= \begin{pmatrix} (1 + O(\epsilon^2))(a_{+,T} + \tilde{f}_{33}a_{+,I}) & \left(\frac{-1}{(a_{11}-a_{55})_l}+O(\epsilon^2)\right)a_{+,T}+\left(\tilde{f}_{E^2}+O(\epsilon^2)\right)a_{+,I}\\\left(-\left(\frac{E^2}{a_{11}-a_{55}}\right)_l\tilde{f}_{33}+O(\epsilon^2)\right)a_{-,T}& \left(\frac{1}{(a_{11}-a_{55})_l}-\left(\frac{E^2}{a_{11}-a_{55}}\right)_l\tilde{f}_{E^2}+O(\epsilon^2)\right)a_{-,T}\end{pmatrix}.
\end{aligned}\label{a11func}
\end{equation}
These more quantitative forms of the symbols will be used in Section \ref{recovsec}.

\subsection{Behavior of the operators as scattering operators}
\label{scat-op-sec}

We conclude by analyzing the behavior of the operators associated to matrix-weighted ray transforms viewed as operators in the \emph{Melrose scattering calculus}, as was done in \cite{ti}. We refer the reader to \cite{sslaes,convex,br1,br2,ti} for discussions regarding the properties of the scattering calculus and how to compute the symbol of a scattering operator. The purpose of this computation is to demonstrate that an \emph{additional} complication arises in attempting to follow the artificial boundary approach of \cite{ti} (which in turn follows the approach originally introduced in \cite{convex}), which justifies taking the alternative ``global'' approach in this paper.

We thus take $z=(x,y_1,y_2)$ as our coordinates, with $x$ denoting the boundary-defining function for our boundary $\{x=0\}$ and also strictly convex with respect to the relevant Hamiltonian dynamics, so that our manifold is now $X = \{(x,y)\in\mathbb{R}^3\,|\,x\ge 0\}$. In the formula for the formal adjoint $L$, we replace $\mathbb{S}^2$ with $\mathbb{R}_{\lambda}\times\mathbb{S}^1_{\omega}$, identifying the latter with a subset of the tangent bundle $T_zX$ by the identification $(\lambda,\omega)\mapsto\lambda\partial_x+\omega\cdot\partial_y$. To make the corresponding operator $N$ a scattering operator, we take our cutoff $B(z,\lambda,\omega)$ in the formula defining $L$ to be of the form $x^{-2}\chi_s(\lambda/x)\tilde{B}(z,\lambda,\omega)$ (in the notation of Section \ref{matrixweight}) where $\chi_s\in C_c^{\infty}(\mathbb{R})$, $\chi_s\ge 0$, and $\chi_s(0)>0$. We also conjugate by a factor of $e^{\digamma/x}$, which is equivalent to replacing the weight $A(Z(t,z_0,\xi_0),\Xi(t,z_0,\xi_0))$ by $e^{\digamma/X(t,z_0,\xi_0)}A(Z(t,z_0,\xi_0),\Xi(t,z_0,\xi_0))$ and the factor $\tilde{B}(z,\lambda,\omega)$ by $e^{-\digamma/x}\tilde{B}(z,\lambda,\omega)$. This means that $C(z,t,\lambda,\omega)$ will be of the form
\begin{align*} C(z,t,\lambda,\omega) &= x(z)^{-2}e^{\digamma\left(\frac{1}{X(t,z,\xi(\lambda,\omega))} - \frac{1}{x(z)}\right)}\chi_s\left(\frac{\lambda}{x(z)}\right)\\
&\cdot\chi(z,\lambda,\omega)F(z,t,\lambda,\omega)E^{\nu}(Z(t,z_0,\xi_0),\Xi(t,z_0,\xi_0))
\end{align*}
where $F$ is matrix-valued, but $F|_{t=0}$ is the identity matrix. As was shown in \cite{ti}, the operator associated to the $C$ above is a scattering operator of order $(-1,0)$, whose principal symbol is elliptic away from $\Sigma = \text{span }\overline{\xi}$ in the interior $\{x>0\}$ (though in the scattering cotangent bundle this means away from $(\xi,\eta)$ which satisfy $\xi\frac{dx}{x^2} + \eta\cdot\frac{dy}{x}\in\text{span }\frac{\overline{\xi}}{x}$, i.e. $(\xi,\eta)$ parallel to $(x\overline{\xi}_x,\overline{\xi}_y)$ if we write $\overline{\xi} = \overline{\xi}_x\,dx+\overline{\xi}_y\cdot dy$), as well\footnote{Elliptic in the sense of being nonzero; it does not satisfy the uniform elliptic estimate on $\{x=0\}$ as $(\xi,\eta)\rightarrow\infty$ since it vanishes on $\Sigma$ at fiber infinity.} as at finite points on the boundary $\{x=0\}$. Away from $\{x=0\}$ this can be shown by considering an oscillatory integral of the form
\[\int{e^{i(X(t,z,\lambda,\omega)-x,Y(t,z,\lambda,\omega)-y)\cdot\left(\frac{\xi}{x^2},\frac{\eta}{x}\right)}C(z,t,\lambda,\omega)\,dt\,d\lambda\,d\omega}\]
and analyzing the expression using stationary phase as $(\xi,\eta)\rightarrow\infty$. 
We now show the following:
\begin{proposition}
\label{scatsub}
The subprincipal symbol degenerates near the boundary as a power of $x$ relative to the principal symbol. Thus the arguments in the rest of this paper cannot be directly applied to the scattering situation in \cite{ti}.
\end{proposition}

We give a sketch of the calculation here. Note from the approximations
\begin{align*} (X(t,z,\lambda,\omega)-x,Y(t,z,\lambda,\omega)-y)
&= (\lambda t + \alpha_xt^2 + O(t^3), \omega t + \alpha_yt^2 + O(t^3)) \\
&= (x^2(\hat\lambda\hat{t}+\alpha_x\hat{t}^2 + O(x\hat{t}^3)),x(\omega\hat{t}+x\alpha_y\hat{t}^2+O(x^2\hat{t}^3)))
\end{align*}
(where $\lambda = x\hat{\lambda}$ and $t = x\hat{t}$) that it suffices to consider the oscillatory integral
\[\int{e^{i(\hat\lambda\hat{t}+\alpha_x\hat{t}^2,\omega\hat{t}+x\alpha_y\hat{t}^2)\cdot(\xi,\eta)}C(x,y,x\hat{t},x\hat{\lambda},\omega)x^2\,d\hat{t}\,d\hat{\lambda}\,d\omega}\]
(note that the $x^2$ from the change of variables cancels with the $x^{-2}$ factor in $C$). Since on the characteristic set we have that $(\xi,\eta)$ is parallel to $(x\overline{\xi}_x,\overline{\xi}_y)$, if we assume that the axis of isotropy does not coincide with $dx$ near the boundary, as was assumed in \cite{ti}, then we have $|\overline{\xi}_y|>\epsilon$ uniformly for some $\epsilon>0$, and hence for $(\xi,\eta)$ in the characteristic set, we have $\frac{\xi}{|\eta|} = x\frac{\overline{\xi}_x}{|\overline{\xi}_y|}\rightarrow 0$ as $x\rightarrow 0$. Thus we may take $\xi$ to be small compared to $\eta$. In that case, decompose $\omega = (\omega_{\parallel},\omega_{\perp})$ where $\omega_{\parallel}$ is parallel to $\eta$, i.e. write $\omega = \omega_{\parallel}\frac{\eta}{|\eta|} + \sqrt{1-\omega_{\parallel}^2}\omega_{\perp}$, $\omega_{\perp}\in\eta^{\perp}\cap\mathbb{S}^1$ (so the set of possible $\omega^{\perp}$ can be identified with $\mathbb{S}^0$, i.e. two points). Then $\eta\cdot\omega = |\eta|\omega_{\parallel}$, $d\omega = (1-\omega_{\parallel}^2)^{-1/2}$, and overall the phase becomes
\[|\eta|\left(\left(\omega_{\parallel} +\frac{\xi}{|\eta|}\hat{\lambda}\right)\hat{t} + \left(\frac{\xi}{|\eta|}\alpha_x+\frac{\eta}{|\eta|}\cdot x\alpha_y\right)\hat{t}^2\right).\]
Thus the integral becomes
\[\int_{\mathbb{R}\times\mathbb{S}^0}{\left(\int_{[-1,1]\times\mathbb{R}}{e^{i|\eta|\left(\left(\omega_{\parallel} +\frac{\xi}{|\eta|}\hat{\lambda}\right)\hat{t} + \left(\frac{\xi}{|\eta|}\alpha_x+\frac{\eta}{|\eta|}\cdot x\alpha_y\right)\hat{t}^2\right)}\tilde{C}(x,y,x\hat{t},x\hat{\lambda},\omega)\,d\mu_{\parallel}\,d\hat{t}}\right)\,d\hat{\lambda}\,d\omega_{\perp}}\]
where
\[\tilde{C}(x,y,\lambda,\omega,t)  = x^2C(x,y,\lambda,\omega,t)(1-\omega_{\parallel}^2)^{-1/2}.\]
The phase is then $Q(\hat{t},\omega_{\parallel}+\frac{\xi}{|\eta|}\hat{\lambda})/2$ where $Q$ is the quadratic form associated to the matrix
\[A = \begin{pmatrix} 2\left(\frac{\xi}{|\eta|}\alpha_x+\frac{\eta}{|\eta|}\cdot x\alpha_y\right) & 1 \\ 1 & 0 \end{pmatrix}\quad\left(\text{so }A^{-1} = \begin{pmatrix} 0 & 1 \\ 1 & -2\left(\frac{\xi}{|\eta|}\alpha_x+\frac{\eta}{|\eta|}\cdot x\alpha_y\right)\end{pmatrix}\right).\]
It follows that the stationary point is at $\{\omega_{\parallel} = -\frac{\xi}{|\eta|}\hat\lambda,\hat{t} = 0\}$ (note the first condition is equivalent to $\xi\hat{\lambda} + \omega\cdot\eta = 0$), and thus by stationary phase this is $a_{-1}|\eta|^{-1} + a_{-2}|\eta|^{-2} + O(|\eta|^{-3})$ where
\[ a_{-1} = 2\pi\int_{\mathbb{R}\times \mathbb{S}^0}{\tilde{C}(x,y,0,x\hat{\lambda},\omega)|_{\omega_{\parallel} = -\frac{\xi}{|\eta|}\hat\lambda}\,d\hat{\lambda}\,d\omega_{\perp}} = 2\pi\int_{\{\xi\hat\lambda + \eta\cdot\omega = 0\}}{\tilde{C}(x,y,0,x\hat{\lambda},\omega)\,d\mathbb{S}^0(\omega)\,d\hat{\lambda}}.
\]
and
\begin{align*} a_{-2} &= 2\pi i\int_{\mathbb{R}\times \mathbb{S}^0}\left(\partial_{\hat{t}}\partial_{\omega_{\parallel}} - \left(\frac{\xi}{|\eta|}\alpha_x+\frac{\eta}{|\eta|}\cdot x\alpha_y\right)\partial^2_{\omega_{\parallel}}\right)\left.\left(\tilde{C}(x,y,x\hat{t},x\hat{\lambda},\omega)\right)\middle|\right._{\omega_{\parallel} = -\frac{\xi}{|\eta|}\hat\lambda,\hat{t}=0}\,d\hat{\lambda}\,d\omega_{\perp}\\
&=2\pi i\int_{\{\xi\hat\lambda+\eta\cdot\omega = 0\}}\left(x\partial_t\partial_{\omega_{\parallel}} - \left(\frac{\xi}{|\eta|}\alpha_x+\frac{\eta}{|\eta|}\cdot x\alpha_y\right)\partial^2_{\omega_{\parallel}}\right)\left.\left(\tilde{C}(x,y,0,x\hat{\lambda},\omega)\right)\middle|\right._{\omega_{\parallel}=-\frac{\xi}{|\eta|}\hat\lambda}\,d\mathbb{S}^0(\omega)\,d\hat{\lambda}.
\end{align*}
Note that
\[
\tilde{C}(x,y,0,x\hat{\lambda},\omega) = \chi_s(\hat{\lambda})\chi(x,y,x\hat{\lambda},\omega)E^{\nu}(x,y,\xi(x\hat{\lambda},\omega))(1-\omega_{\parallel}^2)^{-1/2},
\]
from which it follows that $a_{-1}$ will not vanish away from the characteristic set $\{(\xi,\eta)\,|\,\xi\frac{dx}{x^2}+\eta\cdot\frac{dy}{x}\in\text{span }\frac{\overline{\xi}}{x}\}$ as $x\rightarrow 0$ (though it will otherwise be $O(1)$). On the other hand, for $(\xi,\eta)$ in the characteristic set, i.e. parallel to $(x\overline{\xi}_x,\overline{\xi}_y)$, we have that the subprincipal coefficient is
\begin{align*} a_{-2} &= 2\pi i\int_{\left\{\frac{\overline{\xi}}{x}\cdot(\lambda\partial_x+\omega\cdot\partial_y) = 0\right\}}\left(x\partial_t\partial_{\mu} - \left(\frac{x\overline{\xi}_x}{\left|\overline{\xi}_y\right|}\alpha_x+\frac{\overline{\xi}_y}{\left|\overline{\xi}_y\right|}\cdot x\alpha_y\right)\partial^2_{\mu}\right)\left.\left(\tilde{C}(x,y,0,x\hat{\lambda},\omega)\right)\middle|\right._{\omega_{\parallel}=-\frac{\xi}{|\eta|}\hat\lambda}\,d\mathbb{S}^0(\omega)\,d\hat{\lambda} \\
&= 2\pi i x\int_{\left\{\frac{\overline{\xi}}{x}\cdot(\lambda\partial_x+\omega\cdot\partial_y) = 0\right\}}\left(\partial_t\partial_{\mu} - \left(\frac{\overline{\xi}_x}{\left|\overline{\xi}_y\right|}\alpha_x+\frac{\overline{\xi}_y}{\left|\overline{\xi}_y\right|}\cdot \alpha_y\right)\partial^2_{\mu}\right)\left.\left(\tilde{C}(x,y,0,x\hat{\lambda},\omega)\right)\middle|\right._{\omega_{\parallel}=-\frac{\xi}{|\eta|}\hat\lambda}\,d\mathbb{S}^0(\omega)\,d\hat{\lambda}.
\end{align*}
Noting that the terms inside the integral are $O(1)$ as $x$ goes to $0$, it follows that $a_{-2}$ restricted to the characteristic set will vanish at a rate of $x$ as $x$ goes to $0$. It follows that the subprincipal symbol, while not vanishing away from $x=0$, will vanish at a rate of $x$ relative to the principal symbol as $x\rightarrow 0$.

\section{Symbols of Inverse Parabolic Type}
\label{invparsec}
For the operators $N^{11}_-$ (if $E^2>0$), $N^{33}_+$, and $N^{E^2}_{\pm}$, the scalar part of the symbol (modulo a factor of $|\zeta|^{-3}$) is thus ``parabolic'': it is second order elliptic except on a fiber-dimension 1 subset, where it has a nondegenerate purely imaginary order 1 subprincipal term. The prototypical example of such a symbol is $|\xi|^2+i\tau$ on $T^*(\mathbb{R}^{n-1}_x\times\mathbb{R}_t)$, the symbol of the heat operator $\partial_t-\Delta$. It is easy to show that the inverse of the heat symbol is a $(1/2,0)$ symbol of order $-1$, i.e. that $\frac{1}{|\xi|^2+i\tau}$ satisfies the estimates
\[\left|D^{\beta}_{(\xi,\tau)}D^{\alpha}_x\left(\frac{1}{|\xi|^2+i\tau}\right)\right|\le C|(\xi,\tau)|^{-1-|\beta|/2}.\]
Such symbolic estimates allows one to construct a parametrix for the heat operator which belongs to $\Psi^{-1}_{1/2,0}(\mathbb{R}^{n-1}_x\times\mathbb{R}_t)$, which in turn is one way to obtain standard parabolic regularity estimates. Boutet de Monvel \cite{bdm} generalized this idea by developing a symbol and pseudodifferential calculus to construct parametrices for certain hypoelliptic operators with double characteristics (i.e. the principal symbol vanishes to second order on the characteristic set), which contains the parametrix for the heat operator above. We will use this calculus to construct parametrices for our operators, which are of ``parabolic'' type.

We first review the calculus constructed by Boutet de Monvel; the full proofs of all statements in this section can be found in the original paper \cite{bdm}. Thus, consider a conic subset $\Sigma$ of $T^*M\backslash o$, say of codimension $\nu$, where $M$ is an $n$-dimensional manifold and $o$ is the zero section. Locally we can choose coordinates\footnote{The coordinate $p$ was called $x$ in \cite{bdm}; we will not use $x$ here in order to reserve its use for a base variable.} $(p,y,r) = (p_1,\dots,p_{\nu},y_1,\dots,y_{2n-1-\nu},r)$ where $p_i$, $y_i$ are homogeneous of degree $0$ and $r$ is homogeneous of degree $1$ such that $\Sigma = \{p_i = 0\}$. We then let
\[d_{\Sigma}^2 = |p|^2 + \frac{1}{r}.\]
Note that if different coordinates were chosen, then $d_{\Sigma}$ would change by a positive smooth multiple.

For example, if $\Sigma = \text{span }dx_n = \{\zeta' = 0\}\subset T^*\mathbb{R}^n$ where $\zeta' = (\zeta_1,\dots,\zeta_{n-1})$, then we can choose $p_i = \frac{\zeta_i}{|\zeta|}$ for $1\le i\le n-1$, $y_i = x_i$, and $r = |\zeta|$. In this case, we have
\[d_{\Sigma}^2 = \frac{|\zeta'|^2}{|\zeta|} + \frac{1}{|\zeta|} = \frac{|\zeta|'^2+|\zeta|}{|\zeta|^2}.\]
For the $\Sigma$ relevant in our problem, since it has an integrable kernel, it follows that every point in $\mathbb{R}^3$ admits local coordinates on a neighborhood where we can write $\Sigma$ in the above form, and hence we can take $d_{\Sigma}$ to be defined as above.

\begin{remark}
If we consider the fiber-compactified cotangent bundle $\overline{T^*M}$ and consider $\partial\Sigma = \overline{\Sigma}\cap\partial\overline{T^*M}\subset\overline{T^*M}$ (i.e. ``$\Sigma$ at fiber infinity''), then $d_{\Sigma}$ is a boundary-defining function for the front face of the parabolic blow-up of $\partial\Sigma$ in $\overline{T^*M}$. Indeed, the standard boundary-defining functions for $\partial\Sigma$ are given by $p = (p_1,\dots,p_{\nu})$ and $1/r$, so if we blow up $\partial\Sigma$ with respect to the coordinates $p$ and $1/r^{1/2}$, then $d_{\Sigma}^2 = |p|^2 + (1/r^{1/2})^2$, i.e. a boundary-defining function for the front face.
\end{remark}

Recall that a vector field $V$ on $T^*M\backslash o$ is \emph{homogeneous of degree $\nu$} if $\tau^*(Vf) = \tau^{\nu}V(\tau^*f)$ for all $f\in C^{\infty}(V)$ and $\tau\in\mathbb{R}_+$, where we identify elements of $\mathbb{R}_+$ with their dilation action on $T^*M$. Such vector fields can locally be written as $a(x,\xi)\cdot\partial_x+b(x,\xi)\cdot\partial_{\xi}$ where $a$ is homogeneous of degree $\nu$ and $b$ is homogeneous of degree $\nu+1$. Note that the commutator of two vector fields which are homogeneous of degrees $\nu_1$ and $\nu_2$ is homogeneous of degree $\nu_1+\nu_2$. Moreover, if for $\Sigma\subset T^*M$ we let $\mathcal{V}(\Sigma)$ denote the vector fields which are homogeneous of degree $0$ which are also tangent to $\Sigma$, then we have that the commutator of two vector fields in $\mathcal{V}(\Sigma)$ is also in $\mathcal{V}(\Sigma)$, i.e. $\mathcal{V}(\Sigma)$ forms a Lie algebra.
We can now define the symbol class, as follows:
\begin{definition}
Let $m,k\in\mathbb{R}$. The space $S^{m,k}(T^*M,\Sigma)$ is the set of all $a\in C^{\infty}(T^*M;\mathbb{C})$ satisfying the property that whenever $W^{\alpha} = W^{\alpha_1}\dots W^{\alpha_{|\alpha|}}$ is a product of vector fields on $T^*M\backslash o$ homogeneous of degree $0$, and $V^{\beta} = V^{\beta_1}\dots V^{\beta_{|\beta|}}$ is a product of vector fields in $\mathcal{V}(\Sigma)$, that (recalling the local coordinates $(p,y,r)$ described above) we have the local estimate
\[|W^{\alpha}V^{\beta}a|\le Cr^md_{\Sigma}^{k-|\alpha|}.\]
\end{definition}
Roughly speaking $S^{m,k}$ are symbols of order $m$ whose principal part vanishes to order $k$ on $\Sigma$, with the subprincipal symbols of order less than $k/2$ lower also vanishing as well.
We list several properties of this symbol class:
\begin{itemize}
\item It is an algebra, and in particular $a\in S^{m,k}(T^*M,\Sigma)$, $b\in S^{m',k'}(T^*M,\Sigma)\implies ab\in S^{m+m',k+k'}(T^*M,\Sigma)$. (This just follows from the Leibniz rule.)
\item We have that 
\[S^{m,k}(T^*M,\Sigma)\subset S^{m+k_-/2}_{1/2}(T^*M),\]
where $k_- = \max(0,-k)$ and $S_{1/2}$ is the $(1/2,1/2)$ symbol class of H\"ormander. (This is a correction to the statement in \cite{bdm} before Example 1.4, where the sign is flipped.) Indeed, notice that $d_{\Sigma}\ge r^{-1/2}$; on the other hand, away from a neighborhood of the zero section we have $r>\epsilon$, while we are free to take the defining functions $p$ for $\Sigma$ to be bounded as well since $p$ is homogeneous of degree $0$, so $d_{\Sigma}\le C$, say locally in the base away from the zero section. This implies $d_{\Sigma}^s\le r^{-s/2}$ if $s\le 0$ and $d_{\Sigma}^s\le C$ if $s>0$; thus for any $k$ and any $\alpha$ we have\footnote{If $k-|\alpha|\le 0$ then $d_{\Sigma}^{k-|\alpha|}\le r^{-(k-|\alpha|)/2}\le r^{(k_-+|\alpha|)/2}$ for $r\ge 1$, since $-k\le k_-$ by definition. If $k-|\alpha|>0$, then $d_{\Sigma}^{k-|\alpha|}\le C$ and $k>0\implies k_-=0\implies r^{(k_-+|\alpha|)/2} = r^{|\alpha|/2}\ge 1\ge C^{-1}d_{\Sigma}^{k-|\alpha|}$.} $d_{\Sigma}^{k-|\alpha|}\le Cr^{(k_-+|\alpha|)/2}$. 
\item More generally, by the same logic above we have 
\begin{equation}
\label{smkinclude}
\begin{aligned}
&S^{m,k}(T^*M,\Sigma)\subset S^{m',k'}(T^*M,\Sigma)\\
&\text{iff }m\le m'\text{ and }m-k/2\le m'-k'/2.
\end{aligned}
\end{equation}
\item If $V\in\mathcal{V}(\Sigma)$ and $a\in S^{m,k}(T^*M,\Sigma)$, then $Va\in S^{m,k}(T^*M,\Sigma)$. If $W$ is homogeneous of degree $0$ (but not necessarily tangent to $\Sigma$), then $Wa\in S^{m,k-1}(T^*M,\Sigma)$. In particular, if $\tilde{W}$ is homogeneous of degree $-1$ (e.g. the standard $\xi$ derivatives) then $\tilde{W}a\in S^{m-1,k-1}(T^*M,\Sigma)\subset S^{m-1/2,k}(T^*M,\Sigma)$ by the above comment.
\item The standard $(1,0)$ symbol class $S_{1,0}^m$ is contained in $S^{m,0}(T^*M,\Sigma)$. On the other hand, if the symbol vanishes appropriately on $\Sigma$, then we can say more. In fact, if $a\sim\sum{a_{m-j/2}}$ with $a_{m-j/2}$ homogeneous of degree $m-j/2$, and $a_{m-j/2}$ vanishes of order at least $k-j$ on $\Sigma$ for all $0\le j<k$, then $a\in S^{m,k}(T^*M,\Sigma)$. In particular, if $k=1$ or $2$, and $a$ is a classical symbol (so $a\sim a_m + a_{m-1} + \dots$), then $a_m$ vanishing to order $k$ on $\Sigma$ implies that $a\in S^{m,k}(T^*M,\Sigma)$.
\end{itemize}
Note that the symbol class is invariant under diffeomorphisms. Since $S^{m,k}(T^*M,\Sigma)\subset S^{m-k_-/2}_{1/2}(T^*M)$, these symbols can be quantized to $\Psi$DOs which are bounded from $H^s$ to $H^{s-m-k_-/2}$. For $a\in S^{m,k}(T^*M,\Sigma)$, let $a(x,D)$ denote (a) corresponding quantization, and denote $\Psi^{m,k}(M,\Sigma)$ the collection of all such operators. Then $\Psi^{m,k}(M,\Sigma)$ is defined independently of coordinates as well, in the sense that for any $A\in\Psi^{m,k}(M,\Sigma)$ and any local coordinates $(U,x)$ and cutoff $\chi\in C_c^{\infty}(U)$ we have $x^*\chi A\chi(x^{-1})^* = a(x,D)$ for some $a\in S^{m,k}(T^*x(U),\Sigma^x)$, where $\Sigma^x$ is the image $\Sigma$ under the symplectomorphism obtained by lifting the coordinate map $x$ to the cotangent bundle (so explicitly $\Sigma^x = \{(x_0,\xi)\in x(U)\times\mathbb{R}^n\,:\,\sum{\xi_i\,dx_i}|_{x_0}\in\Sigma\}$). Furthermore, if $A = a(x,D) = \tilde{a}(\tilde{x},D)$, i.e. we have two symbols quantizing the same operator under different coordinates, then (viewed as functions on $T^*M$) we have $a = \tilde{a}$ modulo a symbol in $S^{m-1/2,k-1}(T^*M,\Sigma)$. Note however that while this error is $1/2$ order better away from $\Sigma$, near $\Sigma$ the error is still of the same size as the original symbol.

This can be fixed in our case, where $\Sigma$ is a line subbundle of $T^*M$ with an integrable kernel. This means that for every $x_0$ there is a function $f$ such that $df|_{x_0} \ne 0$ and, for $x$ near $x_0$, we have
\[\ker\Sigma_x = T_xf^{-1}(\{f(x)\}),\]
i.e. $f$ labels the leaves of a foliation where, for every $x$, we have that the kernel $\ker\Sigma_x$ of the fiber of $\Sigma$ at $x$ coincides with the tangent space of the leaf at $x$. If $\Sigma$ is given as the span of a covector field $\overline{\xi}$, then this implies that $\overline{\xi}$ is a local smooth multiple of $df$, so if $\overline\xi$ is normalized in an appropriate manner, then $\overline{\xi} = df/|df|$. Conversely, if $\overline\xi$ is a smooth multiple of a closed 1-form, then its kernel is integrable, by Poincar\'e's Lemma.

In this case, we can consider charts $(U,x)$ where the last coordinate labels the leaves of the foliation, i.e. $x_n = f$ where $f$ satisfies the properties above, so that $\Sigma = \text{span }dx_n$ on $U$; such charts will be called \emph{foliated charts}. We will show below that symbols quantizing the same operator by foliated charts will differ by an element of $S^{m-1,k-1}$, a $1/2$ order improvement over the general case. Since this does not appear to be discussed in \cite{bdm}, we explain the details below.

Suppose that $(U,x)$ and $(U,y)$ are both foliated charts, and let $\varphi = y\circ x^{-1}$. Since $dy_n = \sum_{j=1}^n{\frac{\partial\varphi_n}{\partial x_j}dx_j}$, it follows from $\text{span }dx_n = \text{span }dy_n$ that $\frac{\partial\varphi_n}{\partial x_j} = 0$ for $j\ne n$. Let $\Sigma^x$ and $\Sigma^y$ denote the images of $\Sigma$ under the symplectomorphisms obtained by lifting the coordinate maps $x$ and $y$ to the cotangent bundle. Let $a(y,\eta)$ be a symbol in $S^{m,k}(T^*y(U),\Sigma^y)$ (say with spatial compact support in $y(U)$ so that we are free to view it as a symbol on $\mathbb{R}^n$), and $A$ be given by the left quantization of $a$. We study the symbol $b(x,\xi)$ of $B = \varphi^*A(\varphi^{-1})^*$. We review the so-called ``Kuranishi trick'': we write
\begin{align*}&\varphi^*A(\varphi^{-1})^*u(x) \\
&= (2\pi)^{-n}\int{e^{i(\varphi(x)-y)\cdot\eta}a(\varphi(x),\eta)u(\varphi^{-1}(y))\,dy\,d\eta} \\
&= (2\pi)^{-n}\int{e^{i(\varphi(x)-\varphi(x'))\cdot\eta}a(\varphi(x),\eta)u(x')|\det D\varphi(x')|\,dx'\,d\eta}\\
&= (2\pi)^{-n}\int{e^{i(F(x,x')(x-x'))\cdot\eta}a(\varphi(x),\eta)u(x')|\det D\varphi(x')|\,dx'\,d\eta} \\
&= (2\pi)^{-n}\int e^{i(x-x')\cdot\xi}a(\varphi(x),(F(x,x')^T)^{-1}\xi)|\det D\varphi(x')|\\
&\phantom{(2\pi)^{-n}\int}|\det F(x,x')|^{-1}u(x')\,dx'\,d\xi
\end{align*}
where we make the substitution $y=\varphi(x')$ and use the fact that
\[\varphi(x)-\varphi(x') = F(x,x')(x-x'),\quad F(x,x') = \int_0^1{D\varphi(tx+(1-t)x')\,dt}\]
(i.e. $F_{ij}(x,x')=\int_0^1{\partial_j\varphi_i(tx+(1-t)x')\,dt}$; note that $F(x,x) = D\varphi(x)$) to write
\[(\varphi(x)-\varphi(x'))\cdot\eta = F(x,x')(x-x')\cdot\eta = (x-x')\cdot F(x,x')^T\eta;\]
we then make the substitution $\xi = F(x,x')^T\eta$ in the final step. Thus we have
\[\varphi^*A(\varphi^{-1})^*(x,x') = (2\pi)^{-n}\int{e^{i(x-x')\cdot\xi}\tilde{b}(x,x',\xi)\,d\xi}\]
where
\[\tilde{b}(x,x',\xi) = a(\varphi(x),(F(x,x')^T)^{-1}\xi)|\det D\varphi(x')||\det F(x,x')|^{-1}.\]
We now use the fact that
\[(2\pi)^{-n}\int{e^{i(x-x')\cdot\xi}\tilde{b}(x,x',\xi)\,d\xi} = (2\pi)^{-n}\int{e^{i(x-x')\cdot\xi}b(x,\xi)\,d\xi}\]
where
\[b(x,\xi)\sim\sum_{\alpha}{\frac{(-i)^{|\alpha|}}{\alpha!}\partial_{\xi}^{\alpha}\partial_{x'}^{\alpha}\tilde{b}(x,x,\xi)}\]
to study the effects of various vector fields on $b$ in terms of those effects on $a$.

We first note that since $F_{ij}(x,x')=\int_0^1{\partial_j\varphi_i(tx+(1-t)x')\,dt}$ and $\partial_j\varphi_n = 0$ for $j\ne 0$, it follows that $F_{nj}\equiv 0$ for $j\ne n$. Thus $F$ has the block matrix form $\begin{pmatrix} * & * \\ 0 & * \end{pmatrix}$ where the blocks are with respect to separating the first $n-1$ variables from the last variable, and hence $(F^T)^{-1}$ has the block form $\begin{pmatrix} * & 0 \\ * & * \end{pmatrix}$.

We next study the functions $\partial_{\xi}^{\alpha}\partial_{x'}^{\beta}(a(\varphi(x),(F(x,x')^T)^{-1}\xi))$ (the other two terms in the product defining $\tilde{b}$ will not affect the differential behavior very much.) We note that applying derivatives in $x'$ results in a sum of quantities which are applications of vector fields of the form
\[\partial_{x'}^{\gamma}(F(x,x')^T)^{-1}\xi\cdot\partial_{\eta}\]
to $a$, evaluated at $(\varphi(x),(F(x,x')^T)^{-1}\xi)$. Since $(F^T)^{-1}_{jn}\equiv 0$ for $j\le n$, it follows that the same is true of the derivatives: $\partial_{x'}^{\gamma}(F(x,x')^T)^{-1}_{jn}\equiv 0$. It follows that
\begin{align*} \partial_{x'}^{\gamma}(F(x,x')^T)^{-1}\xi\cdot\partial_{\eta} 
&= \sum_{jk}{\partial_{x'}^{\gamma}(F(x,x')^T)^{-1}_{jk}\xi_k\partial_{\eta_j}} \\
&=\left(\sum_{j=1}^{n-1}\sum_{k=1}^{n-1}{\partial_{x'}^{\gamma}(F(x,x')^T)^{-1}_{jk}\xi_k\partial_{\eta_j}}\right) + \sum_{k=1}^n{\partial_{x'}^{\gamma}(F(x,x')^T)^{-1}_{nk}\xi_k\partial_{\eta_n}}.
\end{align*}
If we let $\xi = D\varphi(x)\eta$, then the above provides a vector field tangent to $\Sigma^y$ for all $(x,x')$, since for $k\ne n$ we have that $\xi_k$ is a combination of $\eta_l$ for $l\ne n$. In particular, evaluating at $x'=x$ gives
\[\partial_{x'}^{\beta}(a(\varphi(x),(F(x,x')^T)^{-1}\xi))|_{x'=x} = \sum{(V^{\beta'}a)(\varphi(x),(D\varphi(x)^T)^{-1}\xi)}\]
where $V^{\beta'}$ is a product of vector fields in $\mathcal{V}(\Sigma^y)$. Then taking derivatives in $\xi$ results in application of vector fields of the form $\sum{(D\varphi(x)^T)^{-1}_{jk}\partial_{\eta_j}}$, i.e. smooth in $x$ times one $\eta$ derivative. Thus, we have that $\partial_{\xi}^{\alpha}\partial_{x'}^{\beta}(a(\varphi(x),(F(x,x')^T)^{-1}\xi))|_{x'=x}$ is a sum of terms of the form
\[(\text{smooth function on }U)\times(\partial_{\eta}^{\alpha'}V^{\beta'}a)(\varphi(x),(D\varphi(x)^T)^{-1}\xi)\]
where $|\alpha'| = |\alpha|$ and $V^{\beta'}$ is a product of vector fields in $\mathcal{V}(\Sigma^y)$. Since $\partial_{\eta}^{\alpha'}V^{\beta'}a\in S^{m-|\alpha|,k-|\alpha|}(T^*y(U),\Sigma_y)$, and $(x,\xi)\mapsto (\varphi(x),(D\varphi(x)^T)^{-1}\xi) = (\varphi(x),(D\varphi(x)^T)^{-1}\xi)$ is precisely the symplectomorphism obtained by lifting the diffeomorphism $\varphi$, it follows that terms of the above form belong to $S^{m-|\alpha|,k-|\alpha|}(T^*x(U);\Sigma_x)$.

Finally, since
\begin{align*}&\partial^{\alpha}_{\xi}\partial^{\alpha}_x\left(a(\varphi(x),(F(x,x')^T)^{-1}\xi)|\det D\varphi(x')||\det F(x,x')|^{-1}\right)|_{x'=x} \\
&= \sum_{\beta\le\alpha}\left[\binom{\alpha}{\beta}\partial^{\alpha}_{\xi}\partial^{\beta}_{x'}\left(a(\varphi(x),(F(x,x')^T)^{-1}\xi)\right)\cdot\partial_{x'}^{\alpha-\beta}\left(|\det D\varphi(x')||\det F(x,x')|^{-1}\right)|_{x'=x}\right],
\end{align*}
with the term $\partial_{x'}^{\alpha-\beta}\left(|\det D\varphi(x')||\det F(x,x')|^{-1}\right)|_{x'=x}$ a smooth function in $x$, it follows that 
\begin{align*} &\partial^{\alpha}_{\xi}\partial^{\alpha}_x\left(a(\varphi(x),(F(x,x')^T)^{-1}\xi)|\det D\varphi(x')||\det F(x,x')|^{-1}\right)|_{x'=x}\\
&\in S^{m-|\alpha|,k-|\alpha|}(T^*x(U);\Sigma_x)
\end{align*}
as well.

Thus, we have that $b(x,\xi)\sim\sum_{\alpha}{b_{\alpha}(x,\xi)}$ where $b_{\alpha}\in S^{m-|\alpha|,k-|\alpha|}(T^*x(U);\Sigma^x)$. For $0\le j'\le j$ we have $S^{m-j,k-j}\subset S^{m-j+j'/2,k-j+j'}$, so we have both $S^{m-|\alpha|,k-|\alpha|}\subset S^{m-1,k-1}$ for $|\alpha|\ge 1$ and $S^{m-|\alpha|,k-|\alpha|}\subset S^{m-|\alpha|/2,k}$. Thus, the terms in the asymptotic expansion really are ``lower order'' both away from and near $\Sigma$. In particular, $b$ differs from $b_0(x,\xi) = a(\varphi(x),(D\varphi(x)^T)^{-1}\xi)$ (i.e. $a$ evaluated at the appropriate covector given by the change of coordinates) by an element of $S^{m-1,k-1}(T^*M,\Sigma)$. Thus in this case we can define the principal symbol $\sigma:\Psi^{m,k}(M,\Sigma)\rightarrow S^{m,k}(T^*M,\Sigma)/S^{m-1,k-1}(T^*M,\Sigma)$ by having $\sigma(A)$ be the representative class of any symbol which quantizes $A$ with respect to \emph{foliated} charts.

With this notion of principal symbol, we can establish the composition rule
\[ A\in\Psi^{m,k}(M,\Sigma), B\in\Psi^{m',k'}(M,\Sigma) \implies AB\in\Psi^{m+m',k+k'}(M,\Sigma)\quad\text{with }\sigma(AB) = \sigma(A)\sigma(B)
\]
so that in particular $AB$ differs from any quantization of $\sigma(A)\sigma(B)$ by an element quantized by a symbol in $S^{m+m'-1,k+k'-1}(T^*M,\Sigma)$. To do so, we note that $A$ and $B$ are pseudolocal, and hence so is their composition, so it only suffices to check that $AB$ is locally quantized by an element of $S^{m+m',k+k'}(T^*M,\Sigma)$. Thus, if $(U,x)$ is a foliated chart, and $\chi\in C_c^{\infty}(U)$, and $A$ and $B$ are locally quantized in this chart by $a$ and $b$ in the sense that $\chi x^*A(x^{-1})^*\chi = \chi a(x,D)\chi$ and similarly for $b$, then
\[
\chi x^*AB(x^{-1})^*\chi = \chi(a\# b)(x,D)\chi, \quad\text{with }a\# b(x,\xi)\sim\sum_{\alpha}{\frac{(-i)^{|\alpha|}}{\alpha!}\partial^{\alpha}_{\xi}a(x,\xi)\partial^{\alpha}_xb(x,\xi)},
\]
where $a\in S^{m,k}(T^*M,\Sigma)$ and $b\in S^{m',k'}(T^*M,\Sigma)$. Note that $\partial^{\alpha}_{\xi}a(x,\xi)\in S^{m-|\alpha|,k-|\alpha|}(T^*M,\Sigma)$ while $\partial^{\alpha}_xb(x,\xi)\in S^{m',k'}(T^*M,\Sigma)$ since $x$ derivatives are tangent to $\Sigma$ with respect to the coordinates chosen; hence their product belongs to $S^{m+m'-|\alpha|,k+k'-|\alpha|}(T^*M,\Sigma)$. This shows that $a\# b$ belongs to $S^{m+m',k+k'}(T^*M,\Sigma)$ and agrees with $ab$ up to a symbol in $S^{m+m'-1,k+k'-1}$, as desired.

Perhaps the most important property of this symbol class is that it contains the inverse of parabolic symbols like that of the heat operator. Indeed, if $p_m\in S^m$ is nonnegative and vanishes nondegenerately quadratically on $\Sigma$, and $p_{m-1}\in S^{m-1}$ is real-valued and is elliptic on $\Sigma$, then we have the key estimate
\[|p_m+ip_{m-1}|\ge c\left(|\zeta|^m\frac{|\zeta'|^2}{|\zeta|^2}+|\zeta|^{m-1}\right) = cr^md_{\Sigma}^2.\]
More generally we have the following:
\begin{proposition}
Suppose $p\in S^{m,k}(T^*M,\Sigma)$ satisfies the lower bound $|p|\ge cr^md_{\Sigma}^k$. Then $1/p\in S^{-m,-k}(T^*M,\Sigma)$.
\end{proposition}
The proof is analogous to the standard proof that the inverse of an elliptic symbol is a symbol.

We can extend this calculus to operators on vector bundles on manifolds in the same way that the standard pseudodifferential calculus extends, namely by considering operators which under local coordinates (say under foliated coordinates) can be written as a matrix of $\Psi$DOs belonging to this calculus; the principal symbol of such operators will be a matrix whose entries belong to the symbol calculus. For $M = \mathbb{R}^n$, denote this operator calculus by $\Psi^{m,k}(\mathbb{R}^n,\Sigma)\otimes\text{Mat}_{n\times n}(\mathbb{C})$.

An easy application of the calculus constructed above is the following lemma:
\begin{lemma}
\label{paropinv}
Suppose $N\in\Psi^m(\mathbb{R}^n;\text{Mat}_{n\times n}(\mathbb{C}))$ has a left-reduced symbol of the form
\[\sigma_L(N)(x,\zeta) = p(x,\zeta)\,\text{Id} + P_{m-1}(x,\zeta) + P_{m-2}(x,\zeta)\]
where $p = p_m+ip_{m-1}$ with $p_i\in S^i(T^*\mathbb{R}^n)$, $p_m$ nonnegative and vanishing nondegenerately quadratically on $\Sigma$, $p_{m-1}$ is elliptic on $\Sigma$, and $P_i\in S^i(T^*\mathbb{R}^n)\otimes\text{Mat}_{n\times n}(\mathbb{C})$ with $P_{m-1}$ vanishing on $\Sigma$. Then $p\in S^{m,2}(T^*\mathbb{R}^n,\Sigma)$, $\sigma_L(N)\in S^{m,2}(T^*\mathbb{R}^n,\Sigma)$ with the principal symbol satisfying $\sigma_{m,2}(N) = p\,\text{Id}$, and if we let $q = 1/p$, then $q\in S^{-m,-2}(T^*\mathbb{R}^n,\Sigma)$, and for $Q = q(x,D)$ we have $Q\circ N = \text{Id} + R$ where $R\in\Psi^{-1,-1}(\mathbb{R}^n,\Sigma)\otimes\text{Mat}_{n\times n}(\mathbb{C})$.
\end{lemma}
\begin{proof} 
Note that $p_m\in S^{m,2}(T^*\mathbb{R}^n,\Sigma)$ since it vanishes quadratically on $\Sigma$, and $p_{m-1}\in S^{m-1,0}(T^*\mathbb{R}^n,\Sigma)\subset S^{m,2}(T^*\mathbb{R}^n,\Sigma)$ by \eqref{smkinclude}. By the hypothesis of the lemma, we have that $p$ satisfies the lower bound $|p|\ge cr^md_{\Sigma}^2$, and hence $q\in S^{-m,-2}(T^*\mathbb{R}^n,\Sigma)$. 
It now suffices to show the remaining terms in $\sigma_L(N)$ are in $S^{m-1,1}(T^*\mathbb{R}^n,\Sigma)\otimes\text{Mat}_{n\times n}(\mathbb{C})$. Since $P_{m-1}$ vanishes on $\Sigma$, it follows that $P_{m-1}\in S^{m-1,1}(T^*\mathbb{R}^n,\Sigma)\otimes\text{Mat}_{n\times n}(\mathbb{C})$, while $P_{m-2}\in S^{m-2,0}(T^*\mathbb{R}^n,\Sigma)\otimes\text{Mat}_{n\times n}(\mathbb{C})\subset S^{m-1,1}(T^*\mathbb{R}^n,\Sigma)\otimes\text{Mat}_{n\times n}(\mathbb{C})$ by \eqref{smkinclude}. Thus, we have that $\sigma_L(N)\in S^{m,2}(T^*\mathbb{R}^n,\Sigma)\otimes\text{Mat}_{n\times n}(\mathbb{C})$, with $\sigma_{m,2}(N) = p\,\text{Id}$, and since $q\cdot\sigma_L(N) = \text{Id}$ modulo $S^{-1,-1}$, it follows that 
\[Q\circ N = \text{Id}\mod \Psi^{-1,-1}(\mathbb{R}^n,\Sigma)\otimes\text{Mat}_{n\times n}(\mathbb{C}),\]
as desired.
\end{proof}
\begin{remark}
In fact, a parametrix can be chosen to invert up to an element of $\cap_{j\ge 0}{\Psi^{m-j,k-j}}$, which happens to coincide with the standard residual class $\Psi^{-\infty}$ essentially because $\Psi^{m-j,k-j}\subset\Psi_{1/2}^{m-j+(k-j)_{-}/2} = \Psi_{1/2}^{m-j/2-k/2}$ for $j$ large enough. This is analogous to the situation in the $(1/2,0)$ calculus on $\mathbb{R}^n$, since the principal symbol result for compositions hold. Such a parametrix will in general not be scalar-valued (though its principal symbol will be). However, we will not take advantage of this fact here, since we will also need to apply $Q\in\Psi^{-m,-2}(\mathbb{R}^n,\Sigma)$ to the operators $\tilde{N} = \tilde{N}^{\nu}_{\pm}$, which contribute terms that end up being comparable (in differential order) to the $\Psi^{-1,-1}$ error obtained above.
\end{remark}

With this calculus constructed, we can now rephrase the symbol calculations of Theorem \ref{symbolthm} as follows:
\begin{theorem}
\label{newsymbolthm}
We have $N^{11}_+\in\Psi^{-1,0}(\mathbb{R}^3,\Sigma)\otimes\text{Mat}_{3\times 3}(\mathbb{C})$, while for all other $N^{\nu}_{\pm}$ we have $N^{\nu}_{\pm}\in\Psi^{-1,2}(\mathbb{R}^3,\Sigma)\otimes\text{Mat}_{3\times 3}(\mathbb{C})$. In addition, the $S^{-1,k}$ principal symbols of the $N^{\nu}_{\pm}$ ($k=0$ for $N^{11}_+$ and $k=2$ for all other $N^{\nu}_{\pm}$) can be taken to be scalar multiples of the identity matrix. Furthermore, under the assumptions in the remarks following \ref{subsymb}, we have that $N^{11}_-$ (if $E^2>0$ everywhere), $N^{33}_+$, and $N^{E^2}_{\pm}$ are elliptic as $\Psi^{-1,2}$ operators. Finally, we have $\tilde{N}^{11}_{\pm}\in\Psi^{-1,0}(\mathbb{R}^3,\Sigma)\otimes\text{Mat}_{3\times 1}(\mathbb{C})$, while for $\nu\ne 11$ we have $\tilde{N}^{\nu}_{\pm}\in\Psi^{-1,1}(\mathbb{R}^3,\Sigma)\otimes\text{Mat}_{3\times 1}(\mathbb{C})$.
\end{theorem}

\section{Recovery estimates}
\label{recovsec}
We are now in a position to analyze possible inversion situations and obtain estimates in these situations, and hence prove Theorems \ref{oneparam}, \ref{twoparam}, and \ref{func}. We recall from Section \ref{intro-sec} that we wish to prove general ``stability estimates'' of the form \eqref{stabest} in order to prove our theorems. (See Corollary \ref{oneparamest} and Propositions \ref{twocoeffprop} and \ref{funcprop} for the precise statements of the desired stability estimates.)

We start with the case of inverting one parameter, assuming the others are known. In \cite{ti}, where there was an artificial boundary, the authors noted that the operator for the $qP$-travel time data for $a_{11}$ was an elliptic (scattering) $\Psi$DO, and hence one can obtain an estimate of the form
\begin{equation}
\|\nabla u\|_{L^2}\le C\left(\|N^{11}_+u\|_{H^1}+\|u\|_{L^2}\right)\label{pastestimate}
\end{equation}
from elliptic regularity. By taking the artificial boundary to be sufficiently close to the actual boundary, one can then absorb the $\|u\|_{L^2}$ term into the left-hand side via an argument using Poincar\'e's inequality.

We aim to obtain similar kinds of estimates when the operators in question are parabolic and not elliptic. It turns out that we obtain optimal estimates when the support of the differences $r_{\nu}$ are supported in sets of small \emph{width}; we define this notion now.
\begin{definition}
\label{width}
A (closed) rectangular domain $R\subset\mathbb{R}^n$ is a set for which there exist $A\in O(n)$, $b\in\mathbb{R}^n$, and $r_1,\dots,r_n\in(0,\infty)$ such that
\[AR - b = \{(x_1,\dots,x_n)\,:\,0\le x_i\le r_i\text{ for all }1\le i\le n\}.\]
We define the \emph{width} $w(R)$ of $R$ as the minimum value of $r_i$ over all $r_i$ in the condition above. For a bounded set $D\subset\mathbb{R}^n$, define its width $w(D)$ as
\[w(D) = \inf\{w(R)\,:\,D\subset R, R\text{ rectangular}\}.\]
\end{definition}
The upshot of this definition is the following quantitative version of Poincar\'e's inequality: if $u\in C_c^{\infty}(\mathbb{R}^n)$, then
\[\|u\|_{L^2(\mathbb{R}^n)}\le\frac{w(\text{supp }u)}{\sqrt{2}}\|\nabla u\|_{L^2(\mathbb{R}^n)}.\]
It suffices to prove the estimate with $w(\text{supp }u)$ replaced by $w(R)$ for $u\in C_c^{\infty}(R)$ where $R$ is of the form $R = \{x\in\mathbb{R}^n\,:\,0\le x_i\le r_i\}$. We can estimate $\|u\|_{L^2(R)}$ by computing
\begin{align*}&\int_R{|u(x)|^2\,dx} \\
&= \int_{0}^{r_1}\dots\int_{0}^{r_{n-1}}\int_{0}^{r_n}{\left|\int_{0}^{x_n}{\partial_{x_n}u(x_1,\dots,x_{n-1},y)\,dy}\right|^2\,dx_n\,dx_{n-1}\dots\,dx_1} \\
&\le\int_{0}^{r_1}\dots\int_{0}^{r_{n-1}}\int_{0}^{r_n}{\left(\int_{0}^{r_n}{|\partial_{x_n}u(x_1,\dots,x_{n-1},y)|^2\,dy}\right)x_n\,dx_n\,dx_{n-1}\dots\,dx_1}\\
&\le\int_{0}^{r_1}\dots\int_{0}^{r_{n-1}}{\left(\int_{0}^{r_n}{|\nabla u(x_1,\dots,x_{n-1},y)|^2\,dy}\right)\frac{r_n^2}{2}\,dx_{n-1}\dots\,dx_1}\\
&=\frac{r_n^2}{2}\|\nabla u\|_{L^2(R)}^2
\end{align*}
where the second line follows from Cauchy-Schwarz, and hence $\|u\|_{L^2(R)}\le\frac{r_n}{\sqrt{2}}\|\nabla u\|_{L^2(R)}$. Changing the order of coordinates so that we can take $r_n = \min\{r_i\}$, it follows that $\|u||_{L^2(R)}\le\frac{w(R)}{\sqrt{2}}\|\nabla u\|_{L^2(R)}$, as desired.

With this quantitative version of Poincar\'e's inequality, we first note without further proof that, in this setting, the elliptic regularity result for $a_{11}$ also holds:
\begin{proposition}
Suppose that $a_{33}$ and $E^2$ are known, and let $f = N^{11}_+[\nabla r_{11}] + \tilde{N}^{11}_+[r_{11}]$. Then
\begin{equation}
\|\nabla r_{11}\|_{L^2(\mathbb{R}^3)}\le C\left(\|f\|_{H^1(\mathbb{R}^3)} + \|r_{11}\|_{L^2(\mathbb{R}^3)}\right).\label{a11+global}
\end{equation}
In particular, for $r_{11}$ with sufficiently small width of support we have
\begin{equation}
\|\nabla r_{11}\|_{L^2(\mathbb{R}^3)}\le C\|f\|_{H^1(\mathbb{R}^3)}.\label{a11+globalinj}
\end{equation}
In particular, if $r_{11}$ is known to have sufficiently small width of support, and $a_{11}$ and $\tilde{a}_{11}$ give the same travel time data, then $r_{11}\equiv 0$, i.e. we have uniqueness within functions that differ only on sets of sufficiently small width.
\end{proposition}
\begin{remark} If $\text{supp }r_{11}$ can be written as a disjoint union of closed connected components, then the width can be replaced with the maximum width of each component. In general, if the support is contained in a ``thin'' set of sufficiently small curvature, so that it can be covered by a union of rectangles of small width with a ``low number of overlaps'', then a similar Poincar\'e inequality argument should be possible by taking a partition of unity subordinate to the cover of thin rectangles and applying the Poincar\'e inequality argument to each piece; the ``low number of overlaps'' then helps patch the estimates back together.
\end{remark}
\begin{remark} The stability estimate in \eqref{a11+global} is the crucial result: indeed, if the injectivity of the operator can be otherwise established, then the stability estimate upgrades to an estimate of the form \eqref{a11+globalinj}.
\end{remark}
We now establish the analogous estimates of \eqref{a11+global} and \eqref{a11+globalinj} for the other parameters. The estimates will follow from the following general argument:
\begin{proposition} Suppose $N\in\Psi^m(\mathbb{R}^n;\text{Mat}_{n\times n}(\mathbb{C}))$ satisfies the assumptions of Lemma \ref{paropinv} and $\tilde{N}\in\Psi^m(\mathbb{R}^n;\text{Mat}_{n\times 1}(\mathbb{C}))$ satisfies $\sigma(\tilde{N})|_{\Sigma}\equiv 0$, and let $f = N[\nabla u]+\tilde{N}[u]$. Then we have the estimate
\begin{equation}
\|\nabla u\|_{H^s(\mathbb{R}^n)}\le C\left(\|f\|_{H^{s+1-m}(\mathbb{R}^n)} + \|u\|_{H^{s+1/2}(\mathbb{R}^n)}\right).
\label{parestimate}
\end{equation}
Furthermore, a $H^{1/2}$-version of the Poincar\'e inequality holds:
\[\|u\|_{H^{1/2}(\mathbb{R}^n)}\le \left(\frac{w^2}{2}+\frac{w}{\sqrt{2}}\right)^{1/2}\|\nabla u\|_{L^2(\mathbb{R}^n)},\quad w = w(\text{supp }u).\]
Thus, if $u$ has sufficiently small width of support, we can conclude
\begin{equation}
\|\nabla u\|_{L^2(\mathbb{R}^n)}\le C\|f\|_{H^{1-m}(\mathbb{R}^n)}.
\label{parestimateinj}
\end{equation}
\end{proposition}
\begin{proof}
Let $Q\in\Psi^{-m,-2}(\mathbb{R}^n,\Sigma)$ be the operator obtained from the proof of Lemma \ref{paropinv} which satisfies $Q\circ N = \text{Id} + R$, with $R\in\Psi^{-1,-1}(\mathbb{R}^n,\Sigma)\otimes\text{Mat}_{n\times n}(\mathbb{C})$. Applying $Q$ to the equation $f = N[\nabla u]+\tilde{N}[u]$ yields
\[\nabla u = Qf - R[\nabla u] + (Q\circ\tilde{N})u.\]
Note that $Q\in\Psi^{-m,-2}(\mathbb{R}^n,\Sigma)\subset\Psi_{1/2}^{-m+1}(\mathbb{R}^n)$ implies that it maps boundedly from $H^{s+1-m}(\mathbb{R}^n)$ to $H^s(\mathbb{R}^n)$, and $R\in\Psi^{-1,-1}(\mathbb{R}^n,\Sigma)\subset\Psi_{1/2}^{-1/2}(\mathbb{R}^n)$ implies it maps boundedly from $H^{s-1/2}(\mathbb{R}^n)$ to $H^s(\mathbb{R}^n)$. Since $\sigma(\tilde{N})\in S^m(T^*\mathbb{R}^n)\otimes\text{Mat}_{n\times 1}(\mathbb{C})$ vanishes on $\Sigma$, it follows that $\sigma(\tilde{N})\in S^{m,1}(T^*\mathbb{R}^n,\Sigma)\otimes\text{Mat}_{n\times 1}(\mathbb{C})$, and hence $Q\circ\tilde{N}\in\Psi^{0,-1}(\mathbb{R}^n,\Sigma)\otimes\text{Mat}_{n\times 1}(\mathbb{C})$; in particular $Q\circ\tilde{N}$ maps boundedly from $H^{s+1/2}(\mathbb{R}^n)$ to $H^s(\mathbb{R}^n)$. It follows that
\begin{align*}
\|\nabla u\|_{H^s(\mathbb{R}^n)} &\le \|Qf\|_{H^s(\mathbb{R}^n)} + \|R[\nabla u]\|_{H^s(\mathbb{R}^n)} + \|(Q\circ\tilde{N})u\|_{H^s(\mathbb{R}^n)} \\
&\le C\left(\|f\|_{H^{s+1-m}(\mathbb{R}^n)}+\|\nabla u\|_{H^{s-1/2}(\mathbb{R}^n)} + \|u\|_{H^{s+1/2}(\mathbb{R}^n)}\right)\\
&\le C\left(\|f\|_{H^{s+1-m}(\mathbb{R}^n)} + \|u\|_{H^{s+1/2}(\mathbb{R}^n)}\right),
\end{align*}
thus giving \eqref{parestimate}.

The $H^{1/2}$ Poincar\'e inequality can be obtained from the standard Poincar\'e inequality $\|u\|_{L^2(\mathbb{R}^n)}\le \frac{w}{\sqrt{2}}\|\nabla u\|_{L^2(\mathbb{R}^n)}$ by Cauchy-Schwarz:
\begin{align*}
\|u\|_{H^{1/2}(\mathbb{R}^n)}^2&\le\|u\|_{L^2(\mathbb{R}^n)}\|u\|_{H^1(\mathbb{R}^n)}\\
&=\|u\|_{L^2(\mathbb{R}^n)}(\|u\|_{L^2(\mathbb{R}^n)}+\|\nabla u\|_{L^2(\mathbb{R}^n)})\\
&\le (w^2/2+w/\sqrt{2})\|\nabla u\|_{L^2(\mathbb{R}^n)}^2.
\end{align*}
Thus, if $w$ is sufficiently small, we can move the $\|u\|_{H^{1/2}}$ term to the LHS of \eqref{parestimate} (with $s=0$) to obtain \eqref{parestimateinj}.
\end{proof}
Noting that most of our operators of interest $N^{\nu}_{\pm}$ satisfy the assumptions of Lemma \ref{paropinv}, we immediately obtain the following corollary:
\begin{corollary}
\label{oneparamest}
For $N^{11}_-$ (if $E^2>0$), $N^{33}_+$, and $N^{E^2}_{\pm}$, we have the stability estimates 
\[\|\nabla r_{\nu}\|_{L^2(\mathbb{R}^3)}\le C\left(\|N^{\nu}_{\pm}(\nabla r_{\nu})+\tilde{N}^{\nu}_{\pm}(r_{\nu})\|_{H^2(\mathbb{R}^3)} + \|r_{\nu}\|_{H^{1/2}(\mathbb{R}^3)}\right).\]
\end{corollary}
From this corollary, we can prove Theorem \ref{oneparam}, as follows:
\begin{proof}[Proof of Theorem \ref{oneparam}]
Assuming that two of the parameters are known and that we are aiming to recover the final parameter $\nu$, the pseudolinearization equation \eqref{npseudo} reduces to the equation $\vec{0}_3 = N^{\nu}_{\pm}[\nabla r_{\nu}] + \tilde{N}^{\nu}_{\pm}[r_{\nu}]$. It follows that if the qSV travel times agree and we are trying to recover $a_{11}$ or $E^2$, or if the qP travel times agree and we are trying to recover $a_{33}$ or $E^2$, that the estimate in Corollary \ref{oneparamest} reduces to the estimate
\[\|\nabla r_{\nu}\|_{L^2(\mathbb{R}^3)}\le C\|r_{\nu}\|_{H^{1/2}(\mathbb{R}^3)},\]
so that if furthermore the width of support of $r_{\nu}$ is sufficiently small, then the Poincar\'e inequality implies $\nabla r_{\nu}\equiv 0$, and hence $r_{\nu}\equiv 0$ since it is compactly supported.
\end{proof}

\begin{remark}
For the width of support $w$ to be ``sufficiently small'', we need
\[\left(\frac{w^2}{2}+\frac{w}{\sqrt{2}}\right)^{1/2}\le (1-\epsilon)C^{-1}\]
where $C$ is a constant depending on the size (i.e. operator norm) of the 1st order parametrix $Q$ for $N^{\nu}_{\pm}$, the corresponding error operator $R$, and the operator $\tilde{N}^{\nu}_{\pm}$. The operator norm of a parametrix and the associated error is generally difficult to estimate quantitatively using conventional microlocal methods, though one potential workaround would be to attempt to make semiclassical versions of the arguments above, since one can more readily relate the operator norm of a semiclassical operator with the size of its semiclassical symbol. See \cite{semiclassical} for an example of a semiclassical treatment of analogous arguments made in \cite{convex}.
\end{remark}

We now analyze the problem of recovering two of three parameters, with the third either known or as a known function of the other two, from using both the $qP$ and $qSV$ travel time data. Recall that we have the equations
\begin{equation}
\label{joint}
\begin{aligned}
\vec{0}_3 &= \sum_{\nu}{N^{\nu}_+[\nabla r_{\nu}] + \tilde{N}^{\nu}_+[r_{\nu}]}, \\
\vec{0}_3 &= \sum_{\nu}{N^{\nu}_-[\nabla r_{\nu}] + \tilde{N}^{\nu}_-[r_{\nu}]}.
\end{aligned}
\end{equation}
For $N^{11}_+$, we let $\sigma\left(N^{11}_+\right)$ denote the principal symbol of $N^{11}_+$, while for all other $N^{\nu}_{\pm}$ we let $\sigma(N^{\nu}_{\pm})$ denote the sum of their principal and subprincipal symbols (so that $\sigma(N^{\nu}_{\pm})^{-1}\in S^{1,-2}(T^*\mathbb{R}^3,\Sigma)$ for $N^{33}_+$, $N^{E^2}_{\pm}$, and $N^{11}_-$ if $E^2>0$). To analyze the invertibility of the (matrix-valued) symbols of the operators in \eqref{joint}, we use the subprincipal behavior of the operators near $\Sigma$ and the symbol calculus developed in Section \ref{invparsec} to analyze the symbols near $\Sigma$, while away from $\Sigma$ we use the quantitative estimates developed at the end of Section \ref{excomp}; in particular we will take our $\chi$ to be supported in a sufficiently small neighborhood $\{|\xi_T|<\epsilon|\xi|\}$ of the equatorial sphere and identically one in a smaller neighborhood. We will use the following idea: if $a, d\ne 0$, then the inverse of the matrix $\begin{pmatrix} a & b \\ c & d \end{pmatrix}$ can be written as
\[\begin{pmatrix} a & b \\ c & d \end{pmatrix}^{-1} = \left(1 - \frac{bc}{ad}\right)^{-1}\begin{pmatrix} \frac{1}{a} & -\frac{b}{ad} \\ -\frac{c}{ad} & \frac{1}{d} \end{pmatrix},\]
provided that $1 - \frac{bc}{ad}$ is invertible.

In the rest of this section, we will write $6\times 6$ matrices as $2\times 2$ block matrices with $3\times 3$ blocks. A block containing a scalar expression should be identified with that scalar multiple of the $3\times 3$ identity.

First, let's suppose $a_{11}$ is known. Then we write the above equations as
\[\vec{0}_6 = \begin{pmatrix} N^{33}_+ & N^{E^2}_+ \\ N^{33}_- & N^{E^2}_- \end{pmatrix}\begin{pmatrix} \nabla r_{33} \\ \nabla r_{E^2} \end{pmatrix} + \begin{pmatrix} \tilde{N}^{33}_+ & \tilde{N}^{E^2}_+ \\ \tilde{N}^{33}_- & \tilde{N}^{E^2}_-\end{pmatrix}\begin{pmatrix} r_{33} \\ r_{E^2}\end{pmatrix}.\]
The inverse of the symbol of the first matrix can be written as
\[q = \left(1 - \frac{\sigma\left(N^{33}_-\right)\sigma\left(N^{E^2}_+\right)}{\sigma\left(N^{33}_+\right)\sigma\left(N^{E^2}_-\right)}\right)^{-1}\begin{pmatrix} \frac{1}{\sigma\left(N^{33}_+\right)} & - \frac{\sigma\left(N^{E^2}_+\right)}{\sigma\left(N^{33}_+\right)\sigma\left(N^{E^2}_-\right)} \\ - \frac{\sigma\left(N^{33}_-\right)}{\sigma\left(N^{33}_+\right)\sigma\left(N^{E^2}_-\right)} & \frac{1}{\sigma\left(N^{E^2}_-\right)}\end{pmatrix},\]
assuming the invertibility of $1 - \frac{\sigma\left(N^{33}_-\right)\sigma\left(N^{E^2}_+\right)}{\sigma\left(N^{33}_+\right)\sigma\left(N^{E^2}_-\right)}$. Since the principal parts of $\sigma\left(N^{33}_-\right)$ and $\sigma\left(N^{E^2}_+\right)$ both vanish quadratically on $\Sigma$ and hence are in $S^{-1,2}(T^*\mathbb{R}^3,\Sigma)$, and $\sigma\left(N^{33}_+\right)^{-1}$ and $\sigma\left(N^{E^2}_-\right)^{-1}$ are both of inverse parabolic type, i.e. belong to $S^{1,-2}(T^*\mathbb{R}^3,\Sigma)$, it follows from the symbol calculus that $\frac{\sigma\left(N^{33}_-\right)\sigma\left(N^{E^2}_+\right)}{\sigma\left(N^{33}_+\right)\sigma\left(N^{E^2}_-\right)}$ belongs to $S^{0,0}(T^*\mathbb{R}^3,\Sigma)$. 
Furthermore, since $\sigma\left(N^{33}_-\right)$ actually vanishes quartically on $\Sigma$, it follows that $\frac{\sigma\left(N^{33}_-\right)\sigma\left(N^{E^2}_+\right)}{\sigma\left(N^{33}_+\right)\sigma\left(N^{E^2}_-\right)}$ is small (say $\left|\frac{\sigma\left(N^{33}_-\right)\sigma\left(N^{E^2}_+\right)}{\sigma\left(N^{33}_+\right)\sigma\left(N^{E^2}_-\right)}\right|<\frac{1}{2}$) in a conic neighborhood of $\Sigma$. Away from $\Sigma$, we can estimate the fraction by replacing the terms in the fraction with their respective principal symbols, since in the denominator the principal symbols are elliptic away from $\Sigma$. From \eqref{cutoffest} we have that if $\chi\equiv 1$ in a neighborhood of the equatorial sphere $\{\xi_T = 0\}$ and is supported in $\{|\xi_T|<\epsilon|\xi|\}$, then we have
\begin{align*}\sigma_{-1}(N^{33}_-)\sigma_{-1}(N^{E^2}_+) &= \frac{O(\epsilon^2)}{(a_{11}-a_{55})_l}a_{-,T}a_{+,T}\\
\text{and}\quad\sigma_{-1}(N^{33}_+)\sigma_{-1}(N^{E^2}_-) &= \left(\frac{1}{(a_{11}-a_{55})_l}+O(\epsilon^2)\right)a_{+,T}a_{-,T}
\end{align*}
(using the notation of \eqref{aint}), and hence
\[\frac{\sigma_{-1}(N^{33}_-)\sigma_{-1}(N^{E^2}_+)}{\sigma_{-1}(N^{33}_+)\sigma_{-1}(N^{E^2}_-)} = O(\epsilon^2).\]
It follows that if $\epsilon$ is sufficiently small, then $1-\frac{\sigma\left(N^{33}_-\right)\sigma\left(N^{E^2}_+\right)}{\sigma\left(N^{33}_+\right)\sigma\left(N^{E^2}_-\right)}$ is an everywhere elliptic symbol belonging to $S^{0,0}(T^*\mathbb{R}^3,\Sigma)$. Then every component of the inverse matrix $q$ is of an element of $S^{1,-2}(T^*\mathbb{R}^3,\Sigma)$. It then follows (essentially by applying Lemma \ref{paropinv} to each component) that the quantization $q(x,D)$ belongs to $\Psi^{1,-2}(\mathbb{R}^3,\Sigma)\otimes\text{Mat}_{6\times 6}(\mathbb{C})$, with
\begin{align*}
R_{-1/2} &= q(x,D)\begin{pmatrix} N^{33}_+ & N^{E^2}_+ \\ N^{33}_- & N^{E^2}_- \end{pmatrix}-\text{Id}\in\Psi^{-1,-1}(\mathbb{R}^3,\Sigma)\otimes\text{Mat}_{6\times 6}(\mathbb{C})\\
\text{and}\quad R_{1/2} &= q(x,D)\begin{pmatrix} \tilde{N}^{33}_+ & \tilde{N}^{E^2}_+ \\ \tilde{N}^{33}_- & \tilde{N}^{E^2}_-\end{pmatrix}\in\Psi^{0,-1}(\mathbb{R}^3,\Sigma)\otimes\text{Mat}_{6\times 2}(\mathbb{C}).
\end{align*}
Thus, if
\[f = \begin{pmatrix} N^{33}_+ & N^{E^2}_+ \\ N^{33}_- & N^{E^2}_- \end{pmatrix}\begin{pmatrix} \nabla r_{33} \\ \nabla r_{E^2} \end{pmatrix} + \begin{pmatrix} \tilde{N}^{33}_+ & \tilde{N}^{E^2}_+ \\ \tilde{N}^{33}_- & \tilde{N}^{E^2}_-\end{pmatrix}\begin{pmatrix} r_{33} \\ r_{E^2}\end{pmatrix},\]
then applying $q(x,D)$ to both sides yields
\[q(x,D)f = \begin{pmatrix} \nabla r_{33} \\ \nabla r_{E^2}\end{pmatrix} + R_{-1/2}\begin{pmatrix} \nabla r_{33} \\ \nabla r_{E^2}\end{pmatrix} + R_{1/2}\begin{pmatrix} r_{33} \\ r_{E^2}\end{pmatrix},\]
and hence we obtain the stability estimate
\begin{equation}
\|(\nabla r_{33},\nabla r_{E^2})\|_{L^2}\le C\left(\|f\|_{H^2} + \|(r_{33},r_{E^2})\|_{H^{1/2}}\right).
\label{33e2}
\end{equation}
Next, let's suppose $a_{33}$ is known instead. Then we have
\[\vec{0}_6 = \begin{pmatrix} N^{11}_+ & N^{E^2}_+ \\ N^{11}_- & N^{E^2}_- \end{pmatrix}\begin{pmatrix} \nabla r_{11} \\ \nabla r_{E^2} \end{pmatrix} + \begin{pmatrix} \tilde{N}^{11}_+ & \tilde{N}^{E^2}_+ \\ \tilde{N}^{11}_- & \tilde{N}^{E^2}_-\end{pmatrix}\begin{pmatrix} r_{11} \\ r_{E^2}\end{pmatrix}.\]
The inverse of the symbol of the first matrix can be written as
\[q = \left(1 - \frac{\sigma\left(N^{11}_-\right)\sigma\left(N^{E^2}_+\right)}{\sigma\left(N^{11}_+\right)\sigma\left(N^{E^2}_-\right)}\right)^{-1}\begin{pmatrix} \frac{1}{\sigma\left(N^{11}_+\right)} & - \frac{\sigma\left(N^{E^2}_+\right)}{\sigma\left(N^{11}_+\right)\sigma\left(N^{E^2}_-\right)} \\ - \frac{\sigma\left(N^{11}_-\right)}{\sigma\left(N^{11}_+\right)\sigma\left(N^{E^2}_-\right)} & \frac{1}{\sigma\left(N^{E^2}_-\right)}\end{pmatrix}.\]
In this case, since $\sigma\left(N^{E^2}_+\right)$ and $\sigma\left(N^{11}_-\right)$ have principal parts vanishing on $\Sigma$, $\sigma\left(N^{E^2}_-\right)^{-1}$ is of inverse parabolic type, and $\sigma\left(N^{11}_+\right)$ is \emph{elliptic}, it follows (similarly to the above case) that $\frac{\sigma\left(N^{11}_-\right)\sigma\left(N^{E^2}_+\right)}{\sigma\left(N^{11}_+\right)\sigma\left(N^{E^2}_-\right)}$ belongs to $S^{0,0}(T^*\mathbb{R}^3,\Sigma)$ and is guaranteed to be small in a conical neighborhood of $\Sigma$. We can analyze the behavior away from $\Sigma$ by analyzing the principal symbols as before: in this case we have
\begin{align*}
\sigma_{-1}(N^{11}_-)\sigma_{-1}(N^{E^2}_+) &= \frac{O(\epsilon^2)}{(a_{11}-a_{55})_l}a_{-,I}a_{+,T}\\
\text{and}\quad \sigma_{-1}(N^{11}_+)\sigma_{-1}(N^{E^2}_-) &= \left(\frac{1}{(a_{11}-a_{55})_l}+O(\epsilon^2)\right)a_{+,I}a_{-,T},
\end{align*}
and hence
\[\frac{\sigma_{-1}(N^{11}_-)\sigma_{-1}(N^{E^2}_+)}{\sigma_{-1}(N^{11}_+)\sigma_{-1}(N^{E^2}_-)} = O(\epsilon^2)\frac{a_{-,I}a_{+,T}}{a_{+,I}a_{-,T}}.\]
The latter fraction is close to $1$ when $\epsilon$ is small. 
Thus like above we have that $1 - \frac{\sigma\left(N^{11}_-\right)\sigma\left(N^{E^2}_+\right)}{\sigma\left(N^{11}_+\right)\sigma\left(N^{E^2}_-\right)}$ is an everywhere elliptic symbol belonging to $S^{0,0}(T^*\mathbb{R}^3,\Sigma)$. In addition, the terms in the matrix can be analyzed as follows:
\begin{itemize}
\item $\frac{1}{\sigma\left(N^{11}_+\right)}$ is a order $1$ symbol of type $(1,0)$, since $N^{11}_+$ is elliptic.
\item Writing the top right and bottom left terms as $\left(-\frac{\sigma\left(N^{E^2}_+\right)}{\sigma\left(N^{E^2}_-\right)}\right)\frac{1}{\sigma\left(N^{11}_+\right)}$ and $\left(-\frac{\sigma\left(N^{11}_-\right)}{\sigma\left(N^{E^2}_-\right)}\right)\frac{1}{\sigma\left(N^{11}_+\right)}$, we see that these terms belong to $S^{1,0}(T^*\mathbb{R}^3,\Sigma)$.
\item Finally, $\frac{1}{\sigma\left(N^{E^2}_-\right)}\in S^{1,-2}(T^*\mathbb{R}^3,\Sigma)$, similar to the terms in the previous case.
\end{itemize}
Due to the above observations, we can similarly conclude in the previous case that
\[q(x,D)\begin{pmatrix} N^{11}_+ & N^{E^2}_+ \\ N^{11}_- & N^{E^2}_- \end{pmatrix}-\text{Id}\in\Psi^{-1,-1}(\mathbb{R}^3,\Sigma)\otimes\text{Mat}_{6\times 6}(\mathbb{C}).\]
Furthermore, a careful analysis of the entries of the product
\[\quad q(D)\begin{pmatrix} \tilde{N}^{11}_+ & \tilde{N}^{E^2}_+ \\ \tilde{N}^{11}_- & \tilde{N}^{E^2}_-\end{pmatrix}\]
shows that it is in $\Psi^{0,-1}(\mathbb{R}^3,\Sigma)$, similarly to the previous case (a slightly different argument is needed since $\sigma_{-1}(\tilde{N}^{11}_+)$ does not necessarily vanish on $\Sigma$). For example, the symbol of the bottom-left entry is $\left(1 - \frac{\sigma\left(N^{11}_-\right)\sigma\left(N^{E^2}_+\right)}{\sigma\left(N^{11}_+\right)\sigma\left(N^{E^2}_-\right)}\right)^{-1}$ times
\[- \frac{\sigma\left(N^{11}_-\right)}{\sigma\left(N^{11}_+\right)\sigma\left(N^{E^2}_-\right)}(\sigma(\tilde{N}^{11}_+)) + \frac{1}{\sigma\left(N^{E^2}_-\right)}(\sigma(\tilde{N}^{E^2}_+)).\]
The first term in fact belongs to $S^{0,0}(T^*\mathbb{R}^3,\Sigma)$ since $\frac{\sigma\left(N^{11}_-\right)}{\sigma\left(N^{11}_+\right)\sigma\left(N^{E^2}_-\right)}\in S^{1,0}(T^*\mathbb{R}^3,\Sigma)$ and $\sigma(\tilde{N}^{11}_+)\in S^{-1}(T^*\mathbb{R}^3)$, while the second term belongs to $S^{0,-1}(T^*\mathbb{R}^3,\Sigma)$. The analysis of the other entries follow similarly. Thus, since $q(x,D)\begin{pmatrix} \tilde{N}^{11}_+ & \tilde{N}^{E^2}_+ \\ \tilde{N}^{11}_- & \tilde{N}^{E^2}_-\end{pmatrix}\in\Psi^{0,-1}(\mathbb{R}^3,\Sigma)\otimes\text{Mat}_{6\times 2}(\mathbb{C})$, similar kinds of estimates follow as in the previous case.

We summarize the arguments above in the following proposition, which, when combined with an argument similar to that in the proof of Theorem \ref{oneparam}, suffices to prove Theorem \ref{twoparam}:
\begin{proposition}
\label{twocoeffprop}
For the problem of recovering $(a_{33},E^2)$ (resp. $(a_{11},E^2)$) given a known value for $a_{11}$ (resp. $a_{33}$), if $\chi$ is supported in $\{|\xi_T|<\epsilon|\xi|\}$ for $\epsilon$ sufficiently small and identically $1$ in a smaller neighborhood, then we have the stability estimates
\[
\|(\nabla r_{33},\nabla r_{E^2})\|_{L^2}\le C(\|f\|_{H^2} + \|(r_{33},r_{E^2})\|_{H^{1/2}}).
\]
and
\[
\|(\nabla r_{11},\nabla r_{E^2})\|_{L^2}\le C(\|f\|_{H^2} + \|(r_{11},r_{E^2})\|_{H^{1/2}}).
\]
where
\begin{align*} f = &\begin{pmatrix} N^{33}_+ & N^{E^2}_+ \\ N^{33}_- & N^{E^2}_- \end{pmatrix}\begin{pmatrix} \nabla r_{33} \\ \nabla r_{E^2} \end{pmatrix} + \begin{pmatrix} \tilde{N}^{33}_+ & \tilde{N}^{E^2}_+ \\ \tilde{N}^{33}_- & \tilde{N}^{E^2}_-\end{pmatrix}\begin{pmatrix} r_{33} \\ r_{E^2}\end{pmatrix}\\
(\text{resp. }&\begin{pmatrix} N^{11}_+ & N^{E^2}_+ \\ N^{11}_- & N^{E^2}_- \end{pmatrix}\begin{pmatrix} \nabla r_{11} \\ \nabla r_{E^2} \end{pmatrix} + \begin{pmatrix} \tilde{N}^{11}_+ & \tilde{N}^{E^2}_+ \\ \tilde{N}^{11}_- & \tilde{N}^{E^2}_-\end{pmatrix}\begin{pmatrix} r_{11} \\ r_{E^2}\end{pmatrix}).
\end{align*}
In particular, if the $r_{\nu}$ have sufficiently small width of support, then we can recover $(a_{33},E^2)$ and $(a_{11},E^2)$ from the combined $qP$ and $qSV$ travel time data, assuming in each case that the remaining parameter is known.
\end{proposition}
\begin{proof}[This thus proves Theorem \ref{twoparam}]\end{proof}

We now look at the case where there is a functional relationship between one of the material parameters and the other two. We recall the calculations in \eqref{a33func}, \eqref{e2func}, and \eqref{a11func}.

Suppose first that $a_{33} = f(a_{11},E^2)$. Recall that in this case the effective symbol $\begin{pmatrix} \sigma(N^{11}_{eff,+}) & \sigma(N^{E^2}_{eff,+}) \\ \sigma(N^{11}_{eff,-}) & \sigma(N^{E^2}_{eff,-})\end{pmatrix}$ has the upper-left entry being elliptic on $\Sigma$, while the other symbols have their principal parts vanishing quadratically on $\Sigma$, with the principal part of $\sigma(N^{E^2}_{eff,-})$ vanishing nondegenerately and the subprincipal part nonvanishing. Then similar arguments from above show that $\frac{\sigma(N^{E^2}_{eff,+})\sigma(N^{11}_{eff,-})}{\sigma(N^{11}_{eff,-})\sigma(N^{E^2}_{eff,-})}\in S^{0,0}(T^*\mathbb{R}^3,\Sigma)$ and in fact vanishes near $\Sigma$. Away from $\Sigma$, from \eqref{a33func} we have
\begin{align*}&\begin{pmatrix} \sigma_{-1}(N^{11}_{eff,+}) & \sigma_{-1}(N^{E^2}_{eff,+}) \\ \sigma_{-1}(N^{11}_{eff,-}) & \sigma_{-1}(N^{E^2}_{eff,-})\end{pmatrix} \\
&= \begin{pmatrix} (1 + O(\epsilon^2))(a_{+,I} + \tilde{f}_{11}a_{+,T}) & \left(\frac{-1}{(a_{11}-a_{55})_l}+\tilde{f}_{E^2}+O(\epsilon^2)\right)a_{+,T}\\ O(\epsilon^2)(a_{-,I}+a_{-,T})& \left(\frac{1}{(a_{11}-a_{55})_l}+O(\epsilon^2)\right)a_{-,T} \end{pmatrix}.
\end{align*}
In particular, if $\tilde{f}_{11}\ge 0$ and $\epsilon$ is sufficiently small, then $\sigma_{-1}(N^{11}_{eff,+})$ is elliptic everywhere, and if $\epsilon$ is sufficiently small then $\sigma_{-1}(N^{E^2}_{eff,-})$ is a positive multiple of $a_{-,T}$ and hence elliptic away from $\Sigma$. Furthermore, we can compute the determinant of the above matrix to be
\[\left(\frac{1}{(a_{11}-a_{55})_l}+O(\epsilon^2)\right)\left(a_{+,I} + \tilde{f}_{11}a_{+,T}\right)a_{-,T} - O(\epsilon^2)(a_{-,I}+a_{-,T})a_{+,T}.\]
Since $a_{-,T}$ and $a_{+,T}$ are of comparable sizes since they both vanish nondegenerately quadratically on $\Sigma$, it follows that the above expression is always nonzero away from $\Sigma$ if $\epsilon$ is sufficiently small. This then implies that $1 - \frac{\sigma(N^{E^2}_{eff,+})\sigma(N^{11}_{eff,-})}{\sigma(N^{11}_{eff,-})\sigma(N^{E^2}_{eff,-})}$ is everywhere elliptic, and thus the conclusions are exactly the same as if $a_{33}$ were known in Proposition \ref{twocoeffprop} by following the same line of reasoning. (Note in this case that the operators $\tilde{N}^{\nu}_{eff,\pm}$ have the same qualitative behavior as the operators $\tilde{N}^{\nu}_{\pm}$ in the non-functional case, namely that all of the operators vanish on $\Sigma$ except for $\tilde{N}^{11}_{eff,+}$.)

If instead $E^2 = f(a_{11},a_{33})$, then again we have that $\sigma(N^{11}_{eff,+})$ is elliptic, and that the other operators have principal parts vanishing quadratically on $\Sigma$, with the principal part of $\sigma(N^{33}_{eff,-})$ vanishing nondegenerately and the subprincipal part nonvanishing, as long as $\frac{\partial f}{\partial a_{33}}$ is uniformly nonzero. Thus as before we have $\frac{\sigma(N^{33}_{eff,+})\sigma(N^{11}_{eff,-})}{\sigma(N^{11}_{eff,-})\sigma(N^{33}_{eff,-})}\in S^{0,0}(T^*\mathbb{R}^3,\Sigma)$. Away from $\Sigma$, from \eqref{e2func} we have that the effective symbol $\begin{pmatrix} \sigma_{-1}(N^{11}_{eff,+}) & \sigma_{-1}(N^{33}_{eff,+}) \\ \sigma_{-1}(N^{11}_{eff,-}) & \sigma_{-1}(N^{33}_{eff,-})\end{pmatrix}$ is given by
\small
\[\begin{pmatrix} (1 + O(\epsilon^2))a_{+,I} - \left(\frac{1}{(a_{11}-a_{55})_l}\tilde{f}_{11}+O(\epsilon^2)\right)a_{+,T} & \left(1 - \frac{1}{(a_{11}-a_{55})_l}\tilde{f}_{33}+O(\epsilon^2)\right)a_{+,T}\\ O(\epsilon^2)a_{-,I}+\left(\frac{2}{(a_{11}-a_{55})_l}\tilde{f}_{11}+O(\epsilon^2)\right) a_{-,T}& \left(\frac{1}{(a_{11}-a_{55})_l}\tilde{f}_{33}+O(\epsilon^2)\right)a_{-,T} \end{pmatrix}.\]
\normalsize
The determinant of the above matrix is
\[\left(\frac{1}{(a_{11}-a_{55})_l}(\tilde{f}_{33}a_{+,I} - \tilde{f}_{11}a_{+,T})+O(\epsilon^2)\right)a_{-,T} + O(\epsilon^2)a_{+,T}.\]
Since on the support of $\chi$ we have $\xi_T^2\le\frac{\epsilon^2}{1-\epsilon^2}\xi_I^2$, it follows that $a_{+,T}\le\frac{\epsilon^2}{1-\epsilon^2}a_{+,I}$. Hence, as long as $\tilde{f}_{33}$ is uniformly bounded away from zero, by choosing $\epsilon$ sufficiently small we can guarantee ${\frac{1}{(a_{11}-a_{55})_l}(\tilde{f}_{33}a_{+,I} - \tilde{f}_{11}a_{+,T})+O(\epsilon^2)}>0$, and hence for $\epsilon$ sufficiently small (depending on the possible values of $\frac{\partial f}{\partial a_{33}}$ and $\frac{\partial f}{\partial a_{11}}$) the determinant is nonvanishing away from $\Sigma$. Thus, the same conclusions from the above paragraph \emph{mutatis mutandis} hold.

Finally, suppose $a_{11} = f(a_{33},E^2)$. We work with the simplifying assumption that $\frac{\partial f}{\partial a_{33}}$ and $\frac{\partial f}{\partial E^2}$ are constant, so that they equal $\tilde{f}_{33}$ and $\tilde{f}_{E^2}$, respectively. In the effective symbol $\begin{pmatrix} \sigma(N^{33}_{eff,+}) & \sigma(N^{E^2}_{eff,+}) \\ \sigma(N^{33}_{eff,-}) & \sigma(N^{E^2}_{eff,-})\end{pmatrix}$, we have that $\sigma(N^{33}_{eff,-})$ and $\sigma(N^{E^2}_{eff,-})$ both have principal parts vanishing quadratically on $\Sigma$. It follows that the principal part of the determinant $\sigma(N^{33}_{eff,+})\sigma(N^{E^2}_{eff,-}) - \sigma(N^{E^2}_{eff,+})\sigma(N^{33}_{eff,-})$ vanishes quadratically on $\Sigma$, since the principal part of the determinant is given by the above expression with the symbols replaced by their principal parts, and more importantly its subprincipal part on $\Sigma$ is given by
\[\sigma_{-1}(N^{33}_{eff,+})\sigma_{-2}(N^{E^2}_{eff,-}) - \sigma_{-1}(N^{E^2}_{eff,+})\sigma_{-2}(N^{33}_{eff,-})\]
since $\sigma_{-1}(N^{E^2}_{eff,-})$ and $\sigma_{-1}(N^{33}_{eff,-})$ vanish on $\Sigma$. Using \eqref{n33+1}, \eqref{ne2+1}, \eqref{n33-2}, and \eqref{ne2-2}, the constancy of the derivative allows us to rewrite the above expression as
\begin{align*}
&\tilde{f}_{33}\sigma_{-1}(N^{11}_+)(\sigma_{-2}(N^{E^2}_-) + \tilde{f}_{E^2}\sigma_{-2}(N^{11}_-)) - \tilde{f}_{E^2}\sigma_{-1}(N^{11}_+)(\tilde{f}_{33}\sigma_{-1}(N^{11}_-)) \\
&= \tilde{f}_{33}\sigma_{-1}(N^{11}_+)\sigma_{-2}(N^{E^2}_-)
\end{align*}
(without the constancy assumption the subprincipal parts have a more complicated expression, and in particular no guarantee of cancellation of the $\tilde{f}_{33}\tilde{f}_{E^2}$ terms). Thus, we see that as long as $\tilde{f}_{33}\ne 0$ (without any assumption on $\tilde{f}_{E^2}$) we have that the subprincipal part of the determinant does not vanish. Since the principal part of the effective symbol can be written via \eqref{a11func} as
\[\begin{pmatrix} (1 + O(\epsilon^2))(a_{+,T} + \tilde{f}_{33}a_{+,I}) & \left(\frac{-1}{(a_{11}-a_{55})_l}+O(\epsilon^2)\right)a_{+,T}+\left(\tilde{f}_{E^2}+O(\epsilon^2)\right)a_{+,I}\\\left(-\left(\frac{E^2}{a_{11}-a_{55}}\right)_l\tilde{f}_{33}+O(\epsilon^2)\right)a_{-,T}& \left(\frac{1}{(a_{11}-a_{55})_l}-\left(\frac{E^2}{a_{11}-a_{55}}\right)_l\tilde{f}_{E^2}+O(\epsilon^2)\right)a_{-,T}\end{pmatrix}\]
we see that its determinant is given by
\small
\begin{align*}&\left[\frac{\tilde{f}_{33}}{(a_{11}-a_{55})_l}a_{+,I} + \left(\frac{1}{(a_{11}-a_{55})_l}-\left(\frac{E^2}{a_{11}-a_{55}}\right)_l\tilde{f}_{E^2}\right.\right.\\
&\left.\left.-\left(\frac{E^2}{a_{11}-a_{55}}\right)_l\frac{1}{(a_{11}-a_{55})_l}\tilde{f}_{33}\right)a_{+,T} + O(\epsilon^2)\right]a_{-T}.
\end{align*}
\normalsize
Again using $a_{+,T}\le\frac{\epsilon^2}{1-\epsilon^2}a_{+,I}$, we see that as long as $\tilde{f}_{33}$ is bounded away from zero, for $\epsilon$ small enough (depending on the $\tilde{f}$'s) the prefactor is bounded away from zero; in particular the determinant has principal symbol which vanishes nondegenerately quadratically. Writing
\[\begin{pmatrix} \sigma(N^{33}_{eff,+}) & \sigma(N^{E^2}_{eff,+}) \\ \sigma(N^{33}_{eff,-}) & \sigma(N^{E^2}_{eff,-})\end{pmatrix}^{-1} = \frac{1}{d}\begin{pmatrix} \sigma(N^{E^2}_{eff,-}) & -\sigma(N^{E^2}_{eff,+}) \\ -\sigma(N^{33}_{eff,-}) & \sigma(N^{33}_{eff,+})\end{pmatrix}\]
with $d = \sigma(N^{33}_{eff,+})\sigma(N^{E^2}_{eff,-}) - \sigma(N^{E^2}_{eff,+})\sigma(N^{33}_{eff,-})$, we have that $\frac{1}{d}\in S^{2,-2}(T^*\mathbb{R}^3,\Sigma)$ by the comments above, and hence the left entries of the inverse matrix are symbols in $S^{1,0}(T^*\mathbb{R}^3,\Sigma)$ while the right entries are symbols in $S^{1,-2}(T^*\mathbb{R}^3,\Sigma)$. For the operator matrix $\begin{pmatrix}\tilde{N}^{33}_{eff,+} & \tilde{N}^{E^2}_{eff,+} \\ \tilde{N}^{33}_{eff,-} & \tilde{N}^{E^2}_{eff,-}\end{pmatrix}$, we have that the principal symbols in the bottom row vanish on $\Sigma$ (essentially because $\sigma_{-1}(\tilde{N}^{11}_-)$ vanishes on $\Sigma$), and hence the full symbols in the bottom row are in $S^{-1,1}(T^*\mathbb{R}^3,\Sigma)$; hence applying the quantization of $\begin{pmatrix} \sigma(N^{33}_{eff,+}) & \sigma(N^{E^2}_{eff,+}) \\ \sigma(N^{33}_{eff,-}) & \sigma(N^{E^2}_{eff,-})\end{pmatrix}^{-1}$ to this operator matrix results in a matrix-valued operator in $\Psi^{0,-1}(\mathbb{R}^3,\Sigma)$ as before. Thus the same conclusions hold as before.

We summarize the arguments above in the following proposition, which suffices to prove Theorem \ref{func}:
\begin{proposition}
\label{funcprop}
Suppose there is a known functional relationship $a_{33} = f(a_{11},E^2)$ with $\frac{\partial f}{\partial a_{11}}\ge 0$, or $E^2 = f(a_{11},a_{33})$ with $\left|\frac{\partial f}{\partial a_{33}}\right|>0$, or $a_{11} = f(a_{33},E^2)$ with the derivatives $\frac{\partial f}{\partial a_{33}}$ and $\frac{\partial f}{\partial E^2}$ constant and $\frac{\partial f}{\partial a_{33}}\ne 0$, and if the $r_{\nu}$ have sufficiently small width of support, then we can recover $(a_{11},E^2)$ (resp. $(a_{11},a_{33})$ and $(a_{33},E^2)$) from the combined $qP$ and $qSV$ travel time data.
\end{proposition}
\begin{proof}[This thus proves Theorem \ref{func}]\end{proof}

We conclude by commenting that the problem of recovering $a_{11}$ and $a_{33}$ from $E^2$ data cannot be solved using the techniques above, since the operator for the $qSV$ speed in the $a_{33}$ component vanishes quartically on $\Sigma$. Furthermore, if $E^2$ is identically zero, then the operators for the $qSV$ speed are identically zero.

\nocite{*}
\bibliographystyle{plain}
\bibliography{ti_global_072622}

\end{document}